\documentclass[leqno,12pt]{amsart}

\usepackage{ stmaryrd }
\usepackage{todonotes} 

\newcommand\franco[1]{\todo[inline,color=cyan!30]{#1}}

\newcommand{\danote}[1]{{\color{purple} \sf $\diamondsuit$ Daniele : #1 $\diamondsuit$ }}

\usepackage[left=2.3cm, right=2.3cm,top=2.6cm,bottom=2.6cm]{geometry}

\usepackage{times}
\usepackage{tensor}
\usepackage[bbgreekl]{mathbbol}
\usepackage[all]{xy}
\usepackage{tikz}
\usepackage{tikz-cd}
\usepackage{pgfplots}
\usepgfplotslibrary{fillbetween}
\newsavebox{\pullback}
\sbox\pullback{%
\begin{tikzpicture}%
\draw (0,0) -- (1ex,0ex);%
\draw (1ex,0ex) -- (1ex,1ex);%
\end{tikzpicture}}
\usepackage{subcaption}
\usepackage{comment}

\usetikzlibrary{arrows}

\usepackage{amsmath, amssymb, amsfonts, latexsym, mdwlist, amsthm}

\usepackage{mathrsfs}
\usepackage{graphicx}
\usepackage{wrapfig}
\usepackage{longtable}
\usepackage{enumerate}
\usepackage{mathtools} 

\usepackage{bbm}
\usepackage{tkz-euclide}
\tikzset{
  dotted/.style={pattern=dots,pattern color=#1},
  dotted/.default=black
}
\tikzset{
  fdotted/.style={pattern=crosshatch dots,pattern color=#1},
  fdotted/.default=black
}
\tikzset{
  scopedlines/.style={pattern=north east lines,pattern color=#1},
  scopedlines/.default=black
}
\tikzset{
  hrlines/.style={pattern=horizontal lines,pattern color=#1},
  hrlines/.default=black
}
\newcommand*{\DashedArrow}[1][]{\mathbin{\tikz [baseline=-0.25ex,-latex, dashed,#1] \draw [#1] (0pt,0.5ex) -- (1.3em,0.5ex);}}

\makeindex

\usepackage[colorlinks=true,citecolor=blue, urlcolor=blue, linkcolor=blue, pagebackref]{hyperref}
\usepackage[capitalize]{cleveref}

\def\P{\ensuremath{\mathbb{P}}}
\def\Pd{\ensuremath{\check{\mathbb{P}}}}
\def\Q{\ensuremath{\mathbb{Q}}}

\def\Z{\ensuremath{\mathbb{Z}}}

\def\CC{\ensuremath{\mathcal C}}
\def\cD{\ensuremath{\mathcal D}}

\def\cK{\ensuremath{\mathcal K}}

\def\ZZ{{\mathbb{Z}}}
\def\QQ{{\mathbb{Q}}}
\def\CC{{\mathbb{C}}}
\def\PP{{\mathbb{P}}}

\def\sI{\ensuremath{\mathscr I}}
\def\sO{\ensuremath{\mathscr O}}
\def\sC{\ensuremath{\mathscr C}}

\def\Sec{\mathop{\mathrm{Sec}}\nolimits}

\def\dim{\mathop{\mathrm{dim}}\nolimits}

\def\Hom{\mathop{\mathrm{Hom}}\nolimits}

\def\NS{\mathop{\mathrm{NS}}\nolimits}
\def\Pic{\mathop{\mathrm{Pic}}\nolimits}
\def\Kum{\mathop{\mathrm{Kum}}\nolimits}
\def\Bs{\mathop{\mathrm{Bs}}\nolimits}
\def\Fix{\mathop{\mathrm{Fix}}\nolimits}
\def\Sym{\mathop{\mathrm{Sym}}\nolimits}

\newtheorem{Thm}{Theorem}[section]
\newtheorem{Prop}[Thm]{Proposition}

\newtheorem{Lem}[Thm]{Lemma}
\newtheorem{Cor}[Thm]{Corollary}

\newtheorem*{Ques*}{Question}

\theoremstyle{definition}
\newtheorem{Def}[Thm]{Definition}
\newtheorem{Rem}[Thm]{Remark}
\newtheorem{Not}[Thm]{Notation}

\newtheorem{Ex}[Thm]{Example}

\newcommand{\thistheoremname}{}
\newtheorem*{genericthm*}{\thistheoremname}
\newenvironment{namedthm*}[1]
{\renewcommand{\thistheoremname}{#1}%
	\begin{genericthm*}}
	{\end{genericthm*}}
\newenvironment{namedtheorem*}[1]
{\renewcommand{\thistheoremname}{#1}%
	\begin{genericthm*}}
	{\end{genericthm*}}

\makeatletter
\newtheoremstyle{italicsname}
 {3pt}
 {3pt}
 {\itshape}
 {}
 {\itshape}
 {.}
 {.5em}
 {\thmname{#1}\thmnumber{\@ifnotempty{#1}{ }#2}%
 \thmnote{ {\the\thm@notefont(#3)}}}
\makeatother
\theoremstyle{italicsname}

\author[D. Agostini]{Daniele Agostini}
\address[Daniele Agostini]{ Eberhard Karls Universit\"at T\"ubingen, Fachbereich Mathematik, Auf der Morgenstelle 10, 72076 T\"ubingen, DE.}
\email{daniele.agostini@uni-tuebingen.de}
\author[P. Beri]{Pietro Beri}
\address[Pietro Beri]{Universit\'e de Lorraine, CNRS, IECL, F-54000 Nancy – France.}
\email{pietro.beri@univ-lorraine.fr}
\author[F. Giovenzana]{Franco Giovenzana}
\address[Franco Giovenzana]{Université Paris-Saclay, CNRS, Laboratoire de mathématiques d'Orsay,  B\^{a}t,  307, 91405 Orsay, France.}
\email{ franco.giovenzana@universite-paris-saclay.fr}
\author[A.D. R\'ios Ortiz]{\'Angel David R\'ios Ortiz}
\address[\'Angel David R\'ios Ortiz]{Université Paris-Saclay, CNRS, Laboratoire de mathématiques d'Orsay,  B\^{a}t,  307, 91405 Orsay, France.}
\email{angel-david.rios-ortiz@universite-paris-saclay.fr}

\title[Coble duality for  Jacobian Kummer fourfolds]{Coble duality for  Jacobian Kummer fourfolds}

\usepackage{pgfplots}

\pgfplotsset{
compat=newest,
every axis plot/.append style={no marks,thick},
every axis/.style={
  axis lines=middle,
  width=7cm,
  height=3cm,
  }
}

\begin{document}

\begin{abstract} \noindent
We study projective models of generalized Kummer fourfolds via O'Grady's theta groups and the classical Coble cubic. More precisely, we establish a duality between two singular models of the generalized Kummer fourfold of a Jacobian abelian surface. We also give projective models for singular Jacobian Kummer varieties of arbitrary dimension. Along the way, we also construct a first non-natural involution on the Hilbert square of a Jacobian surface. In the appendix, we study singularities of secants of arbitrary varieties at identifiable points, following Choi, Lacini, Park and Sheridan.
\end{abstract}

\maketitle


\section{Introduction}

The generalized Kummer manifolds associated to an abelian surface $A$ were one of the first examples of hyperk\"ahler varieties. They are defined from the Hilbert scheme $A^{[n+1]}$ of $n+1$ points on $A$ and the summation map
\[ \overline{s} \colon A^{[n+1]} \longrightarrow A; \quad \xi \mapsto \sum_{a\in \xi} \ell(O_{\xi,a})\cdot a \]
that associates to a scheme the sum inside $A$ of the points in its support, counted with multiplicity. The $n$-th generalized Kummer variety of $A$ is then the fiber of the summation map over zero
\[ \Kum_n(A) \coloneqq \overline{s}^{-1}(0). \]
The singular generalized Kummer variety inside the symmetric product $A^{(n+1)}$ is the image of $\Kum_n(A)$ under the Hilbert-Chow morphism $\operatorname{HC}\colon A^{[n+1]} \to A^{(n+1)}$:
\[ \Sigma_n(A) \coloneqq \operatorname{HC}(\Kum_n(A)).  \]
The generalized Kummer $\Kum_n(A)$ is a smooth $2n$-dimensional hyperk\"ahler variety: for example, $\Kum_1(A)$ is the usual K3 Kummer surface associated to $A$. Hyperk\"ahler varieties of Kummer type are those that are deformation equivalent to a $\Kum_n(A)$, and they form one of the  known examples of hyperk\"ahler varieties of arbitrary even dimension, the other one being those deformation equivalent to the Hilbert scheme of points on a K3 surface. The intrinsic geometric aspects of generalized Kummer varieties have been extensively studied: for example the Hodge Conjecture \cite{Hassett-Tschinkel-lagrangian-planes, FloccariVaresco2025} or their intermediate Jacobians \cite{OGrady2021,MarkmanMonodromyGeneralizedKummer}. However, the extrinsic geometry of generalized Kummer beyond the surface case is still relatively unexplored. For example, to the best of our knowledge, there is yet no description of a complete family for varieties of Kummer type of dimension larger than two\footnote{A forthcoming work by Bayer-Perry-Pertusi-Zhao provides a construction of locally complete families of polarized generalized Kummer varieties in arbitrary dimension, by describing all such families arising from moduli spaces of stable objects on noncommutative abelian surfaces \cite{Bayer2025}.}.
In recent years, the first systematic steps in studying projective models for generalized Kummer have taken place: see for example \cite{Varesco2023},\cite{BMT2021} and \cite{OGTheta}. In this paper, we focus on this last topic, when $A = \operatorname{Pic}^0(C)$ is the Jacobian of an arbitrary smooth projective curve $C$ of genus two. This is also the setting of the recent paper \cite{bswkummer} where the authors exhibit $\Kum_2(A)$ as a birational cover of $\mathbb{P}^4$. It would be interesting to explore possible connections between that description and ours.
\vspace{10pt}

We first describe birational projective models of $\Kum_n(A)$, or rather of $\Sigma_n(A)$, for any $n$. To do so, we fix one odd theta characteristic $\eta$ on the curve $C$, so that we have a corresponding symmetric theta divisor $\Theta\subseteq A$ and all its translates $\Theta_a = \Theta+a$, for $a\in A$. The line bundle $\Theta$ on $A$ induces via symmetrization a line bundle $\mu_n(\Theta)$ on $\Kum_n(A)$. The map induced by the linear system $H^0(\Kum_n(A),\mu_n(\Theta))$ can be described very concretely:

\begin{namedthm*}{Theorem A}
\emph{
    Let $A$ be the Jacobian of an arbitrary smooth genus two curve. There is a well-defined morphism
    \[ u\colon \Sigma_n(A) \longrightarrow |(n+1)\Theta| = \PP(H^0(A,(n+1)\Theta)); \quad [a_0,\dots,a_n] \mapsto \Theta_{a_0}+\dots+\Theta_{a_n}\]
    which is injective for all $n$ and an embedding for $n=1,2$. The composition 
    \[ u\circ \operatorname{HC}\colon \Kum_n(A) \longrightarrow |(n+1)\Theta|\]
    coincides with the map $\varphi_{\mu_n(\Theta)}$ induced by the complete linear system of $\mu_n(\Theta)$. In particular, it is birational onto the image and it is a contraction of  the divisor $E\subseteq \Kum_n(A)$ of non-reduced subschemes.
}
\end{namedthm*}

The proof of this result is based on the classical theorem of the square for abelian varieties. This is essentially the same argument used by \cite{Varesco2023} in his study of base-point-freeness on generalized Kummer. Instead, the case of $n=2$ has been already treated by \cite{BMT2021} from the point of view of representation theory. The statement of Theorem A, however, encodes more the geometry of the singular Kummer variety $\Sigma_n(A)$, than the one of $\Kum_n(A)$. The main subject of our paper is the detailed study of another birational contraction, in the case of the fourfold $\Kum_2(A)$. Our main tool in this is the beautiful geometry of the natural embedding
\[ \varphi_{3\Theta}\colon A \hookrightarrow \Pd^8 \coloneqq \PP(H^0(A,3\Theta)^{\vee}).\]
In particular, there is a unique cubic hypersurface $\mathscr{C}_3 = Z(F)$ in $\Pd^8$ which is singular precisely along $A$. This is known as the \emph{Coble cubic}, after Arthur Coble \cite{Coble1917} and it is well-studied \cite{Beauville2003,Ortega2005,nguyen07}. The partial derivatives of the equation of $\mathscr{C}_3$ cut out the surface $A$ and they define a map
\[ \mathcal{D}\colon \Pd^8 \dashrightarrow \P^8 \coloneqq |3\Theta|; \quad \mathcal{D} = \left[ \frac{\partial F}{\partial X_0},\dots,\frac{\partial F}{\partial X_8} \right]\]
This is related to the moduli space $\mathcal{SU}_C(3)$ of semi-stable vector bundles on $C$ of rank $3$ and trivial determinant. There is a double cover
\[ \Phi_{3}\colon \mathcal{SU}_C(3) \longrightarrow \P^8; \quad E\mapsto \{a\in A\,|\, h^0(C,E\otimes a) \ne 0\}\]
branched over a sextic hypersurface $\mathscr{C}_6 \subseteq \P^8$, called the \emph{Coble sextic}. Dolgachev conjectured that these two hypersurfaces are dual to each other: $\mathscr{C}_6 = \mathcal{D}(\mathscr{C}_3)$. This was proven by Ortega \cite{Ortega2005} and later also by Nguyen \cite{nguyen07}.
\vspace{10pt}

The main result of our paper is a manifestation of an analogous ``Coble duality'' for two birational models of the Kummer fourfold $\Kum_2(A)$. To state our result, recall that $\Kum_2(A)$ has a line bundle $\mu_2(\Theta)$ as well as the divisor $E\subseteq \Kum_2(A)$ of nonreduced subschemes. This latter divisor is divisible by two, meaning that there is a divisor class $\delta$ such that $2\delta\sim E$. There is also another natural divisor $F\subseteq \Kum_2(A)$ given by
\[F = \{\xi\in \Kum_2(A) \,|\, \xi\subseteq \Theta_a \text{ for one } a\in A \}. \]
Finally, one can also show that for any $a\in A$ the translate $\Theta_a\subseteq A$ spans a 4-dimensional space $\ell(\Theta_a) \subseteq \Pd^8$. With this in mind, we can state:

\begin{namedthm*}{Theorem B}
\emph{
    Let $A$ be the Jacobian of an arbitrary smooth genus two curve. The two line bundles $\mu_2(\Theta)$ and $\mu_2(2\Theta)-\delta$ on $\Kum_2(A)$ are globally generated, and there is a commutative diagram:}
\[
\begin{tikzcd}
		& \Kum_2(A)\ar[ld, "\varphi_{\mu_2(2\Theta)-\delta}"']\ar[rd, "\varphi_{\mu_2(\Theta)}"]\\
		\Pd^8 \ar[rr, dashed, "\mathcal{D}"]&& \P^8
	\end{tikzcd}
\]
\emph{
The two maps are defined on a reduced scheme $\{a,b,c\} \in \Kum_2(A)$ as
\[ \varphi_{\mu_2(2\Theta)-\delta}(\{a,b,c\}) = \ell(\Theta_a) \cap \ell(\Theta_b) \cap  \ell(\Theta_c), \qquad \varphi_{\mu_2(\Theta)}(\{a,b,c\}) = \Theta_a + \Theta_{b} + \Theta_c, \]
in particular $\ell(\Theta_a) \cap \ell(\Theta_b) \cap  \ell(\Theta_c)$ is a single point.
The map $\varphi_{\mu_2(\Theta)}$ is the contraction of the divisor $E$ described in Theorem A. The map $\varphi_{\mu_2(2\Theta)-\delta}$ contracts the divisor $F$ onto $A\subseteq \Pd^8$ and is an embedding everywhere else. The image of $\varphi_{\mu_2(2\Theta)-\delta}$ is a fourfold of degree $36$ and it is the singular locus of the secant variety $\operatorname{Sec}(A) \subseteq \Pd^8$:
\[ \varphi_{\mu_2(2\Theta)-\delta}(\Kum_2(A)) = \operatorname{Sing}(\operatorname{Sec}(A)). \]
}
\end{namedthm*}

Note that if $A$ has Picard rank one  then the two line bundles $\mu_2(\Theta),\mu_2(2\Theta)-\delta$ are the generators of the nef and movable cone of $\Kum_2(A)$ \cite[Theorem 0.1]{Mori2021}, so, in a way our theorem describes the most natural projective models of $\Kum_2(A)$. Furthermore, an analogous result is also known classically for the Kummer K3 surface $\Kum_1(A)$: in this case the role of the Coble cubic is played by the singular quartic Kummer surface associated to $A$, see \Cref{thm:kummerduality} for a precise statement.
\vspace{10pt}

One ingredient in our proof of Theorem B is the geometry of the embedded Jacobian variety $A$ in $\Pd^8$ and of the theta divisors $\Theta_a$. Another one is the study of a natural cover of $\Kum_2(A)$ by surfaces birational to $\Kum_1(A)$. These surfaces are related to the classical Weddle surfaces of $A$ via the map $\varphi_{\mu_2(2\Theta)-\delta}$. The third essential ingredient is not directly related to generalized Kummers: we exhibit an involution on the Hilbert square $A^{[2]}$, which is the first example of a non-natural automorphism on the Hilbert square of a Jacobian \footnote{It was claimed in \cite[Theorem 2]{girardet} that such an automorphism does not exist for a Jacobian of Picard rank one, but the argument given there has a gap. We thank Patrick Girardet for discussions on this topic.}: see also \cite{sasaki}, \cite{girardet}. This construction produces actually $16$ involutions, which are related to the classical ``switches'' on the Kummer K3 surface of $A$, see  \Cref{rem:16involutions}.

\begin{namedthm*}{Theorem C}
\emph{
Let $A$ be the Jacobian of an arbitrary smooth genus two curve. There exists an involution $\tau\colon A^{[2]} \to A^{[2]}$, defined on any reduced scheme $\{a,b\}$ as
\[ \tau(\{a,b\}) = \Theta_a\cap \Theta_b. \]
Furthermore, this involution is not natural, meaning that it is not induced by an automorphism of $A$.
}
\end{namedthm*}

Finally, in order to identify the image of the map $\varphi_{\mu_2(2\Theta)-\delta}$ of Theorem B with the singular locus of the secant, we need some results on the singularities of secant varieties. In the Appendix, we recall some recent results on these varieties contained in \cite{CPLS25}, but we spell out explicitly some details that we need for our purposes. In particular, we characterize singularities of the secant variety at identifiable points.
\vspace{10pt}

The paper is structured as follows: in \Cref{sec:preliminaries} we recall various preliminary results and background material. In \Cref{sec:hilbertsquares} we discuss Hilbert squares of Jacobian surfaces, and we prove Theorem C. We also recall in \Cref{thm:kummerduality} the classical Kummer duality which is analogous to our Theorem B in the case of Kummer K3 surfaces. In \Cref{sec:contractionmutheta} we prove Theorem A, while in \Cref{sec: covering} we describe a covering of Kummer fourfolds in terms of Kummer K3 surfaces that we will need for Theorem B. The set up and the proof of Theorem B takes up \Cref{sec:contractionD} and \Cref{sec:secants}. In \cref{sec:weddle} we show how to relate our constructions to the classical Weddle surfaces. We then mention some possible further directions in \Cref{sec:openquestions} and we conclude with \cref{sec: appsec} on singularities of secant varieties.

\vspace{10pt}

\textbf{Acknowledgements:} We would like to thank Vladimiro Benedetti, Samuel Boissi\`ere, Doyoung Choi, Olivier Debarre, Salvatore Floccari, Patrick Girardet, Andreas H\"oring, Emanuele Macrì, Kieran O' Grady, Jinhyung Park, Bert van Geemen  for useful discussions. Special thanks to all the participants of the workshop ``Kummers in Krakow'' that was held on the Jagiellonian University in Krakow in May 2024, and to the organizers Grzegorz Kapustka and Calla Tschanz.

D.A. was supported by the Deutsche Forschungsgemeinschaft (DFG) under Project 530132094 ``Positivity on K-trivial varieties''. P.B was supported by the project ANR-23-CE40-0026 ``Positivity on K-trivial varieties''. P.B., A-D.R-O. and F.G. were supported by the European Research Council (ERC) under the European Union’s Horizon 2020 research and innovation programme (ERC-2020-SyG-854361-HyperK).
F.G. was funded by the Deutsche Forschungsgemeinschaft (DFG) Projektnummer 509501007 and is a member of the INdAM GNSAGA.

\section{Preliminaries}\label{sec:preliminaries}

We start with some preliminary results and background material. As a general source on abelian varieties, we refer to \cite{BL2004}. We will work with equivalence classes of Cartier divisors or line bundles on varieties, and we will use interchangeably the additive notation for divisors or the tensor notation for line bundles. Furthermore, if $D,E$ are two divisors (resp. line bundles) on a smooth variety, we will write $D\sim E$ to denote that they are linearly equivalent (resp. isomorphic).

\subsection{Jacobian surfaces and their theta divisors}
\label{sec: intro jac and Theta}

Let $C$ be a smooth  projective complex curve of genus $2$, with hyperelliptic involution $\iota\colon C \to C$ and hyperelliptic cover
\[ \phi_{K_C}\colon C \to \PP^1_{K_C} := \PP(H^0(C,K_C)^{\vee}) \] 

\begin{Not}\label{rem:notationv}
	We will  use the notation $v = [x,\iota(x)]$, if $v\in \PP^1_{K_C}$ and $\phi_{K_C}^{-1}(v)=\{x,\iota(x)\}$. In particular $v=[x,x]$ if $x=\iota(x)$ is a Weierstrass point of $C$. 
\end{Not}

Let now $A:=\Pic^0(C)$ be the Jacobian of $C$, so that $A$ is an abelian surface, and we also set $J:=\Pic^1(C)$. Both $A$ and $J$ come equipped with involutions
\[ \iota\colon A\longrightarrow A; \quad a\mapsto -a \qquad  \iota' \colon J \longrightarrow J; \quad L \mapsto K_C-L.\] 
There is  a canonical Abel-Jacobi map
\[ j\colon C \hookrightarrow J; \quad x\mapsto \sO_C(x) \]
and the image of $C$ is the canonical theta divisor $j(C) = W^1 = \{ L \in J \,|\, h^0(C,L)>0 \}$: by Riemann-Roch, this divisor is preserved by the involution on $J$.

We now fix, once and for all, an odd theta characteristic $\eta\in \Pic^1(C)$, meaning that $\eta  \sim x_0$, 
where $x_0$ is a Weierstrass point on $C$. This yields an isomorphism between $J$ and $A$ as well as an Abel-Jacobi map
\[
J\longrightarrow  A; \quad L\mapsto L-\eta, \qquad  \alpha\colon C \hookrightarrow A; \quad x \mapsto x-\eta.
\]
Since $\eta$ is a theta characteristic, the isomorphism between $A$ and $J$ is compatible with the two involutions, and then the theta divisor $\Theta := \alpha(C)$ is preserved by the involution $\iota$ on $A$: more precisely the involution $\iota$ restricted to $\Theta$ corresponds to the hyperelliptic involution on $C$. For any $a\in A$ we define the corresponding translation $t_a\colon A\to A$, as well as the the translated Abel-Jacobi map and the translated theta divisor:
\[ \alpha_a\colon C\hookrightarrow  A; \quad x\mapsto x+a-\eta \qquad \Theta_a := t_a(\Theta) = \alpha_a(C) \]
in particular $\Theta_0 = \Theta$.  One has that $(\Theta_a\cdot \Theta_b)=2$ for all $a,b\in A$ and furthermore, $\Theta_a \sim \Theta_b$ if and only if $a=b$. 

\begin{Not}\label{rem:notationalphaD}
Since the Abel-Jacobi maps $\alpha_a\colon C\hookrightarrow A$ are embeddings for any $a\in A$, they induce embeddings $\alpha\colon C^{(n+1)} \hookrightarrow A^{[n+1]}$, where we look at the symmetric product $C^{(n)}$ as the Hilbert scheme of $n$ points on $C$. In particular if $D=x_1+\dots+x_n$ is an effective divisor on $C$, we will denote by $\alpha_a(D)$ the corresponding subscheme of $A$. The schemes of these form are precisely those contained in the theta divisor $\Theta_a$.
\end{Not}

\begin{Not}
We will try to denote with $a,b,c,\dots$ points on $A$ and with $x,y,z,\dots$ points on the curve $C$. Furthermore, we will also try to denote a finite scheme of length $3$ by $\xi$ and a scheme of length two by $\zeta$. If $X$ is a smooth varietly and $p\in X$ is a point, the nonreduced schemes of length two supported at $a$ are parametrized by tangent directions $v\in \PP(T_p X)$. We will denote these schemes as $\zeta=(p,v)$. Finally, if $\pi\colon A^{\times (n+1)} \to A^{(n+1)} = A^{\times (n+1)
}/\mathfrak{S}_{n+1}$ is the quotient map, we denote $[a_0,\dots,a_n]\coloneqq \pi(a_0,\dots,a_n)$.
\end{Not}

\subsection{Linear systems on Jacobians and the Coble cubic}

The linear system $|2\Theta|$ is base point free, and it induces a map
\[
\varphi_{2\Theta}:A\longrightarrow \Pd^3=\P(H^0(A,2\Theta)^\vee) 
\]
The image $\mathscr{K}$  is the singular Kummer surface associated to $A$: it is a quartic hypersurface, and it is isomorphic via $\varphi_{2\Theta}$ to the  quotient $A/\iota$. 
Instead, the linear system $|3\Theta|$ on $A$ is very ample, and it induces an embedding 
\[
\varphi_{3\Theta}:A\hookrightarrow \Pd^8=\P(H^0(A,3\Theta)^\vee). 
\]
We will sometimes identify $A$ with its image in $\Pd^8$. Coble proved that there is an unique cubic hypersurface $\sC_3 = Z(F) $ in $\Pd^8$ which is singular along $A$. This is called the \emph{Coble cubic} associated to $A$. The nine partial derivatives of the equation of $\sC_3$ cut out $A$ scheme-theoretically (even if they do not generate the homogeneous ideal of $A$), and they define a map
\[ \mathcal{D} \colon \Pd^8 \dashrightarrow \P^8 := |3\Theta| \]

The Coble cubic is related to the moduli space $\mathcal{SU}_{C}(3)$ of semi-stable vector bundles on $C$ of rank three and trivial determinant. We recall briefly this beautiful story, following the notation of  \cite{nguyen07} and \cite{BMT2021}. The moduli space has a map
\begin{equation}\label{eq:mapfromSU}
\Phi_3\colon \mathcal{SU}_C(3) \longrightarrow |3\Theta|, \qquad E\mapsto \{a\in A \,|\, h^0(C,E\otimes \sO(\eta+a)) > 0\}    
\end{equation}
which is a double cover, branched over a sextic hypersurface $\mathscr{C}_6 \subseteq |3\Theta|$, called the \emph{Coble sextic}.
The ramification locus is 
\[ \mathcal{R} = \{  E\in \mathcal{SU}_C(3) \,|\, \iota^* E^{\vee} \cong E \}\] 
Dolgachev conjectured that $\sC_6$ is projectively dual to $\sC_3$:
\[ \overline{\mathcal{D}(\sC_3)}=\sC_6. \]
and this was proven  first by Ortega \cite{Ortega2005} and later by \cite{nguyen07} with a different argument. 

\medskip

We also recall that if $L$ is any ample line bundle on the abelian surface $A$, then the group $K(L) = \{x\in A \,|\, t_x^*L \sim L\}$ is finite and acts naturally on the projective spaces $\PP(H^0(A,L)),\PP(H^0(A,L)^{\vee})$. This can be lifted to an action on the vector spaces $H^0(A,L),H^0(A,L)^{\vee}$ via the corresponding \emph{theta group}:
\[ \mathcal{G}_A(L) := \{ (\varphi,a) \,|\, a\in K(L), \varphi\colon L \overset{\sim}{\to} t_a^*L \}.  \]
If $L$ is of type $(d_1,d_2)$, then this group is isomorphic to the Heisenberg group $\mathcal{H}(d_1,d_2)$, and the representation $H^0(A,L)$ is isomorphic to the irreducible Schr\"odinger representation.

\begin{Rem}
	In particular, this applies to the ample line bundles $2\Theta$ and $3\Theta$, of type $(2,2)$ and $(3,3)$ respectively. Furthermore, one has $K(2\Theta) = A[2]$ and $K(3\Theta) = A[3]$, so that the translations by elements of $A[2]$ (resp. $A[3]$) are induced by projective transformations of the singular Kummer surface (resp. of $A\subseteq \Pd^8$).
\end{Rem}

\subsection{Jacobian generalized Kummer varieties}\label{sec:genkums} Consider the Hilbert scheme $A^{[n+1]}$ of $n+1$ points on the abelian surface $A$ and the Hilbert-Chow morphism $\mathrm{HC}\colon  A^{[n+1]} \to A^{(n+1)}$ to the symmetric product.
The group law on $A$ induces the summation map $s\colon A^{(n+1)} \to A$. We call the summation map also the  composition
\[ \overline{s} = s\circ \mathrm{HC} \colon A^{[n+1]} \longrightarrow A. \]
The generalized Kummer $2n$-fold of $A$ is the fiber over zero
\[ \operatorname{Kum}_n(A) := \overline{s}^{-1}(0). \]
while the singular Kummer $2n$-fold is
\[ \Sigma_n(A) := \mathrm{HC}(\Kum_n(A)) = s^{-1}(0).\]
There is  a commutative diagram
\begin{equation}\label{eq:sumationmapKummer}
	\begin{tikzcd}
		\overline{s}^{-1}(0)=\Kum_n(A)\ar[r,hook]\ar[dd,swap,"\mathrm{HC}"] & A^{[n+1]}\ar[rd,"\overline{s}"]\ar[dd,"\mathrm{HC}"]\\
		& & A\\
		s^{-1}(0) = \Sigma_n(A)\ar[r,hook] & A^{(n+1)}\ar[ru,swap,"s"]
	\end{tikzcd}   
\end{equation}
For any $a\in A$, we also denote by
\begin{equation}\label{eq:def Kn,a}
	K_{n,a} = \overline{s}^{-1}(-a) = \{ \xi \in A^{[n+1]} \,|\, \overline{s}(\xi)+a = 0  \}
\end{equation}
the fiber over $-a$ of the summation map. All these varieties are isomorphic to $K_{n,0} = \Kum_n(A)$, and the fibration given by the summation map becomes trivial after a base change: more precisely, there is a cartesian diagram
\begin{equation*}
	\begin{tikzcd}
		\Kum_n(A)\times A \ar[r]\ar[d,"\operatorname{pr}_A"] & A^{[n+1]} \ar[d,"\overline{s}"]\\
		A\ar[r,swap,"\cdot( n+1)"] & A
	\end{tikzcd}    
\end{equation*}
where the bottom map is the multiplication by $n+1$ and the top map is $(\xi,a) \mapsto t_{a}(\xi)$.

There are natural line bundles (or divisors) on $\Kum_n(A)$ that are inherited from the Hilbert scheme $A^{[n+1]}$: first, for any line bundle $L$ on $A$, there is a symmetrization $L^{(n+1)}$ on $A^{(n+1)}$ and its pullback $\mathrm{HC}^* L^{(n+1)} = L^{(n+1)}$ (we keep the same notation) on $A^{[n+1]}$. By restricting to $\Kum_n(A)$ we obtain a group homomorphism
\[ \mu_n\colon \operatorname{Pic}(A) \to \operatorname{Pic}(\Kum_n(A)); \quad L \mapsto  \mu_n(L) :=  L^{(n+1)}_{|\Kum_n(A)}\]
Furthermore, there is also the exceptional divisor $E$ of the Hilbert-Chow morphism, corresponding to the locus of nonreduced schemes. There is a class $\delta$ on $A^{[n+1]}$ such that $2\delta \sim E$ on $A^{[n+1]}$, and we keep denoting by $\delta$ its restriction to $\Kum_n(A)$. With these notations,  if $n\geq 2$ we have an embedding
\[   \mu_n(\NS(A)) \oplus \ZZ\delta \hookrightarrow \NS(A^{[n+1]}) \]
and an isomorphism (\cite[Notation 5.13]{KapferMenet2018}) 
\[ H^2(\Kum_n(A),\ZZ)  \cong \mu_n (H^2(A,\ZZ) ) \oplus \ZZ \delta.  \]
Generalized Kummer varieties are one of the main examples of hyperk\"ahler varieties. The Beauville-Bogomolov-Fujiki form is 
\[ q(\mu_n(L),\mu_n(M)) = (L\cdot M), \quad q(\delta,\delta) = -2(n+1), \quad q(\mu(L),\delta) = 0   \]
for any two line bundles $L,M \in \Pic(A)$. The Fujiki relation is then
\[ \int_{\Kum_n(A)} \alpha^{2n} = (2n-1)!!\cdot (n+1)\cdot  q(\alpha,\alpha)^n  \qquad \text{ for all } \alpha \in H^2(X,\ZZ)\]

\subsection{The theta group}\label{sec:extthetagroup}
O'Grady in \cite{OGTheta} extended the definition of theta groups to hyperk\"ahler varieties of generalized Kummer type. Here we recall some of his results, restricting to the case of $\Kum_n(A)$. The subgroup $A[n+1]\subseteq A$ of $(n+1)$-torsion points acts naturally on the Hilbert scheme $A^{[n+1]}$, and this action restricts to one on $\Kum_n(A)$, which is trivial in $H^2(\Kum_n(A),\ZZ)$, so that any line bundle $\mathcal{L}$ on $\Kum_n(A)$ is invariant with respect to $A[n+1]$. Then the theta group of $\mathcal{L}$ is
\[ \mathcal{G}_{\Kum_n(A)}(\mathcal{L}) := \{ (\varphi,a) \,|\, a\in A[n+1], \varphi\colon \mathcal{L} \overset{\sim}{\to} t_a^*\mathcal{L} \text{ isomorphism } \} \]
This acts naturally on the space of global sections $H^0(\Kum_n(A),\mathcal{L})$. The group and its action is determined by the following:
\begin{Thm}\cite[Theorem 3.2]{OGTheta}\label{thm:ogradytheta}
Let $\mathcal{L}$ be a big and nef primitive line bundle on $\Kum_n(A)$ of the form $\mathcal{L}= \mu_n(L)+b\cdot\delta$, with $b\in \ZZ$ and $L \in \operatorname{Pic}(A)$ of type $(d_1,d_2)$, such that both $d_1,d_2$ are coprime with $n+1$. 
\begin{enumerate}
	\item The theta group $\mathcal{G}_{\Kum_n}(A)$ is isomorphic to the Heisenberg group $\mathcal{H}(n+1,n+1)$.	
	\item The space  $H^0(\Kum_n(A),\mathcal{L})$ decomposes into copies of the Schr\"odinger representation. In particular, it is the Schr\"odinger representation if and only if $q(\mathcal{L},\mathcal{L})=2$.
\end{enumerate}
\end{Thm}

\subsection{The Brian\c{c}on surface}\label{subsec:briancon}

For any $a\in A$, we consider the subvariety 
\[ B_a := \{ \xi \in A^{[3]} \,|\, \operatorname{Supp}(\xi)=a \}.  \] 
This is also called the punctual Hilbert scheme of length three and we call it the Brian\c{c}on surface, after \cite{briancon}. We recall some facts about its structure,  following \cite{briancon} and \cite{ES98}. First we take analytic local coordinates $x,y$ centered at $a$, so that $\mathfrak{m}_a = (x,y)$ is the ideal of the point $a$. Then we can write the ideal $\mathscr{I}_{\xi}$ of $\xi \in B_{a}$ in one of the following forms:
\[ \mathscr{I}_{\xi} = (y+\alpha\cdot x + \beta \cdot x^2,x^3), \quad \text{or} \quad \mathscr{I}_{\xi} = (x+\gamma\cdot y + \delta \cdot y^2,y^3), \quad \text{or} \quad \mathscr{I}_{\xi} = (x,y)^2 = \mathfrak{m}_a^2. \]
In the first two cases the scheme $\xi$ is curvilinear, meaning that it is supported on a smooth curve, or, equivalently, that $\dim T_a {\xi} = 1$. In the last case, the scheme is not curvilinear and  we denote it by $O_a = \operatorname{Spec} \mathcal{O}_{A,a}/\mathfrak{m}^2_a$. To understand the global structure of $B_a$, we first fix one length two subscheme $\zeta$ supported only on $a$ and then we look at all $\xi \in B_a$ such that $\zeta\subseteq \xi$. There is a short exact sequence of sheaves on $A$:
\begin{equation}\label{eq:exseqnestedideals}
	0 \longrightarrow \mathscr{I}_{\xi} \longrightarrow \mathscr{I}_{\zeta} \longrightarrow \kappa(a) \longrightarrow 0
\end{equation}
where $\kappa(a)=\mathscr{O}_{A,a}/\mathfrak{m}_a$ is the residue field of the point. Conversely, if $\mathscr{I}_{\zeta} \to \kappa(a)$ is a surjective (i.e. nonzero) homomorphism , then the kernel is the ideal sheaf of a length three subscheme containing $\eta$. In summary
\[ \{ \xi \in B_a \,|\, \xi \supseteq \zeta\} \cong \PP(\Hom(\mathscr{I}_{\eta},\kappa(a))) \cong \PP((\mathscr{I}_{\eta}\otimes \kappa(a))^{\vee}). \]
The schemes  $\zeta$ as above are parametrized by $\PP(T_a A)$. More precisely, one has $\mathscr{I}_{\zeta} = (\alpha x+\beta y) + \mathfrak{m}_a^2$, where $[\alpha,\beta]\in \PP^1 = \PP(T_a A)$, and $T_a \zeta = \{\alpha x+\beta y=0\}$. Now we want to consider the sheaf $\mathscr{H}$ on $\PP(T_a A) = \PP^1$ whose fiber at $\zeta$ is $\mathscr{H}_{|\zeta} = \Hom(\mathscr{I}_{\zeta},\kappa(a))$: if $\alpha\ne 0$, then $\mathscr{I}_{\zeta} = (x+\frac{\beta}{\alpha} y,y^2)$ and the relations between the generators lie in $\mathfrak{m}_a$. This means that
a morphism  $\phi\colon \mathscr{I}_{\zeta} \to \kappa(a)$ is uniquely determined by the images of the two generators, which can be chosen arbitrarily: 
\[ \phi\left( x +\frac{\beta}{\alpha}y \right) =\gamma, \quad \phi\left( y^2 \right) = \delta \quad \gamma,\delta\in \kappa(a).\]
We see that in this case
\[ \operatorname{Ker} \phi = \left( x + \frac{\beta}{\alpha}y - \frac{\gamma}{\delta}y^2,y^3 \right)  \text{ if } \delta\ne 0 \qquad \operatorname{Ker} \phi = \mathfrak{m}_a^2 \quad   \text{ if } \delta= 0,\gamma\ne 0. \]
The same applies if $\beta\ne 0$, but in this case $\mathscr{I}_{\zeta} = (\frac{\alpha}{\beta}x+ y,x^2)$. This shows that $\mathscr{H}$ is a free sheaf of rank two on the two open charts $\{\alpha\ne 0\},\{\beta\ne 0\}$. Now we observe that if $\phi\colon \mathscr{I}_{\eta} \to \kappa(a)$ is a morphism of sheaves, then
\[  \phi(\alpha x \cdot(\alpha x+\beta y)) = 0 = \phi(\beta y \cdot (\alpha x+ \beta y))\]
so that, if $\alpha\ne 0,\beta \ne 0$, then
\[ \phi\left(x+\frac{\beta}{\alpha}y\right) = \frac{\beta}{\alpha}\cdot \phi\left(\frac{\alpha}{\beta} x + y\right), \quad \phi(y^2) = \left( \frac{\beta}{\alpha} \right)^{-2} \phi( x^2)   \]
In summary, $\mathscr{H} = \mathscr{O}_{\PP^1}(1)\oplus \mathscr{O}_{\PP^1}(-2)$ and if we consider the fibration $T_a = \PP(\mathscr{O}_{\PP^1}(1)\oplus \mathscr{O}_{\PP^1}(-2)) \to \PP^1$, and we take a point $[\alpha,\beta]$ representing the scheme $\mathscr{I}_{\zeta} = (\alpha x + \beta y) + \mathfrak{m}_a^2$, then the fiber over $[\alpha,\beta]\in \PP^1$ is naturally identified with the set $\{\xi \in B_a \,|\, \xi \supseteq \zeta \}$. The fibration $T_a \to \PP^1$ has a section that associates to each $\zeta$ the non-curvilinear scheme $O_{a}$, and the natural map $T_a \to B_a$ contracts this section to the point $O_a$ and is an isomorphism outside it. 

\section{Hilbert squares of Jacobian surfaces}\label{sec:hilbertsquares}

We start with analyzing Hilbert squares of Jacobian surfaces. 


\begin{Rem}
	Just for the moment, let $A$ be any abelian variety, $a\in A$ a point and $d(t_{-a})$ be the differential of the translation by $-a$. We have canonical identifications
	\[ T_a A \xrightarrow{d(t_{-a})} T_0 A, \qquad \PP(T_a A) \xrightarrow{\PP(d(t_{-a}))} \PP(T_0 A) \] 
	 In particular, any nonreduced subscheme $\zeta \subseteq A$ of length two is uniquely identified by a point $a = \operatorname{Supp}(\zeta)$ in $A$ and by a tangent direction $v \in \PP(T_a A)=\PP(T_0 A)$. We will denote such a scheme as $\zeta = (a,v)$. We can globalize this description to  the Hilbert square $A^{[2]}$ of $A$: it is a standard fact that  $A^{[2]}$ can be obtained as the quotient of the blow-up  $\operatorname{Bl}_{\Delta_A}A^2$ by the involution that lifts the standard one $s\colon (a,b)\mapsto (b,a)$ on $A^2$. Now consider the isomorphism
	 \[ \Phi\colon A^2 \longrightarrow A^2, \qquad (a,b) \mapsto (a,b-a)\]
	 so that $\Phi(\Delta_A) = A\times \{0\}$ and  $s' = \Phi\circ s \circ \Phi^{-1} \colon (a,c) \mapsto (a+c,-c)$. Observe that $\operatorname{Bl}_{A\times \{0\}}(A^2) \cong A \times \operatorname{Bl}_0(A)$, so that, if we keep denoting by $s' \colon A\times \operatorname{Bl}_0 A \to A\times \operatorname{Bl}_0 A$ the lift of $s'$ to the blow-up, we have the following:
\end{Rem}

\begin{Prop}\label{prop:hilb2abvar}
	Let $A$ be an abelian variety. With the above notation, there is an isomorphism
	\[ A^{[2]} \overset{\sim}{\longrightarrow} (A\times \operatorname{Bl}_0 A)/s'; \quad 
	\begin{matrix} 
	\{a,b\}\mapsto [(a,b-a)] &  a,b \in A \text{ distinct}, \\
	(a,v) \mapsto [(a,v)] &  a\in A, v\in \PP(T_a A)
	\end{matrix} 
	\]
\end{Prop} 

This model of the Hilbert square is valid for an arbitrary abelian variety. If $A$, as in our case, is the Jacobian of a genus two curve, we can get more. Indeed, the tangent space $T_0 A$ is $T_0 A = H^1(C,\mathcal{O}_C)$, and Serre's duality provides an isomorphism $H^1(C,\mathcal{O}_C) \cong H^0(C,K_C)^{\vee}$. We then have identifications $\PP(T_0 A)\cong \PP(H^0(C,K_C)^{\vee}) = \PP^1_{K_C}$, and, furthermore,   
the projectivized differential of any Abel-Jacobi map $\alpha_a\colon C\hookrightarrow A$ corresponds to the hyperelliptic cover \cite[Proposition 11.1.4]{BL2004}: 
\begin{equation}\label{eq: diff alpha}
    \mathbb{P}(d\alpha_{a})\colon C \to \PP^1_{K_C} = \phi_{K_C}\colon C\to \PP^1_{K_C} 
\end{equation} 
Hence, using the convention of \Cref{rem:notationv}, we can denote any $v\in \PP(T_aA) \cong \PP^1_{K_C}$ as $v=[x,\iota(x)]$ for a point $x\in C$, and we can use the same notation for  nonreduced schemes $\zeta = (a,v) = (a,[x,\iota(x)])$ in $A^{[2]}$. In order to use this for a global description of $A^{[2]}$ we first need to recall some facts about the difference map:
\[   C\times C \longrightarrow A; \quad (x,y)\mapsto x-y \]
\begin{Lem}\label{lem:diffmap}
	\begin{enumerate}
		\item If $x-y\sim z-w$ then, 
		\[ (x,y)=(z,w), \quad \text{  or } \quad x=y, z=w, \quad \text{  or } \quad (z,w)=(\iota(y),\iota(x)).\]
		\item The difference map is invariant with respect to the involution
		\[ \sigma \colon C^2 \longrightarrow C^2; \quad (x,y) \mapsto (\iota(y),\iota(x)),\]
		\item The difference map  induces an isomorphism $C^2/\sigma \xrightarrow{\sim} \operatorname{Bl}_0 A$, such that the induced map on the diagonal: $\Delta_C \to \PP(T_0 A) = \PP^1_{K_C}$ is the hyperelliptic cover.
	\end{enumerate}
\end{Lem}
\begin{proof}
	\begin{enumerate} 
		\item The assumption  $x-y\sim z-w$ means that $x+w \sim z+y$ and since the canonical divisor is the only one of degree two which is a pencil, this means that either $x+w=z+y$ or that $x+w \sim K_C \sim z+y$. In the first case, if $x=z$ then $y=w$ as well, while if $x=y$ then $w=z$ as well.  In the second case, we must have $w=\iota(x),z=\iota(y)$.
		\item This follows from (1).
		\item The fixed points of $\sigma$ are given by the smooth curve $\{(x,\iota(x)) \,|\, x\in C\}$, so that the quotient $C^2/\sigma$ is smooth. Now consider the curve $\Delta_{C}/\sigma \subseteq C^2/\sigma$: points (1) and (2) show that the induced map $C^2/\sigma \to A$ is birational with exceptional locus given by $\Delta/\sigma$, that maps to $0\in A$. Hence, by the universal property of the blow-up, this induces another birational map $C^2/\sigma \to \operatorname{Bl}_0 A$ and since the resulting map is finite, it must be flat and hence an isomorphism. In particular, the map   $C\cong \Delta_C \to \Delta_C/\sigma \cong \PP(T_0 A) \cong \PP^1_{K_C}$ is a double cover and it must be the hyperelliptic cover.
	\end{enumerate}
\end{proof}

\begin{Prop}\label{prop:hilb2jac}
	The following involutions of $A\times C^2$:
	\[ \sigma(a,x,y) = (a,\iota(y),\iota(x)), \quad \sigma'(a,x,y) = (a+x-y,y,x), \quad \sigma''(a,x,y) = (a+x-y,\iota(x),\iota(y))\]
	 form a subgroup $V_4 = \{\operatorname{id},\sigma,\sigma',\sigma''\} \cong \ZZ/2\ZZ \times \ZZ/2\ZZ$  and there is an isomorphism
	\[ A^{[2]} \overset{\sim}{\longrightarrow} (A\times C^2)/V_4; \quad 
	\begin{matrix}
		\{a,b\} \mapsto [(a,x,y)]  &  a,b\in A  \text{ distinct}, a-b\sim x-y \\
		\{a,v\} \mapsto [(a,x,\iota(x))] &  a\in A ,v=[x,\iota(x)] \in \PP(T_a A) 
	\end{matrix} 
	\]
\end{Prop}
\begin{proof}
	It is straightforward to check that $\sigma\circ \sigma' = \sigma'' = \sigma' \circ \sigma$, so that $V_4$ is indeed a group (isomorphic to a Klein four-group). Point (3) of \Cref{lem:diffmap}, shows that there is an isomorphism $A\times \operatorname{Bl}_0 A \cong (A \times C^2)/\sigma$ such that $(a,c)$ with $a,c\in A,c\ne 0$ corresponds to the class $[(a,x,y)]$ with $x-y\sim c$, while $(a,v)$ with $a\in A,v\in \PP(T_0A)$ corresponds to the class $[(a,x,\iota(x))]$ such that $v=[x,\iota(x)]$. Via this identification,  we observe that both $\sigma',\sigma''$ descend to the  involution $s'$ of $A\times \operatorname{Bl}_0 A$ given in \Cref{prop:hilb2abvar}, and the same proposition shows that $(A\times \operatorname{Bl}_0 A)/s' \cong A^{[2]}$.
\end{proof}

\subsection{Intersections of theta divisors, and an involution on $A^{[2]}$}

Consider two distinct theta divisors $\Theta_a,\Theta_b$. These intersect in a scheme of length two, so  it is natural to define a rational map on the Hilbert square of $A$ 
\[ \tau\colon A^{[2]} \dashrightarrow A^{[2]}; \quad \{a,b\} \mapsto \Theta_a\cap \Theta_b. \]
We will now show that this can be upgraded to an actual involution on $A^{[2]}$. We start by listing some elementary results on the intersection of theta divisors.

\begin{Lem}\label{lem:thetafacts}
	Let $a,b  \in A$.
	\begin{enumerate}
        \item $\iota(\Theta_a)$=$\Theta_{-a}$.
		\item $a\in \Theta_a$.
        \item $a\in \Theta_b$ if and only if $b\in \Theta_a$.

		\item  $b \in \Theta_a$ if and only if $ b \sim x+a-\eta$ for a certain $x\in C$, if and only if  $h^0(C,b-a+\eta) > 0$.
                \item $b\sim\alpha_a(x)$ if and only if $a\sim\alpha_b(\iota(x))$, for $x\in C$.  
		\item If $ b \sim x+a-\eta \in \Theta_a$, for $x\in C$, then the tangent direction to $\Theta_a$ at $b$ is $\PP(T_b \Theta_a) = [x,\iota(x)]$
		\item If $a,b$ are distinct, there are two distinct points $x,y\in C$ such that $x-y\sim b-a$. Then $\Theta_a,\Theta_b$ intersect at the points 
		\[ x-\eta+a \sim y-\eta+b, \qquad \iota(y)-\eta+a \sim \iota(x)-\eta+b\]
		These points coincide if and only if $y=\iota(x)$ and in this case the intersection is the nonreduced scheme $(x-\eta+a,[x,\iota(x)])$.
		
	\end{enumerate}
\end{Lem}
\begin{proof}
	\begin{enumerate} 
        \item We see $\iota(\Theta_a) = \iota(\Theta+a) = \iota(\Theta)-a=\Theta-a = \Theta_{-a}$.
		\item Recall that $\eta \sim x_0$, where $x_0$ is a Weierstrass point. Then we see that $0=x_0-\eta \in \Theta$ and $\Theta_a=\Theta+a$, so that $a = 0+a\in \Theta_a$.
        \item $a\in \Theta+b$ if and only if $b\in -\Theta+a = \Theta-a$, where the second equality is point (1).
		\item This follows from the definition $\Theta_a = \alpha_a(C)$.
        \item We see that $b\sim x+a-\eta$ if and only if $a\sim -x+\eta+b \sim \iota(x)+b-\eta$. The last equivalence $-x+\eta \sim \iota(x)-\eta$, follows from $x+\iota(x)\sim K_C\sim \eta+\eta$. 
		\item The tangent direction $\PP(T_b\Theta_a)$ is the image of the point $x\in C$ with respect to the projectivized differential $\PP(d\alpha_{a})$. Then the conclusion follows from \Cref{eq: diff alpha}. 
		\item Since $\Theta_a,\Theta_b$ are distinct, they intersect in two points, counted with multiplicity. Let $x-x_0+a \sim y-x_0+b$ be one of these intersection points, for $x,y\in C$. In particular $x-y\sim b-a$, and \Cref{lem:diffmap} shows  that $\iota(y)-\iota(x)\sim a-b$, meaning that $\iota(y)-x_0+a \sim \iota(x)-x_0+b$ is the other intersection point. The two intersection points coincide if and only if $x=\iota(x)$ and then the last statement follows from point (4).
	\end{enumerate}
\end{proof}

Inside $A^{[2]}$ we have two divisors
\begin{equation}\label{eq:EFA2} 
E = \{(a,v) \,|\, a\in A,v\in \PP(T_a A)\}, \qquad  F = \{ \alpha_a(x+\iota(x)) \,|\, a\in A, x\in C \}. 
\end{equation}
The first is the locus of non-reduced subschemes, and it exists on any Hilbert scheme of points on a smooth surface. The second is specific to our situation. We now prove Theorem C from the Introduction, in a more precise form.

\begin{Thm}[Theorem C]\label{thm:invtau}
	The morphism $\tau\colon A^{[2]} \longrightarrow A^{[2]}$ defined by
	\begin{align*}
	\tau(\{a,b\}) &:= \Theta_a\cap \Theta_b,  &  &a,b\in A, a\ne b, \\
	\tau((a,[x,\iota(x)])) &:= \{ x+a-\eta,\iota(x)+a-\eta \}, & &a\in A, x\in C, x\ne \iota(x),\\
	\tau((a,[x,x])) &:= (x+a-\eta, [x,x]), & & a\in A,x\in C,x=\iota(x).
	\end{align*}
	is well-defined and an involution. Moreover, $\tau(E)=F$.
\end{Thm} 
\begin{proof}
	Recall from \Cref{prop:hilb2jac}, that $A^{[2]}\cong A\times C^2/V_4$. Consider the automorphism
	\[ \widetilde{\tau}\colon A\times C^2 \longrightarrow A\times C^2; \quad (a,x,y) \mapsto (x+a-\eta,\iota(y),x). \]
	Straightforward computations show that $\widetilde{\tau}\circ \sigma = \sigma'\circ \widetilde{\tau},\widetilde{\tau}\circ \sigma' = \sigma \circ \widetilde{\tau}, \widetilde{\tau}\circ \sigma'' = \sigma''\circ \widetilde{\tau}$, hence $\widetilde{\tau}$ descends to an automorphism of  the quotient $A^{[2]} = (A\times C^2)/V_4$: let us denote it by $\tau\colon A^{[2]} \to A^{[2]}$. Furthermore, another straightforward computation (which uses that $\eta$ is a theta characteristic) shows that $\widetilde{\tau}^2 = \sigma''$, so that $\tau$ is an involution on $A^{[2]}$. Finally, we can also check explicitly that $\tau$ behaves as we want: consider first $\tau(\{a,b\})$ with $a,b \in A$ distinct: $\{a,b\}$ corresponds to the class of $(a,x,y)\in A\times C^2$, so that $x-y\sim b-a$. Then $\widetilde{\tau}(a,x,y) = (x+a-\eta,\iota(y),x)$. If $x\ne \iota(y)$,  this corresponds to the element $\{x+a-\eta,\iota(y)+a-\eta\} \in A^{[2]}$, which is exactly $\Theta_a\cap \Theta_b$ because of \Cref{lem:thetafacts}. if $x=\iota(y)$, then we have instead the element $\{x+a-\eta,[x,\iota(x)]\}$, which corresponds again to $\Theta_a\cap \Theta_b$, thanks again to \Cref{lem:thetafacts}.  Consider now $\tau(\{a,[x,\iota(x)]\})$: this corresponds to the class of $\widetilde{\tau}(a,x,x) = (x+a-\eta,\iota(x),x) \in A\times C^2$, which corresponds to $\{x+a-\eta,\iota(x)+a-\eta\}$ or $\{x+a-\eta,[x,x]\}$ in $A^{[2]}$, according to whether $x\ne \iota(x)$ or $x=\iota(x)$. 

    This proves that $\tau$ is a well defined involution on $A^{[2]}$. It is straightforward from the definition that $\tau(E)=F$.
\end{proof}

\begin{Rem}
    The fact that $\tau$ does not preserve $E$ shows that $\tau$ cannot be induced by an automorphism of $A$. In other words, using the terminology of \cite{BoissiereSarti2012}, it is not natural. This provides a counterexample to \cite[Theorem 2.(i)]{girardet}. We thank Patrick Girardet for helpful correspondence regarding \cite{girardet}.
\end{Rem}

The involution $\tau$ has a useful symmetry:

\begin{Lem}\label{lem:symmtau}
	Let $c\in A$ and $\xi \in A^{[2]}$. Then
	\[ c \in \tau(\xi) \iff \xi \subseteq \Theta_c \]
\end{Lem}
\begin{proof}
	Assume that $\xi = \{a,b\}$ with $a\ne b$. Then $c\in \tau(a,b)$ if and only if $c\in \Theta_a\cap \Theta_b$ and by \cref{lem:thetafacts} this is equivalent to $\{a,b\} \subseteq \Theta_c$.  If instead $\xi = (a,v)$ for $v = [x,\iota(x)]\in \PP(T_a A)$, then $c\in \tau(a,v)$ if and only if $c \in \{x+a-\eta,\iota(x)+a-\eta\}$ which is equivalent to $a \in \{x+c-\eta,\iota(x)+c-\eta\}$. By Lemma \ref{lem:thetafacts}, this means precisely that $a\in \Theta_c$ and that $\PP(T_a\Theta_c) = v$, i.e. $(a,v)\subseteq \Theta_c$. 
\end{proof}

\begin{Rem}

We can study the fixed locus of $\tau$ explicitly: if $\{a,b\}\in A^{[2]}$ is reduced, and then $\tau(\{a,b\})=\Theta_a\cap \Theta_b = \{a,b\}$ then $b\in \Theta_a$. Conversely, if $b\in \Theta_a$, then $\Theta_a\cap \Theta_b \supseteq \{a,b\}$ by \Cref{lem:thetafacts}, so that $\tau(a,b)=\{a,b\}$. If instead $v = [x,\iota(x)]\in \PP(T_aA)$, then \Cref{thm:invtau} shows that $\tau(a,[x,\iota(x)]) = (a,[x,\iota(x)])$ if and only if $x \sim \eta$, .i.e. $x=x_0$. In summary:
\begin{equation}\label{eq:fixtau} 
\operatorname{Fix}(\tau) = \{\{a,b\}\,|\, a\ne b, b\in \Theta_a\} \cup \{(a,[x_0,x_0])\,|\, a\in A\}  
\end{equation}
Note that the condition $b\in \Theta_a$ is symmetric in $a,b$ because of \Cref{lem:thetafacts}. We see that $\operatorname{Fix}(\tau)$ is an irreducible divisor: let $Z$ be the closure of the first subset on the right hand side  of \Cref{eq:fixtau}. It is straightforward to see that $Z$ is irreducible of dimension three, and $Z\subseteq\operatorname{Fix}(\tau)$. We also see that $Z$ must intersect any fiber of the Hilbert-Chow morphism, and since any fiber intersects $\operatorname{Fix}(\tau)$ in exactly one point, $Z$ must contain this point.
\end{Rem}

\subsection{Kummer duality for Kummer surfaces}

Now consider the summation map $\overline{s}\colon A^{[2]} \to A$, whose fibers are all isomorphic to the Kummer K3 surface 
\[ \operatorname{Kum}_1(A) = \{\zeta\in A^{[2]} \,|\, \overline{s}(\zeta) = 0 \}\] associated to $A$. It turns out that the summation map is invariant with respect to $\tau$, which then induces an involution of $\Kum_1(A)$. To identify this involution, we recall some classical notions on Kummer surfaces, following \cite{Keum,KondoKummer}:  the singular Kummer surface of $A$ is $\Sigma_1(A) \cong A/\iota $ and  $\Kum_1(A) \to \Sigma_1(A)$ is a blow up at the $16$ singular points, one for any two-torsion point $\epsilon\in A[2]$. In particular there are $16$ rational curves on $\Kum_1(A)$ given by the  exceptional divisors $E_\epsilon,\epsilon\in A[2]$: 
\[  E_{\epsilon} = \{(\epsilon,v)\in \Kum_1(A) \,|\, v\in \PP(T_\epsilon A) \} \qquad \text{ for all } \epsilon \in A[2].\]
There are also other $16$ rational curves (classically called tropes) $F_\epsilon,\epsilon\in A[2]$, that correspond to the proper transforms in $\Kum_1(A)$ of the quotients $\Theta_\epsilon/\iota \subseteq \Sigma_1(A)$: more precisely
\[F_{\epsilon} = \{ \zeta \in \Kum_1(A) \,|\, \zeta\subseteq \Theta_{\epsilon} \} \qquad \text{ for all } \epsilon \in A[2] \]
If $E,F\subseteq A^{[2]}$ are the divisors of \Cref{eq:EFA2}, it is also straightforward to see that 
\[ \Kum_1(A)\cap E = \sum_{\epsilon\in A[2]} E_{\epsilon}, \qquad \Kum_1(A)\cap F =\sum_{\epsilon \in A[2]} F_{\epsilon} \]
The line bundles $\Theta^{(2)}$ and $\delta$ on $A^{[2]}$ restrict to the bundles $\mu_1(\Theta)$ and $\delta = \frac{1}{2}\sum_{\epsilon \in A[2]} E_{\epsilon}$ on $\Kum_1(A)$. The relation between all these classes is given at \cite[page 5]{KondoKummer} as
\begin{equation}\label{eq:reldivkum1} 
4\mu_1(\Theta) \sim \frac{1}{2}\sum_{\epsilon} F_{\epsilon} + 3\cdot \left( \frac{1}{2} \sum_{\epsilon} E_{\epsilon}\right)  \end{equation}
Klein showed  that there are involutions of $\Kum_1(A)$, called \emph{switches}, characterized from the fact that they exchange $E_\epsilon$ with $T_\epsilon$, see for example \cite[Section 5]{KondoKummer}. Our involution is one of these:

\begin{Prop}\label{prop: restrizione a switch}
	The addition map $\overline{s}\colon A^{[2]} \to A$ is $\tau$-invariant, and the induced involution \[ \tau_{|\Kum_1(A)} \colon \Kum_1(A) \to \Kum_1(A)\]	is a switch that exchanges $E_\epsilon$ with $F_\epsilon$, for all $\epsilon\in A[2]$. Furthermore, it holds that
    \[ \tau^*\mu_1(\Theta) \sim \mu_1(3\Theta)-2\delta\]
\end{Prop}
\begin{proof}
	Let $a,b\in A$ general: we want to show that $\overline{s}(\tau(a,b)) = a+b$. Let $x,y\in C$ such that $x-y\sim b-a$. Since $a,b$ are general we can assume $y\ne \iota(x)$ and then \Cref{thm:invtau} and \Cref{lem:thetafacts} show that 
	\[ \tau(\{a,b\})=\Theta_a \cap \Theta_b = 
	\{ x-\eta+a,\iota(y)-\eta+a\}, 
	\]
	\[\overline{s}(\Theta_a\cap \Theta_b) \sim x-\eta+a+\iota(y)-\eta+a \sim x+\iota(y)-K_C+2a \sim  x-y+2a \sim b-a+2a \sim a+b. \]
	This proves the first part of the statement. For the second one, we use the definition of $\tau$ and we see that if $\epsilon\in A[2]$ and $x\in C$ we have
	\[\tau (\epsilon,[x,\iota(x)] ) = 
	\begin{cases} 
	\{x+\epsilon-\eta,\iota(x)+\epsilon-\eta\} & \text{ if } x\ne \iota(x) \\
	\{x+\epsilon-\eta,[x,x]\} & \text{ if } x=\iota(x)
	\end{cases} 
	\]
	This means that $\tau(E_\epsilon) = F_{\epsilon}$. The statement about $\tau^*\mu_1(\Theta)$ follows from \Cref{eq:reldivkum1}.
\end{proof}

\begin{Rem}\label{rem:16involutions}
 The involution $\tau$ of \Cref{thm:invtau} depends on our initial choice of the theta characteristic $\eta$. If we choose another theta characteristic, we get other involutions of $A^{[2]}$ and $\Kum_1(A)$. In particular, this way we get all the sixteen classical switches on $\Kum_1(A)$.
\end{Rem}

Now consider the map $\varphi_{2\Theta}\colon A \to \Pd^3$: as it is well-known, this map factors through the quotient $A/\iota$ and it induces an embedding $A/\iota \hookrightarrow \Pd^3$. The image is the Kummer quartic surface $\mathscr{K}_4$, which has $16$ nodes. Taking the first derivatives of a quartic equation for $\mathscr{K}_4$ we obtain a map
\[\mathcal{P}\colon \Pd^3 \dashrightarrow \P^3  = |2\Theta| \]
It turns out that this map induces a duality between the two models of $\Kum_1(A)$ given by the line bundles $\mu_1(\Theta)$ and $\tau^*\mu_1(\Theta) \sim \mu_1(3\Theta)-2\delta$. This is a classical result, but since the various statements are scattered throughout the literature, we state it and prove it also here, for completeness:

\begin{Thm}\label{thm:kummerduality}
The two line bundles $\mu_1(\Theta)$ and $\mu_1(3\Theta)-2\delta$ are globally generated, and there is a commutative diagram:
\[
\begin{tikzcd}
		& \Kum_1(A)\ar[ld, "\varphi_{\mu_1(3\Theta)-2\delta}"']\ar[rd, "\varphi_{\mu_1(\Theta)}"]\\
		\Pd^3 \ar[rr, dashed, "\mathcal{P}"]&& \P^3
	\end{tikzcd}
\]
The two maps are defined on a reduced scheme $\{a,-a\} \in \Kum_1(A),a\notin A[2]$ as
\[ \varphi_{\mu_1(3\Theta)-2\delta}(\{a,-a\}) = \varphi_{2\Theta}(\tau(a,-a)), \qquad \varphi_{\mu_1(\Theta)}(\{a,-a\}) = \Theta_a + \Theta_{-a}. \]
\begin{enumerate}
    \item The map $\varphi_{\mu_1(\Theta)}$ factors through $\Sigma_1(A)$ and it induces an embedding $\Sigma_1(A)\hookrightarrow \P^3$, whose image is a quartic with $16$ nodes. The map $\varphi_{\mu_1(\Theta)}$ contracts the divisors $E_{\epsilon}$ onto the nodes and is an isomorphism everywhere else.
    \item The image of $\varphi_{\mu_1(3\Theta)-2\delta}$ is the Kummer quartic surface $\mathscr{K}_4 \subseteq \Pd^3$. The map $\varphi_{\mu_1(3\Theta)-2\delta}$ contracts the divisors $F_{\epsilon}$ onto the nodes of $\mathscr{K}_4$ and is an isomorphism everywhere else.
    \item The image of $\varphi_{\mu_1(\Theta)}$ coincides with the dual hypersurface $\mathscr{K}_4^{\vee}$. The map $\mathcal{P}\colon \mathscr{K}_4 \dashrightarrow \mathscr{K}_4^{\vee}$ is birational and the two quartic surfaces $\mathscr{K}_4,\mathscr{K}_4^{\vee}$ are projectively equivalent.
\end{enumerate}
\end{Thm}

\begin{proof}

We first prove points (1) and (2).

\begin{enumerate}
\item Consider the singular Kummer surface $\Sigma_1(A) = \{[a,-a]\in A^{(2)} \,|\, a\in A \}$.
The theorem of the square for abelian varieties says that for any $a\in A$ we have $\Theta_a + \Theta_{-a} \sim \Theta_{a-a}+\Theta \sim 2\Theta$. Thus, there is a map $u \colon \Sigma_1(A) \to |2\Theta| = \P^3, [a,-a]\mapsto \Theta_a+\Theta_{-a}$. It is classical that this map is an embedding of $\Sigma_1(A)$ and that the image is a quartic with $16$ nodes. If we consider the composition $u\circ \operatorname{HC}\colon \Kum_1(A) \to \P^3$, we need to prove that $u\circ \operatorname{HC} = \varphi_{\mu_1(\Theta)}$. Take the hyperplane  $H \subseteq \P^3$ given by $H = \{D\in |2\Theta| \,|\, 0\in D\}$. Then $u^* H = \{[a,-a] \,|\, 0\in \Theta_a+\Theta_{-a}\} = \{[a,-a]\,|\, a\in \Theta\}$ where in the last equality we have used that $a\in \Theta$ if and only if $-a\in \Theta$. We can rewrite this as $(u\circ \operatorname{HC})^*H = \{\zeta\in \Kum_1(A) \,|\, \operatorname{Supp}(\zeta)\cap \Theta \ne \emptyset\}$, and this is linearly equivalent to $\mu_1(\Theta)$. This proves that $(u\circ \operatorname{HC})^*\sO_{\P^3}(1)\sim \mu_1(\Theta)$, and in particular this line bundle is globally generated and big. The Riemann-Roch theorem then yields $h^0(\Kum_1(A),\mu_1(\Theta)) = 4$, so that $u\circ \operatorname{HC}$ must be induced by the complete linear system of $\mu_1(\Theta)$.

\item First observe that there is a map $\psi\colon \Kum_1(A) \to \Pd^3$ given by $\psi(\zeta)=\varphi_{2\Theta}(\zeta)$. This is well-defined since the elements of $\Kum_1(A)$ are precisely the fibers of $\varphi_{2\Theta}\colon A \to \Pd^3$. The image of $\psi$ is the same as the image of $\varphi_{2\Theta}$, i.e. the Kummer quartic $\mathscr{K}_4$. We can prove as before that $\psi$ is induced by the complete linear system of $\mu_1(\Theta)$: choose a hyperplane $H'\subseteq \Pd^3$ such that $\varphi_{2\Theta}^*H' = \Theta_a + \Theta_{-a}$ for one $a\notin A[2]$. Then by construction $\psi^*H' = \{\zeta\in \Kum_1(A) \,|\, \operatorname{Supp}(\zeta)\cap \Theta_a \ne \emptyset \}$, and this is linearly equivalent to $\mu_1(\Theta)$. At this point, we see that $\tau^*\mu_1(\Theta) \sim \mu_1(3\Theta)-2\delta$, so that $\varphi_{\mu_1(3\Theta)-2\delta} = \psi\circ \tau$, which is what we want. The statement on the $F_{\epsilon}$ follows since $\tau$ exchanges the $E_{\epsilon}$ and the $F_{\epsilon}$.
\end{enumerate}
Now we can show that the diagram commutes: we need to show for a general point $a\in A$ that $\mathcal{P}(\varphi_{2\Theta}(\tau(a,-a))) = \Theta_a+\Theta_{-a}$. Since the map $\mathcal{P}$ restricted to $\mathscr{K}_4$ coincides with the Gauss map of $\mathscr{K}_4$, this means that the tangent plane to $\mathscr{K}_4$ at $\varphi_{2\Theta}(\Theta_a\cap \Theta_{-a})$ is the one spanned by $\Theta_a+\Theta_{-a}$. Rephrasing everything in terms of divisors, this means that the divisor $\Theta_a+\Theta_{-a}$ is singular at the points of $\Theta_a\cap\Theta_{-a}$, which is true.
\begin{itemize}
\item[(3)] Since the diagram commutes, we know that the image $\varphi_{\mu_1(\Theta)}(\Kum_1(A))$ coincides with the image of $\mathscr{K}_4$ through its Gauss map. This is the dual surface $\mathscr{K}_4^{\vee}$. Since the two maps $\varphi_{\mu_1(\Theta)}$ and $\varphi_{\mu_1(3\Theta)-2\delta}$ are birational, the map $\mathcal{P}\colon \mathscr{K}_4 \dashrightarrow \mathscr{K}_4^{\vee}$ is birational as well. Finally, we already observed that $\varphi_{\mu_1(3\Theta)-2\delta} = \varphi_{\mu_1(\Theta)}\circ \tau$, so that the two images $\mathscr{K}_4,\mathscr{K}_4^{\vee}$ are projectively equivalent. 
\end{itemize}
\end{proof}

The main aim of our paper is to generalize this result to generalized Kummer varieties. The map $\varphi_{\mu_n(\Theta)}$ can be generalized in a straightforward way, and we will do it in the next section. The other map can be generalized on Kummer fourfolds: the role of the Kummer quartic will be played by the Coble cubic of $A$ embedded by the linear system of $3\Theta$.

\section{A contraction on Jacobian generalized Kummer varieties}\label{sec:contractionmutheta}

We consider now the line bundle $\mu_n(\Theta)$ on $\Kum_n(A)\subseteq A^{[n+1]}$ and we want to describe the associated map. We can essentially do it by extending the proof of \Cref{thm:kummerduality}.(1). 

Recall first that $\mu_n(\Theta)$ is the restriction to $\Kum_n(A)$ of the line bundle $\Theta^{(n+1)}$ on $A^{[n+1]}$. 
An effective divisor linearly equivalent to this linear system is (cf. \cite[pp. 656]{FogartyPicardSchemePunctualHilbertScheme})
\[ T'_n \coloneqq \{ \xi \in A^{[n+1]} \,|\, \operatorname{Supp}(\xi) \cap \Theta \ne \emptyset \}  \]
so that, if we define
\begin{equation}\label{eq:divofmutheta}
T_n =T'_n\cap \Kum_n(A) = \{\xi\in \Kum_n(A) \,|\, \operatorname{Supp}(\xi)\cap \Theta \ne \emptyset \} 
\end{equation}
we have that $T_n \sim \mu_n(\Theta)$. Now consider the singular Kummer variety $\Sigma_n(A) = \{[a_0,\dots,a_n] \in A^{(n+1)} \,|\, a_0+\dots+a_n \sim 0 \}$. The theorem of the square for abelian varieties shows that if $[a_0,\dots,a_n]\in A^{(n+1)}$, then
\[ \Theta_{a_0}+\dots+\Theta_{a_n} \sim \Theta_{a_0+\dots+a_n} + n\Theta \sim (n+1)\Theta \]
Hence, there is a map
\[ u \colon \Sigma_n(A) \longrightarrow |(n+1)\Theta|,\qquad [a_0,\dots,a_n] \mapsto \Theta_{a_0}+\dots+\Theta_{a_n} \]
and we can consider the composition with the Hilbert-Chow morphism $\operatorname{HC}\colon \Kum_n(A) \to \Sigma_n(A)$. We prove now Theorem A from the introduction

\begin{Thm}[Theorem A]\label{thm:mapmutheta} 
With the previous notation, it holds that
\begin{enumerate}
\item The map $\operatorname{HC}\circ u \colon \Kum_n(A) \to |(n+1)\Theta|$ coincides with the map $\varphi_{\mu_n(\Theta)}$ defined  by the complete linear system $|\mu_n(\Theta)|$. 
\item The map induces a set-theoretic inclusion of $\Sigma_n(A)$ inside $|(n+1)\Theta|$, which is a closed embedding when $n=1,2$.
\end{enumerate}
\end{Thm}

\begin{proof}
\begin{enumerate}
\item Consider the hyperplane in $H_0\subseteq |(n+1)\Theta|$ consisting of all divisors passing through $0\in A$. We see that $u([a_0,\dots,a_n])\in H_0$ if and only if $0\in \Theta_{a_0}+\dots+\Theta_{a_n}$, which means that at least one of the $a_i$ is contained in $\Theta_0=\Theta$ by \Cref{lem:thetafacts}.(4). In other words $(\operatorname{HC}\circ u)^*H_0 \sim T_n$, as in \Cref{eq:divofmutheta}, 
so that $(\operatorname{HC}\circ u)^*\sO(1) \cong \mu_n(\Theta)$. Furthermore, we also see from the Riemann-Roch theorem for generalized Kummers \cite{Nieper03}, that $\dim |\mu(\Theta)| = (n+1)^2-1 = \dim |(n+1)\Theta|$. Hence, if we can show that the image of $\operatorname{HC}\circ u$ is not contained in any hyperplane of $|(n+1)\Theta|$, we get that $\operatorname{HC}\circ u = \varphi_{\mu_n(\Theta)}$. To do so, we will use the action of the Heisenberg group $\mathcal{H}(n+1,n+1)$: the map $u$ is equivariant with respect to the natural action of $A[n+1]$ on both sides: if $b\in A[n+1]$
\[ u(t_b([a_0,\dots,a_n])) =  \sum \Theta_{a_0+b} = t_b \left( \sum \Theta_{a_i} \right)  = t_b(u([a_0,\dots,a_{n+1}]))\]
Hence, the linear span of the image of $\operatorname{HC}\circ u$  must be an $A[n+1]$-invariant subspace of $|(n+1)\Theta|$. There is no proper invariant subspaces, because it would correspond to a proper invariant subspace for the action of $\mathcal{H}(n+1,n+1)$ on $H^0(A,(n+1)\Theta)$, but this is the Schr\"odinger representation which is irreducible.
\item The map $u\colon \Sigma_n(A) \hookrightarrow |(n+1)\Theta|$ is injective  because the irreducible components with multiplicities of an effective divisor are uniquely determined. The fact that this is also an embedding for $n=1$ is \Cref{thm:kummerduality}.(1). If instead $n=2$, one proceeds via vector bundles of rank three on $C$. The argument is known, but we write it down for completeness: recall the map $\Phi_3\colon \mathcal{SU}_C(3) \to |3\Theta|$ from \Cref{eq:mapfromSU} and its ramification locus $\mathcal{R} = \{  E\in \mathcal{SU}_C(3) \,|\, \iota^* E^{\vee} \cong E \}$. The natural map $j\colon\Sigma_2(A) \to \mathcal{R}$ given by $[a,b,c]\mapsto \sO_C(a)\oplus \sO_C(b)\oplus \sO_C(c)$ is an isomorphism onto the image (see \cite[Remark 2.3]{BMT2021}). Since $\Phi_3$ is a double cover, the restriction $\Phi_{3|\mathcal{R}}$ is an isomorphism onto the image, and by construction we have $u = \Phi_3\circ j$, so that $u$ is also an isomorphism onto the image.
\end{enumerate}
\end{proof}

\begin{Rem}
The theorem of the square also gives a quick proof of \cite[Proposition 6.1]{BMT2021}: consider the embedding $\varphi_{n\Theta}\colon  A \hookrightarrow \PP(H^0(A,n\Theta))$ for $n\geq 3$, take $a_0,\dots,a_n\in A$ and the translates $\Theta_{a_0},\dots,\Theta_{a_n}$. We show that $\Theta_{a_0}+\dots+\Theta_{a_n}$ is contained in a hyperplane in $\PP(H^0(A,n\Theta))$ if and only if $a_0+\dots+a_n\sim 0$. Observe that $\Theta_{a_0}+\dots+\Theta_{a_n}$ is contained in a hyperplane  if and only if $h^0(A,(n+1)\Theta-\Theta_{a_0}-\dots-\Theta_{a_n})>0$. The theorem of the square gives that
\[ (n+1)\Theta-\Theta_{a_0}-\dots-\Theta_{a_n} \sim \Theta-\Theta_{a_0+\dots+a_n}  \]
and this is effective if and only if $a_0+\dots+a_n\sim 0$.
\end{Rem}

\section{A covering of $\Kum_2(A)$ via $\Kum_1(A)$}\label{sec: covering}

We now want to construct a cover of $\Kum_2(A)$ via Kummer K3 surfaces. All the constructions of this section are valid for an \emph{arbitrary} abelian surface $A$, not only a Jacobian. 

\begin{Ex}[Kummer K3 surfaces]\label{ex:kummerK3}
	The variety $\Kum_1(A)$ is the usual K3 Kummer surface associated to $A$. More precisely, one considers the rational map
	\[ A \dashrightarrow \Kum_1(A); \quad  a \mapsto \{a,-a\}, \]
	This is $\iota$-invariant, and it can be resolved by blowing-up $A$ at the set $\operatorname{Fix}(\iota) = A[2]$ of two-torsion points. Furthermore, the involution $\iota$ lifts to an involution on $\operatorname{Bl}_{A[2]}A$, that we denote by the same name. Then $\Kum_1(A) \cong (\operatorname{Bl}_{A[2]}(A))/\iota$, and this is the Kummer K3 surface associated to $A$. The 16 exceptional divisors $E_{\epsilon},\epsilon \in A[2]$ of the blow-up correspond to the nonreduced schemes of lenght two supported at the points in $A[2]$.  Something similar holds for any $K_{1,c} = \overline{s}^{-1}(-c) \subseteq A^{[2]}$, see \eqref{eq:def Kn,a} for the definition: the rational map
	\[ A \dashrightarrow K_{1,c};  \quad b \mapsto \{b,-b-c\} \]
	factors through the involution $\iota_c = \iota\circ t_c $, and it can be resolved by blowing up $A$ at the set $\operatorname{\Fix}(\iota\circ t_c) = [2]^{-1}(-c) = \{b\in A \,|\, 2b+c=0\}$. The involution lifts to an involution with the same symbol on the blow-up, and $K_{1,c} \cong (\operatorname{Bl}_{[2]^{-1}(-c)} A) / \iota_c$. The 16 exceptional divisors $E_{\epsilon},\epsilon \in [2]^{-1}(-c)$  correspond to the nonreduced schemes of length two supported on the points of $[2]^{-1}(-c)$.
\end{Ex}

\subsection{A global cover} This can be generalized to fourfolds: there is a dominant rational map
 \[ v': A^{[2]} \dashrightarrow \Kum_2(A); \quad \zeta \mapsto \zeta \cup \{-\overline{s}(\zeta)\} \]
 Since the fibers of the summation map are isomorphic to Kummer K3 surfaces, this yields a rational cover of $\Kum_2(A)$ by these surfaces. Let's be more precise. As explained in  \cite[Proposition 2.2]{ES98}  the rational map
\[ u'\colon A^{[2]} \times A \dashrightarrow A^{[3]}, \qquad (\zeta,a) \mapsto \zeta \cup \{a\} \]
is resolved by blowing up along the incidence variety $A^{[2,1]} = \{(\zeta,a) \,|\, a\in \zeta\}$:
\[ \beta'\colon \operatorname{Bl}_{A^{[2,1]}}(A^{[2]}\times A) \longrightarrow A^{[2]}\times A\]
This blow up is naturally isomorphic to the nested Hilbert scheme 
\[A^{[2,3]} = \{(\zeta,\xi) \in A^{[2]}\times A^{[3]} \,|\, \zeta \subseteq \xi\}.\] 
More precisely, one has an identification 
 \[
 (\beta')^{-1}(\zeta,a) \cong \{(\zeta,\xi) \in A^{[2]}\times A^{[3]} \,|\, \zeta\subseteq \xi, \overline{s}(\xi)-\overline{s}(\zeta) = a \}
 \] 
We also have the closed embedding 
\[ j'\colon A^{[2]} \hookrightarrow A^{[2]}\times A; \quad \zeta \mapsto (\zeta,-\overline{s}(\zeta)) \]
and the nested Kummer fourfold
\[\Kum^{[2,3]}(A) = \{(\zeta,\xi) \in A^{[2]}\times \Kum_2(A) \,|\, \zeta\subseteq \xi\} \]
fits in the following commutative diagrams, where the first commutative square is cartesian.
\begin{equation}\label{eq:kummercoverglobal}
	\begin{tikzcd}
		\Kum^{[2,3]}(A) \ar[r,hook,"j"]\ar[dd,swap,"\beta"] & A^{[2,3]} \ar[rd,"u"]\ar[dd,swap,"\beta'"]\\
		& & A^{[3]}\\
		A^{[2]}\ar[r,hook,"j'"] & A^{[2]}\times A \ar[ru,dashed,swap,"u' "]
	\end{tikzcd}  
	\quad 
	\begin{tikzcd}
		 &\Kum^{[2,3]}(A) \ar[rd,"v = u\circ j"]\ar[dd,swap,"\beta"]\\
		 & & \Kum_2(A)\\
		 &A^{[2]} \ar[ru,dashed,swap,"v' = j'\circ u' "]
	\end{tikzcd} 	 
\end{equation}

\begin{Lem}
	With the previous notation, set $Z = j'^{-1}(A^{[2,1]})$.
	\begin{enumerate}
		\item $Z\subseteq A^{[2]}$ is  isomorphic to $\operatorname{Bl}_{A[3]}A$.
		\item It holds that  $\Kum^{[2,3]}(A) = \operatorname{Bl}_{Z_2}A^{[2]}$, in particular, it is smooth and irreducible and $\beta$ is birational. 
	\end{enumerate}
\end{Lem}
\begin{proof}
\begin{enumerate}
 \item By definition $Z = \{\zeta \in A^{[2]} \,|\, -\overline{s}(\zeta) \in \operatorname{Supp}(\zeta) \}$. If $a,b\in A$ are distinct, then $\{a,b\}\in Z$ if and only if $2a+b=0$ or $2b+a=0$. If instead $(a,v)\in A^{[2]}$ is nonreduced, we see that $(a,v)\in Z$ if and only if $a\in A[3]$. Consider then the subvariety $Z_2 = \{(a,b)\in A^2\,|\, a+2b=0\}$, which is isomorphic to $A$. The preimage $\widetilde{Z}_2 \subseteq \operatorname{Bl}_{\Delta_A} (A^2)$ is isomorphic to the blow-up of $A$ at $A[3]$ and under the double cover $\operatorname{Bl}_{\Delta_A}(A^2) \to A^{[2]}$, we get an isomorphism $\widetilde{Z}_2 \to Z$.
\item There is a natural map $\operatorname{Bl}_{Z}A^{[2]} \to \Kum^{[2,3]}(A)$ which is an isomorphism outside $\beta^{-1}(Z)$. Furthermore the composition $\operatorname{Bl}_{Z}A^{[2]} \to \Kum^{[2,3]}(A) \to A^{[2,3]}$ is a closed embedding, so we are done.
\end{enumerate}
\end{proof}

We consider the map $v\colon \Kum_{[2,3]}(A) \to \Kum_2(A)$ over the open set
\[ U_2(A)\coloneqq \Kum_2(A)\setminus \cup_{a\in A[3]} B_a \]
of schemes supported in at least two distinct points.  Set also $U^{[2,3]}(A) \coloneqq v^{-1}(U_2(A))$: the cover $v\colon U^{[2,3]}(A) \to U_2(A)$ is finite, flat (thanks to miracle flatness) of degree three, but it is not Galois. To determine the Galois cover, observe that there is a natural faithful action of $\mathfrak{S}_3$ on $A^2$ generated by the following two automorphisms, of order two and three respectively:
\[ (a,b) \mapsto (b,a), \qquad (a,b)\mapsto (b,-a-b).  \]
Notice that the fixed point of this action are exactly those of the form $(a,a),a\in A[3]$. Consider the subsets $Z_1 = \{(a,b)\in A^2 \,|\, 2a+b\sim 0\}$ and $Z_2 = \{(a,b) \,|\, a+2b\sim 0\}$ in $A^2$, as well as the diagonal $\Delta_A$. It is straightforward to see that $Z_1\cup Z_2 \cup \Delta_A$ is invariant for the action of $\mathfrak{S}_3$ and also that
\[ Z_1\cap Z_2 = Z_1 \cap \Delta_A = Z_2 \cap \Delta_A = \{(a,a)\,|\, a\in A[3]\}. \]
Let then $V = A^2\setminus \cup_{a\in A[3]}\{(a,a)\}$ and take $Z'_1=Z_1\cap V,Z'_2 = Z_2\cap V,\Delta'_A = \Delta_A\cap V$: these are three smooth pairwise disjoint codimension two varieties of $V$, so that the blow-up 
\[ \widetilde{U}_2(A) \coloneqq \operatorname{Bl}_{Z_1'\cup Z_2' \cup \Delta_A'}(V)\] 
is again smooth and it inherits the action of $\mathfrak{S}_3$ from $V$.

\begin{Prop}\label{prop:covers3}
There is a double cover $p\colon \widetilde{U}_2(A) \to U^{[2,3]}(A)$ such that the composed map
\[ v\circ p \colon \widetilde{U}_2(A) \longrightarrow U_2(A) \]
is the quotient by the action of $\mathfrak{S}_3$: $U_2(A)\cong \widetilde{U}_2(A)/\mathfrak{S}_3$.
\end{Prop}
\begin{proof}
The map $p\colon \widetilde{U}_2(A) \to U^{[2,3]}(A)$ should be defined on a general point $(a,b)$ as $(a,b)\mapsto (\{a,b\},\{a,b,-a-b\})$. To make this more precise, consider the partial blow-up $\operatorname{Bl}_{\Delta_A'}(V) \subseteq \operatorname{Bl}_{\Delta_A}(A^2)$: the composition of the double cover $\operatorname{Bl}_{\Delta_A}(A) \to A^{[2]}$ with the embedding $j'\colon A^{[2]} \hookrightarrow A^{[2]}\times A$ of \Cref{eq:kummercoverglobal} yields a map $p'\colon \operatorname{Bl}_{\Delta_A'}(V) \to A^{[2]} \times A$. A local computation  shows that the preimage $(p')^{-1}(Z)$ is the union of the two transforms of $Z_1',Z_2'$ inside $\operatorname{Bl}_{\Delta_A'}(V)$ and we have an induced map from $\operatorname{Bl}_{(p')^{-1}(Z)}(\operatorname{Bl}_{\Delta_A'}(V)) \cong  \widetilde{U}_2(A)$ to $\operatorname{Bl}_Z(A^{[2]}\times A) \cong A^{[2,3]}$. This is the double cover $p\colon \widetilde{U}_2(A) \to U^{[2,3]}(A)$ that we wanted. The composed map $v\circ p\colon \widetilde{U}^2(A) \to U_2(A)$ is defined on a general point $(a,b)$ as $(a,b) \mapsto \{a,b,-a-b\}$, hence it is invariant under the action of $\mathfrak{S}_3$.  This induces a finite and birational map $\widetilde{U}^2(A)/\mathfrak{S}_3\to U_2(A)$. Since $U_2(A)$ is smooth, this must be an isomorphism.
\end{proof}

\subsection{A cover by $\Kum_1(A)$-like surfaces}\label{sec: coomologia di Ka}

We consider again the diagram of \Cref{eq:kummercoverglobal}.

\begin{Lem}
		For any $c\in A$ consider the Kummer K3 surface $K_{1,c} = \overline{s}^{-1}(-c)\subseteq A^{[2]}$. 
		\begin{enumerate}
			\setcounter{enumi}{2}
			\item If $c\notin A[3]$, then $\beta^{-1}(K_{1,c}) \cong \operatorname{Bl}_{\{c,-2c\}} K_{1,c}$ is the blow up at a point outside the 16 exceptional divisors. 
			\item If $c\in A[3]$, then $\beta^{-1}(K_{1,c})$ is the union of two irreducible components: one isomorphic to $K_{1,c}$ and a component $T_c$ which is a $\P^1$-bundle over the exceptional divisor $E_{-c} \subseteq K_{1,c}$. 
		\end{enumerate}
\end{Lem}
\begin{proof} 
	\begin{enumerate}
	
		\item  Since  $c\notin A[3]$, we see that $K_{1,c}\cap Z = \{c,-2c\}$ so that $\beta^{-1}(K_{1,c})$ is the blow up of $K_{1,c}$ at $\{c,-2c\}$.
		\item If $c\in A[3]$, we see that $K_{1,c}\cap Z = E_{c}$, where $E_{c} \subseteq K_{1,c}$ is the divisor parametrizing nonreduced subschemes of length two supported at the point $c$. Then $\beta^{-1}(K_{1,c})$ is the union of $T_c = \beta^{-1}(E_{c})$ and the proper transform of $K_{1,c}$. The latter one is isomorphic to $K_{1,c}$. Furthermore, since $\eta \in E_{c}$ is a locally complete intersection, \cite{ES98} shows that the fibers of $T_c \to E_{c}$ are isomorphic to $\PP^1$. 
	\end{enumerate}\end{proof}
For any $c\in A$ we define:
\[
\widetilde{K}_{1,c} := 
\begin{cases}
	\beta^{-1}(K_{1,c}), & c\notin A[3] \\
	\text{ component of } \beta^{-1}(K_{1,c}) \text{ isomorphic to } K_{1,c}, & c\in A[3].
\end{cases}
\]
In particular, if $c\in A[3]$, then $\beta^{-1}(K_{1,c}) = \widetilde{K}_{1,c} \cup T_{c}$. We also consider the loci
\begin{equation}\label{eq:locipt} 
X_{c}:= \{\xi \in \Kum_2(A) \,|\, c \in \operatorname{Supp}(\xi) \} 
\end{equation}
of schemes whose support contains $c$. 

\begin{Prop}\label{prop: kum1 in kum2} With the previous notation, it holds that:
	\begin{enumerate}
		\item If $c\notin A[3]$ then $v\colon \widetilde{K}_{1,c} \to X_{c}$ is an isomorphism.
		\item If $c\in A[3]$, then $v\colon \widetilde{K}_{1,c} \to \Kum_2(A)$ is an isomorphism onto the image, and the map $v\colon T_{c} \to B_{c}$ is the resolution of singularities of the Brian\c{c}on surface of $c$. Furthermore, $X_{c} = v(\widetilde{K}_{1,c}) \cup B_{c}$.
	\end{enumerate}
\end{Prop} 
\begin{proof}
Point (1) is in \cite[page 9]{Hassett-Tschinkel-lagrangian-planes}. The first statement of point (2) is in \cite[Lemma 4.2]{floccari2025} and the second statement is in \Cref{subsec:briancon}. The last statement is clear.
\end{proof}

In light of this proposition, we will abuse notation and we will denote also the image $v(\widetilde{K}_{1,c})$ with $\widetilde{K}_{1,c}$ as well. It is a surface $\widetilde{K}_{1,c} \subseteq \Kum_2(A)$.
We now want to study the associated map in cohomology
\begin{equation}\label{eq:restr_in_cohom} 
v_c^*\colon H^2(\Kum_2(A),\ZZ) \longrightarrow  H^2(\widetilde{K}_{1,c},\ZZ). 
\end{equation}
We state a general lemma 

\begin{Lem}\label{lem:pullbackinjective}
    Let $X$ be a hyperk\"ahler variety and $Y\subseteq X$ a smooth subvariety. The restriction map in cohomology
    \[ \rho\colon H^2(X,\ZZ) \to H^2(Y,\ZZ) \]
    is injective if and only if $Y$ is not isotropic and if the map $\rho\colon \NS(X) \to \NS(Y)$ is injective.
\end{Lem}
\begin{proof}
    Let $\sigma_X$ be the holomorphic symplectic form of $X$. The transcendental lattice $T(X):=\NS(X)^{\perp_{q_X}} \subseteq H^2(X,\ZZ)$ is an irreducible weight two Hodge structure containing $\sigma_X$ \cite[Lemma 2.7]{huyK3}. Furthermore,  
   \[
    H^2(X,\QQ)\cong (T(X)\otimes_{\ZZ}\QQ)\oplus (\NS(X)\otimes_{\ZZ}\QQ) 
   \]
   and $H^2(X,\ZZ)$ is torsion free. Hence, if the map $\rho$ is injective, then the map between the Neron-Severi groups must be injective as well, and $\rho(\sigma_X)\ne 0$, meaning that $Y$ is not isotropic. Conversely if $Y$ is not isotropic, we see that $\rho(\sigma_X)\ne 0$, and then $\rho\colon T(X)\to H^2(Y,\ZZ)$ must be injective, because $T(X)$ is irreducible. Furthermore, $\rho(T(X))$ has a non-zero part of type $(2,0)$, so it cannot be that $\rho(T(X))=\rho(\NS(X))$, because the latter is purely of type $(1,1)$. Then, $\rho(T(X)) \cap \rho(\NS(X)) = 0$. This proves that $\rho$ respects the previous decomposition. Since both $\rho\colon T(X)\to H^2(Y,\ZZ),\rho\colon \NS(X) \to H^2(Y,\ZZ)$ are injective, the result follows. 
\end{proof}

For any $c\in A$ the Kummer K3 surface $K_{1,c}$ contains the exceptional divisors $E_{\epsilon},\epsilon \in [2]^{-1}(-c)$, with self-intersection $(E_{\epsilon}^2) = -2$, and there is also an embedding $\mu_1\colon \NS(A) \hookrightarrow \NS(K_{1,c})$, such that $(\mu_1(\Theta)^2)=2\cdot(\Theta^2) =4$.
 The minimal primitive overlattice $\mathcal{K}$ of $H^2(K_{1,c},\ZZ)$ containing all the classes $E_{\epsilon}$ is called the \emph{Kummer lattice} and has been precisely described, see for example \cite[Remark 2.3]{GarbagnatiSarti2016}. The Néron-Severi group of $\Kum_1(A)$ is an index two overlattice of 
$\mu_1(\NS(A)(2))\oplus \mathcal{K}$, given in \cite[Theorem 2.7]{GarbagnatiSarti2016}. In particular, the class $\frac{1}{2}\sum_{\epsilon \in [2]^{-1}(-c)} E_{\epsilon}$ belongs to $ \NS(K_{1,c})$. Now recall that if $c\in A[3]$, then $\widetilde{K}_{1,c}\cong K_{1,c}$, while if $c\notin A[3]$, then $\widetilde{K}_{1,c}$ is the blow-up of the Kummer K3 surface $K_{1,c}$  at the point $\{c,-2c\}$. We denote by $E_{c}$ the corresponding exceptional divisor, so that $(E_{c})^2=-1$. By definition, 
\[
\NS(\widetilde{K}_{1,c}) \cong 
\begin{cases}
	\NS(K_{1,c}) & \mbox{ if } c\in A[3],\\
	\NS(K_{1,c})\oplus \Z\cdot E_{c} & \mbox{otherwise.}
\end{cases}
\]
Since $-c\in [2]^{-1}(c)$ when $c\in A[3]$, the exceptional divisor $E_{c}$ also appears for $c\in A[3]$, as the divisor parametrizing nonreduced subschemes supported at $c$, see also Proposition \ref{prop: kum1 in kum2}, but in this case $(E_{c})^2=-2$.
With this in mind, following Floccari \cite{floccari2025}, we can state.

\begin{Prop}\label{restr in cohom}
	For any $c\in A$, consider the restriction in cohomology map of \eqref{eq:restr_in_cohom}. Then  $v^*(\mu_2(\Theta)) = \mu_1(\Theta)$ and
	\begin{equation}\label{restr to Kum1}
        		v^*(\delta) = \frac{1}{2}\sum_{\epsilon \in [2]^{-1}(c)
} E_\epsilon + E_{c}.
	\end{equation}
	Moreover $\rho_a^*$ is an embedding of Hodge structures of $H^2(\Kum_2(A),\Z)$ in $H^2(\widetilde{K}_{1,c},\Z)$ and induces a primitive embedding of lattices $H^2(\Kum_2(A),\Z)(2)\hookrightarrow H^2(\widetilde{K}_{1,c},\Z)$, where $H^2(\Kum_2(A),\Z)(2)$ is the integral cohomology of $\Kum_2(A)$, endowed with the Beauville-Bogomolov form multiplied by two.
\end{Prop}
\begin{proof}
    We fix $\Lambda=[2]^{-1}(-c)\cup \lbrace c \rbrace$. 
    The exceptional divisor $E$ of the Hilbert-Chow contraction $\mathrm{HC}$ intersects $X_{c}$ exactly along the divisors $E_\lambda$, $\lambda\in \Lambda$, so we write
    \begin{equation}
        v^*(\delta) = \frac{1}{2}\sum_{\lambda\in \Lambda } d_{\lambda}E_\lambda  \quad \text{for }d_\lambda\in \ZZ, \lambda\in \Lambda.
    \end{equation}
    For $c\in A[3]$ the coefficients are computed in \cite[Lemma 4.4 (i)]{floccari2025}; actually only the case $c=0$ is considered, but the others follow since otherwise we can use a translation by an element in $A[3]$. Suppose instead $c\notin A[3]$. In this case the $E_\lambda$'s are one dimensional fibers of the Hilbert-Chow morphisms, so the image of each one of them is a nodal singularity, which implies $(E\cdot E_\epsilon)=-2$, see \cite[Section 2]{Hassett-Tschinkel-lagrangian-planes}.
    We obtain
    \[
    -2=(v_a^*(E)\cdot E_\lambda)=(2v_a^*(\delta)\cdot E_\lambda)=-2d_\lambda.
    \]
    Recalling that $E_\lambda$ has self intersection $-1$ if $\lambda=c$ and $2$ otherwise, we obtain \eqref{restr to Kum1}.
        
    The inclusion of lattices $H^2(\Kum_2(A),\Z)(2)\hookrightarrow H^2(\widetilde{K}_{1,c},\Z)$ is proven in \cite[Remark 4.2]{Hassett-Tschinkel-lagrangian-planes} for $c\notin A[3]$ and in \cite[Lemma 4.8]{floccari2025} otherwise. However, we want to partly repeat the proof of the injectivity here, since we are going to use the same strategy later. We consider again any $c\in A$. To prove the inclusion of lattices one can use a family of Kummer fourfolds: consider a smooth and proper family $\mathcal{X}\to B$ of Kummer fourfolds of Jacobians on a smooth connected base $B$. This means that for any $t\in B$ we have a smooth curve $C_t$ of genus $2$, a Weierstrass point $x_{0,t}\in C_t$ with associated theta characteristic $\eta_t\sim x_{0,t}$, the Jacobian $A_t$ with the theta divisor $\Theta_{0,t}$, a point $c_t\in A_t$ and also $\mathcal{X}_t\cong \Kum_2(A_t)$. We then have a smooth subfamily $\mathcal{Y} \subseteq \mathcal{X}$ such that $\mathcal{Y}_{t}\cong \widetilde{K}_{1,c_t}$ for all $t\in B$. Since the restriction in cohomology $H^2(\mathcal{X}_t,\ZZ) \to H^2(\mathcal{Y}_t,\ZZ)$ is the specialization of a map of local systems on $B$, and since $H^2(\mathcal{X}_t,\ZZ)$ is torsion-free, to prove injectivity on one $t$ it is enough to prove injectivity for a single $t$. We can always find such a family that contains one given Kummer fourfold $\Kum_2(A)$ as well as one coming from an
abelian surface with Picard number 1. Hence, it is enough to prove the statement when $A$
itself has Picard number 1. We use \Cref{lem:pullbackinjective}: the surface $\widetilde{K}_{1,c}$ is not isotropic: see \cite{Hassett-Tschinkel-lagrangian-planes}[Remark 4.2] for $c\notin A[3]$ and \cite[Lemma 4.8]{floccari2025} for $c\in A[3]$. Hence it is enough to show that the map is injective on the Neron-Severi groups. In this case, $\NS(\Kum_2(A)) = \langle \mu_2(\Theta),\delta \rangle$. We have already computed the pullback of $v^*(\delta)$, so if we prove $v^*(\mu_2(\Theta))=\mu_1(\Theta)$ it is easy to get the injectivity that we are looking for. Observe that the lattice $\mathcal{K}_c=\langle E_\lambda \rangle_{\lambda\in \Lambda} $ has corank one in $\NS(\widetilde{K}_{1,c})$. Since the morphism associated to $\mu_2(\Theta)$ contracts any $E_\lambda$ to a point, $v^*(\mu(\Theta))$ lies in the orthogonal complement of $\mathcal{K}_c$ with respect to the intersection form i.e. $v^*(\mu_2(\Theta))=k\mu_1(\Theta)$ with $k>0$. But $v^*(\mu_2(\Theta))^2=[X_{c}]\cdot \mu_2(\Theta)^2=2q(\mu_2(\Theta))=4$ and we are done. This proves the embedding of lattices that we wanted,     primitivity follows by \cite[proof of Proposition 4.3]{floccari2025}.   
\end{proof}

\section{A contraction on Jacobian Kummer fourfolds}\label{sec:contractionD}

Now we want to consider the linear systems induced by the two divisors $\mu_2(\Theta)$ and $D:=\mu_2(2\Theta)-\delta$ on $\Kum_2(A)$. We have already seen in \Cref{thm:mapmutheta} that $\mu_2(\Theta)$ is base-point-free and induces a birational divisorial contractions of $\Kum_2(A)$ onto a singular Jacobian Kummer fourfolds. The same holds for $D$. The classes $\mu_2(\Theta)$ and $D$ are natural to consider: indeed, if $A$ is a Jacobian surface with $\NS(A)\cong \ZZ\cdot \Theta$, then \cite[Theorem 0.1]{Mori2021} shows that $\mu_2(\Theta)$ and $D$ generate the nef and the movable cone of $\Kum_2(A)$. In the case of an arbitrary Jacobian of a smooth genus two curves, these classes are still big and nef, and not ample:



\begin{Prop}\label{prop: bigandnef}
    If $A$ is the Jacobian of a smooth genus two curve, then $\mu_2(\Theta)$ and $D$ are big and nef divisors on $\Kum_2(A)$. They are not ample. 
\end{Prop}
\begin{proof}
     The statement for $\mu_2(\Theta)$ follows from \Cref{thm:mapmutheta}. Since $q(D,D)=2>0$, Fujiki's formula shows that $\int_{\Kum_2(A)} D^{4} >0 $, so that, if $D$ is nef, it is also big. Furthermore \cite[Example 6.1]{HassettTschinkel2010} shows that there is an effective divisor $F$ in $\Kum_2(A)$ such that $q(D,F)=0$ so that $D$ cannot be ample (we study the divisor $F$ thoroughly in \Cref{lem:FP1undleAndClass}). It is enough to prove that $D$ is nef: we will use an argument analogous to \cite[Theorem 3.8]{RiosHGM}, exploiting the description of the nef cone of $\Kum_2(A)$ in terms of wall divisors presented in \cite[\S 1.2]{MTW2018}. The class $D\in \NS(\Kum_2(A))$ fails to be nef if and only if there exists a class $\kappa\in \NS(\Kum_2(A))$ such that
    \begin{itemize}
        \item[(i)] $q(\kappa,\kappa)=-6$ and the divisibility of $\kappa$ in $H^2(\Kum_2(A),\ZZ)$ is $\gamma\in \{2,3,6\}$.
        \item[(ii)] $q(\kappa,D)<0$,
        \item[(iii)] $q(\kappa,\mu_2(\Theta))\geq 0$. 
    \end{itemize}
Since $\mu_2(\Theta)$ is nef, the third condition ensures that $\kappa$ is a positive multiple of the class of an effective curve via the natural embedding $H_2(\Kum_2(A),\ZZ)\hookrightarrow H^2(\Kum_2(A),\QQ)$, induced by the Beauville-Bogomolov-Fujiki form. In our setting, any such $\kappa$ would be of the form 
\begin{equation}\label{eq:kappagamma}
\kappa=\gamma\mu(B)-a\delta \mbox{ with }B\in \NS(A) \mbox{ primitive} , \, 6a^2=\gamma^2 (B^2)+6 .
\end{equation}

Condition (ii) and (iii) together show that $3a>\gamma (B\cdot \Theta) \geq 0$, 
and using the Hodge index theorem we see that 
\[
9a^2> \gamma^2(B\cdot \Theta)^2 \geq 2\gamma^2(B^2) = 12a^2-12.
\]
So it must be $a=1, (B^2)=0, (B\cdot \Theta)=1$. In this case \cite[Lemma 1.1]{Bauer98} shows that the class of $B$ is effective, and then
 \cite[Lemma 5.1]{Terakawa1998}  shows that the polarized abelian variety $(A,\Theta)$ must be a product of elliptic curves, which is impossible since $A$ is a Jacobian of a smooth genus two curve.
\end{proof}

\begin{Rem}
As a consequence of \Cref{prop: bigandnef} we obtain another proof that the nef cone of $\Kum_2(A)$ is generated by $\mu_2(\Theta)$ and $D$ when $\NS(A)\cong \ZZ\cdot \Theta$.
\end{Rem}

We study the  ray $D=\mu_2(2\Theta)-\delta$. The integral generator of the orthogonal line under the Beauville-Bogomolov-Fujiki form is $\mu_2(3\Theta)-2\delta$. Similarly to \cite[Proposition 15]{berimanivel}, under the natural isomorphism $H^0(A^{[3]})\cong \Sym^3H^0(A,3\Theta)$, there is an identification
\[
H^0(A^{[3]},3\Theta^{(3)}-2\delta) = \{\text{cubics in $\Pd^8$ singular along $A$}\}.
\]
The space is of dimension one and it is generated by the Coble cubic. The restriction of $\mu(3\Theta)-2\delta$ to $\Kum_2(A)$ is a prime exceptional divisor which is effective by \cite[pp. 452]{MTW2018}. In \cite[Example 6.1]{HassettTschinkel2010} the following geometric description for the divisor $\mu(3\Theta)-2\delta$ is given: set
\begin{equation}\label{eq:divisorF}
	F' := \{ \text{$\xi\in A^{[3]}$ such that $\xi\subseteq \Theta_a$ for some $a\in A$} \}, \quad F:= F\cap \Kum_2(A)
\end{equation}
Notice that any two distinct theta divisors in $A$ intersect in two points, so that any $\xi\in F'$ belongs to exactly one translate $\Theta_a$, and it is then identified via $\alpha_a\colon C \overset{\sim}{\to} \Theta_a$ with an effective divisor of length three on $C$. This shows that there is an isomorphism
\begin{equation}\label{eq:isoFprime} 
f\colon C^{(3)}\times A \overset{\sim}{\longrightarrow} F'; \quad (\xi,a) \mapsto \alpha_a(\xi). 
\end{equation} 

\begin{Lem}\label{lem:FP1undleAndClass}
Consider the isomorphism $f\colon C^{(3)}\times A \to F'$ of \eqref{eq:isoFprime} and the summation map $\overline{s}\colon A^{[3]} \to A$ of \eqref{eq:sumationmapKummer}. The following hold:
\begin{enumerate}
	\item If $(\xi,a)\in C^{(3)}\times A$, then $\overline{s}(f(\xi,a)) = \xi + 3a - 3\eta$. In particular $f(\xi,a)\in F$ if and only if $\xi \sim 3\eta-3a$.
	\item  If $\alpha\colon C^{(3)} \to A, \xi \mapsto 3\eta-\xi$ is the Abel-Jacobi map, then there is a cartesian diagram
	\begin{equation*}
		\begin{tikzcd}
			F\ar[r]\ar[d,"\pi"] & C^{(3)}\ar[d,"\alpha"]\\
			A\ar[r,swap,"\cdot 3"] & A
		\end{tikzcd}    
	\end{equation*}
	so that $F = \PP(\mathcal{E})$, where $\mathcal{E}$ is a vector bundle on $A$ given fiberwise by $\mathcal{E}_{|a} \cong H^0(C,3\eta-3a)$ for $a\in A$.
	\item The class of $F$ in $\NS(\Kum_2(A))$ is $\mu_2(3\Theta)-2\delta$.
	\item The restriction maps $H^2(\Kum_2(A),\ZZ) \to H^2(F,\ZZ)$ and $\NS(\Kum_2(A)) \to \NS(F)$ are injective. Furthermore, 
	\[ \mu_2(\Theta)_{|F} = \pi^*(-3\Theta)+2H , \qquad \delta_{|F} = \pi^*(-9\Theta)+4H  \]
   where $H$ is the class of the tautological bundle $\sO_F(1) = \sO_{\PP(\mathcal{E})}(1)$.
	\end{enumerate}
\end{Lem}
\begin{proof}
	\begin{enumerate}
		\item Let $\xi = x+y+z$ for $x,y,z\in C$ mutually distinct. Then $f(\xi,a) = \{x+a-\eta,y+a-\eta,z+a-\eta\}$ and the conclusion follows. 
		\item The fact that the diagram is cartesian follows from the previous discussion. Furthermore, the Abel-Jacobi map, identifies $C^{(3)}$ as the projective bundle $\PP(\mathcal{E}')$, where $\mathcal{E}'_{|b} \cong H^0(C,3\eta-b)$. Then $F = \PP(\mathcal{E})$, where $\mathcal{E} \cong (\cdot 3)^*\mathcal{E}'$. This is what we needed to prove.
		\item  Let $R\in H_2(\Kum_2 (A),\mathbb Z)$ be the class of a fiber, of  $F\to A$. We consider $H^2(\Kum_2 (A), \mathbb{Z}) \subseteq H^2(\Kum_2 (A),\mathbb Z)^\vee \subset H_2(\Kum_2 (A), \mathbb Q)$and by \cite[Corollary~3.6]{Markman2013} we have $R = \lambda F$ for some $\lambda\in \mathbb Z_{>0}$. Using \cite[\S 2, page 306]{HassettTschinkel2010} we compute
        \begin{align*}
            -2 = F\cdot R = \lambda R \cdot R 
            = \lambda \left(2-\frac{4}{9}6\right) 
            = \lambda \left(-\frac{2}{3}
            \right)
        \end{align*}
        where we have used $R = \mu(\Theta) - \frac{2}{3}\delta$ \cite[Example~6.1]{Hassett-Tschinkel-lagrangian-planes}. Hence the claim is proven.

        \item The injectivity on cohomology implies the injectivity on the Neron-Severi groups as well. For the map on cohomology, we can reduce to the case where $A$ has Picard number one, in the same way as in the Proof of \Cref{restr in cohom}.         
        We can use \Cref{lem:pullbackinjective}: since $F$ is a divisor, it cannot be isotropic in $\Kum_2(A)$. So it is enough to prove that the induced map on the Neron-Severi groups is injective: since in this case $\NS(\Kum_2(A)) =\langle \mu_2(\Theta),\delta \rangle$, the statement follows from the formulas for $\mu_2(\Theta)_{|F}$ and $\delta_{|F}$. 
        To prove the formulas, let $f$ be a fiber of $F\to A$. We see as in \cite[Example 6.1]{HassettTschinkel2010}, that $(\mu(\Theta)_{|F}\cdot f) = 2$ and that $(\delta_{|F}\cdot f)=4$. Then $\mu_2(\Theta)_{|F} =  \pi^*(a\Theta)+2H, \delta_{|F} = \pi^*(b\Theta)+4H$ for some $a,b\in \ZZ$.  We then look at the canonical divisor: in $\NS(F)$ we have $K_F = \pi^*(-\det \mathcal{E}) - 2H$, and $\det \mathcal{E} \cong  (\cdot 3)^* (\det \mathcal{E}') = (\cdot 3)^*(-\Theta) = -9\Theta$, so that $K_F = \pi^*(9\Theta)-2H$. On the other hand, we also have by the adjunction formula, and point (3), that $K_F = F_{|F} = \mu_2(3\Theta)_{|F} - 2\delta_{|F} = \pi^*((3a-2b)\Theta)-2H$. This shows that $3a-2b = 9$.   
		We also see that $D_{|F} = \mu_2(2\Theta)_{|F}-\delta_{|F} = \pi^*((2a-b)\Theta)$, and using the bilinear form associated with the quadratic Beauville-Bogomolov form (see also \cite[\S 2.1]{4horsemen}), we get 
        \[
        \int_{X} D^2 F^2 = 3\cdot (q(D)q(F)+2q(D,F)) = 3\cdot (2\cdot (-6)+0) = -36,
        \]
        on the other hand, by projection formula and the explict description of $F$ we get
        \begin{align*}
			\int_{X} D^2 F^2 &= \int_{F}D^2 F = \int_{F}\pi^*((2a-b) \Theta)^2\cdot (\pi^*(9\Theta)-2H) \\
			&= -2(2a-b)^2\int_F \pi^*\Theta^2\cdot H = -4(2a-b)^2. 
		\end{align*} 
		hence $(2a-b)^2 = 9$. Since $D$ is big and nef by \Cref{prop: bigandnef}, it must be that $(2a-b) =3$. Thus $a=-3,b=-9$. 
        

        
        
	\end{enumerate}	
\end{proof}

We proceed to study the linear system associated to $D = \mu_2(2\Theta)-\delta$. The linear system $|2\Theta^{(3)}-\delta|$ on the Hilbert scheme $A^{[3]}$ can be interpreted in terms of the map $\varphi_{2\Theta}\colon A \to \Pd^3$. More precisely,  one has \cite[Proof of Lemma 5.1]{EGLOonCobordismHilbertScheme} $H^0(A^{[3]},2\Theta^{(3)}-\delta) \cong \wedge^3 H^0(A,2\Theta) \cong H^0(A,2\Theta)^{\vee}$, and the rational map induced by this linear system is 
\[  \Pi\colon A^{[3]} \dashrightarrow |2\Theta| = \PP^3; \quad \xi \mapsto \Pi( \xi) \]
where $\Pi(\xi)\subseteq \Pd^3$ is the plane spanned by $\varphi_{2\Theta}(\xi)$. This is not defined precisely at those $\xi\in A^{[3]}$ such that $\varphi_{2\Theta}(\xi)$ lies on a line in $\Pd^3$. The restriction of $\Pi$ to $\Kum_2(A)$ induces another map
\begin{equation}\label{eq:phiDlinearsystem} 
\varphi_{\mathcal{L}}\colon \Kum_2(A) \dashrightarrow \PP^3
\end{equation}
which is given by the non complete linear system $\mathcal{L}$ corresponding to the image of the restriction map $H^0(A^{[3]},2\Theta^{(3)}-\delta) \to H^0(\Kum_2(A),D)$. We will first look at the base locus of $\mathcal{L}$. Notice that if $\xi \in A^{[3]}$ is such that $\varphi_{2\Theta}(\xi)$ has length at most two, then it is automatically contained in a line. If $\varphi_{2\Theta}(\xi)$ has length three, then any line containing it is a trisecant line to the Kummer surface of $A$. The study of such trisecant lines is fundamental in the theory of abelian varieties, in particular in the context of the Schottky problem and Welter's conjecture. In our case, we can summarize the results as follows:

\begin{Lem}\label{lem:baselocusL}
    Let $\xi\in A^{[3]}$ be such that $\varphi_{2\Theta}(\xi) \subseteq \Pd^3$ is of length three and contained in a line. 
    \begin{enumerate}
\item If $\xi$ is supported in three distinct points, then $\xi$ belongs to
\[ T_3 = \{\{a,b,c\}\,|\, \tau(a,b)\subseteq \Theta_c\cup \Theta_{-c} \}.  \]
\item If $\xi$ is supported in two distinct points, then $\xi$ belongs to
\begin{align*} 
T_2 &= \{ \{(a,v),c\} \,|\, a\notin A[2], \tau(a,-a)\subseteq \Theta_c\cup \Theta_{-c} \} \\
&\cup \{\{(\epsilon,v),c\} \,|\, \epsilon \in A[2], \,\, \tau(\epsilon,v) \subseteq \Theta_c\cup \Theta_{-c}\}.
\end{align*}
\item If $\xi$ is supported in a single point, which is not a two-torsion point, then $\xi$ belongs to
\[ T_3 = \{ \xi \in A^{[3]} \,|\, \operatorname{Supp}(\xi)=a \notin A[2],\,\, \PP(T_a\xi) = v \in \PP(T_a A), \, \tau(0,2a) \subseteq \tau(0,v)\cup \tau(2a,v)  \}.  \]
\end{enumerate}
If instead $\varphi_{2\Theta}(\xi)$ is of length at most two, then $\xi \in X_0 = \{\xi\in \Kum_2(A) \,|\, 0\in \xi\}$. Hence the linear system $\mathcal{L}$ has base locus 
\[ \operatorname{Bs}(\mathcal{L}) \subseteq X_0 \cup T_1 \cup T_2 \cup T_3. \]
\end{Lem}
\begin{proof}
This classification when $\varphi_{2\Theta}(\xi)$ has length three and is contained in a line follows from the discussion in \cite[page 204]{BD86}.
If instead $\varphi_{2\Theta}(\xi)$ has length less than three, then $\xi\supseteq \zeta$ for a certain $\zeta\in A^{[2]}$ which is contained in a fiber of $\varphi_{2\Theta}$. Since for this fiber it holds that $\overline{s}(\zeta)=0$, and since $\overline{s}(\xi)=0$, it must be that $0\in \xi$.
The last statement on the base locus of $\mathcal{L}$ now follows from the previous discussion.
\end{proof}

\begin{Lem}\label{lem:3torsiontheta}
    For any $a\in A$ the set $\Theta_a\cap A[3]$ contains at most two distinct points.
\end{Lem}
\begin{proof}
Assume that $\Theta_a\cap A[3]$ contains three distinct points. Via a translation we can assume that one of them is $0$: let $0,e,f\in A[3]\cap \Theta_a$ be mutually distinct. Then $0\sim x+a-\eta, e \sim y+a-\eta, f\sim z+a-\eta$ for $x,y,z\in C$ mutually distinct. The fact that $e,f\in A[3]$  means that $3x\sim 3y \sim 3z$. So if we look at the map $f\colon C\to \PP^1$ induced by $H^0(C,3x)$, we see that it has three distinct points with ramification three, but this is impossible by Riemann-Hurwitz.
\end{proof}

\begin{Prop}\label{prop:Dglobgen}
The line bundle $D=\mu(2\Theta)-\delta$ on $\Kum_2(A)$ is globally generated.
\end{Prop}
\begin{proof}
If we consider the linear system $\mathcal{L}$ of \eqref{eq:phiDlinearsystem}, we see that $\operatorname{Bs}(D)\subseteq \operatorname{Bs}(\mathcal{L})$. Furthermore, any line bundle on $\Kum_2(A)$ is invariant under the action of $A[3]$ on $\Kum_2(A)$, see \Cref{sec:extthetagroup}. Hence, it is enough to prove 
\[ \bigcap_{e\in A[3]} t_e(\operatorname{Bs}(\mathcal{L})) = \emptyset. \]
Take $\xi\in \Bs(\mathcal{L})$ and assume that $t_e(\xi)\in \Bs(\mathcal{L})$ for all $e\in A[3]$. We use Lemma \ref{lem:baselocusL}: since the support of $\xi$ has at most three points, there are at most $3$ distinct $e\in A[3]$ such $t_e(\xi) \in X_0$, so there must be at least $78 = 81-3$ elements $e\in A[3]$ such that $t_e(\xi)\in (T_1\cup T_2\cup T_3)\setminus X_0$. We show that this is impossible:

\begin{enumerate}
    \item Assume that $\xi = \{a,b,c\}$ is supported in three distinct points: then $t_e(\xi) \in T_3$ if and only if $\tau(a+e,b+e)= t_e(\tau(a,b)) \subseteq \Theta_{c+e}\cup \Theta_{-c-e}$ and that is equivalent to
\[\tau(a,b) = \Theta_a\cap \Theta_b \subseteq \Theta_c \cup \Theta_{-c+e}\]

Observe that it cannot be $\tau(a,b)\subseteq \Theta_c$, otherwise \Cref{lem:symmtau} proves that $c\in \tau(\tau(a,b)) = \{a,b\}$. Hence it must be that $\Theta_a\cap\Theta_b\cap\Theta_{-c+e} \ne \emptyset$. If we fix $f\in \Theta_a\cap\Theta_b$ we see that $f\in \Theta_{-c+e}$ if and only if $e\in \Theta_{c+f}$. By \Cref{lem:3torsiontheta} there are at most two $e\in A[3]$ with this property. Since there are at most $2$ elements $f\in \Theta_a\cap\Theta_b$, we have at most $4$ possibilities for $e$, and $4<78$, so that we conclude.

\item Assume now that $\xi = \{(a,v),c\}$ is supported in two distinct points. Then $t_e(\xi) = \{(a+e,v),c+e\}$ and one possibility for this to be in $T_2$ is that $a+e\notin A[2]$ and $\tau(a+e,-a-e)\in \Theta_{c+e}\cup \Theta_{-c-e}$. This is equivalent to 
\[ \tau(a,-a+e) = \Theta_a\cap \Theta_{-a+e} \subseteq \Theta_c\cup \Theta_{-c+e} \]
So in particular it must be $\Theta_a\cap\Theta_c\cap \Theta_{-a+e}\ne \emptyset$ or $\Theta_{a}\cap\Theta_{-c+e}\cap \Theta_{-a+e}\ne \emptyset$. It is easy to see that these conditions are equivalent, and reasoning as in the previous case, we see that we have at most $4$ possibilities for $e$. This was under the assumption that $a+e\notin A[2]$. If instead $a+e\in A[2]$, we see that we have at most $16$ possibilities for $e$. In total, we have at most $16+4 = 20<78$ possibilities, so that we conclude.

\item Assume that $\xi\in T_3$ has $\operatorname{Supp}(\xi) = a\notin A[2]$ and $\PP(T_a \xi) = v \in \PP(T_a A)$. If $t_e(\xi)\in T_3\setminus X_0$, we know that $\tau(0,2a-e)\subseteq \tau(0,v)\cup \tau(2a-e,v)$ and $\tau(0,2a-e)\nsubseteq \tau(2a-e,v)$ otherwise they would be equal. Hence $\tau(0,2a-e)\cap \tau(0,v) \neq \emptyset$. If $f\in \tau(0,v)$ is such that $f\in \tau(0,2a-e)$, then \Cref{lem:symmtau} shows that $2a-e\in \Theta_f$ so that $e\in \Theta_{2a-f}$. By \Cref{lem:3torsiontheta}, there are then at most $2$ possibilities for $e$. Since there are also at most $2$ possibilities for $f\in \tau(0,v)$, we have at most $4<78$ possibilities for $e$.
\end{enumerate}

\end{proof}

Now we consider the restriction of the map $\varphi_{D}\colon \Kum_2(A) \to |D|^{\vee}$ to the divisor $F\subseteq \Kum_2(A)$. Recall that $\pi \colon F \to A$ is naturally a $\PP^1$-bundle over $A$. We want to use this to show that $|D|^{\vee}$ can be naturally identified with the target space of $\varphi_{3\Theta}\colon A \to |3\Theta|^{\vee} \cong \Pd^8$. 

\begin{Rem}\label{rem:thetaD}
	The divisor class $D = \mu_2(2\Theta)-\delta$ is big and nef and primitive. Furthermore, $2$ is coprime with $3$ and $q(D,D) = 8-6=2$. Hence, O'Grady's Theorem \ref{thm:ogradytheta} shows that the theta group $\mathcal{G}_{\Kum_2}(D)$ is isomorphic to the Heisenberg group $\mathcal{H}(3,3)$ and that the space of global sections $H^0(\Kum_2(A),D)$ is the Schr\"odinger representation of this group. In particular $h^0(\Kum_2(A),D)=9$. This representation appears also as the group of global sections $H^0(A,3\Theta)$. 
\end{Rem}

\begin{Prop}\label{prop: A-and-N}
	There is a natural isomorphism of $\mathcal{H}(3,3)$-representations $H^0(\Kum_2(A),D) \cong  H^0(A,3\Theta)$, which yields a commutative diagram. 
	\[
	\begin{tikzcd}
		F\ar[r,hook] \ar[d,"\pi"]& \Kum_2(A)\ar[d,"\varphi_D"]  \\
		A\ar[r,hook,"\varphi_{3\Theta}"] &  \Pd^8  = |3\Theta|^{\vee} \cong |D|^{\vee} 
	\end{tikzcd}
	\]
\end{Prop}
\begin{proof}
	\Cref{lem:FP1undleAndClass} shows that $F\sim \mu_2(3\Theta)-2\delta$, so that there is a short exact sequence of sheaves on $\Kum_2(A)$:
	\[ 0 \longrightarrow \mathscr{O}(-\mu_2(\Theta)+\delta) \longrightarrow \mathscr{O}(D) \longrightarrow \mathscr{O}_F(D_{|F}) \longrightarrow 0 \]
	We observe that $\mu_2(\Theta)$ is big and nef and $q(-\mu_2(\Theta)+\delta,\mu_2(\Theta)) = -2$. Since the cone of effective divisors is dual (with respect to the Beauville-Bogomolov form) to the movable cone (cf. \cite[pp. 63]{BoucksomDivisorialZariskiDecomposition}), it follows $H^0(\Kum_2(A),-\mu_2(\Theta)+\delta)=0$.
    Furthermore \Cref{lem:FP1undleAndClass}, shows also that
	 $D_{|F} \cong \pi^*(3\Theta)$, so that, in summary there is an injective restriction map
	\[ H^0(\Kum_2(A),D) \hookrightarrow H^0(F,D_{|F}) \cong H^0(A,3\Theta) \]
	which must then be an isomorphism, since the two spaces have dimension 9 because of Riemann-Roch (\cite[Section 3.3]{debarreHK},\cite[Theorem 3.6.1.]{BL2004}). This shows that we have a commutative diagram as in the picture. We can even be more explicit: the points in $F$ are of the form $t_a(\alpha(\xi))$ for uniquely determined $a\in A$ and $\xi \in |3\eta-3a|$. The previous reasoning shows that
	\begin{equation}\label{eq:DonF} 
	\varphi_{D}(t_{a}(\xi)) = \varphi_{3\Theta}(a) \qquad \text{ for all } a\in A,\xi \in |3\eta-3a|. 
	\end{equation}
	Now we show that the restriction map is also an isomorphism of Heisenberg representations. Recall that $\mathcal{H}(3,3)$ is a central extension of $A[3]$ by an algebraic torus:
	\[  0 \longrightarrow \CC^{\star} \longrightarrow \mathcal{H}(3,3) \longrightarrow A[3] \longrightarrow 0\]
	and that the Schr\"odinger representation on $H^0(A,3\Theta)$ is the only one such that $\CC^{\times}$ acts on $H^0(A,3\Theta)$ via scalar multiplication and $A[3]$ acts on $\varphi_{3\Theta}(A)\subseteq |3\Theta|^{\vee}$ as translations \cite[Proposition 3, pag. 295]{mumford66}. If we consider the action of $\mathcal{H}(3,3)$ on $H^0(\Kum_2(A),D)$, we see that $\CC^{\times}$ acts via scalar multiplication. Thus we only need to check the action of $A[3]$. Any $c\in A[3]$ acts on $\Kum_2(A)$ by translations: $\xi \mapsto t_c(\xi)$. Thus, if we take any $a\in A$ and $\xi \in |3\eta-3a|$ we see from \eqref{eq:DonF} that
	\[ c\cdot \varphi_{3\Theta}(a) =  c \cdot \varphi_D(t_a(\xi)) = \varphi_D(t_c(t_a(\xi))) = \varphi_{D}(t_{a+c}(\xi)) = \varphi_{3\Theta}(a+c) \]
	and this is precisely what we wanted to show. 
\end{proof}

\subsection{Duality of the two projective models}

Now we want to show that the two projective models of $\Kum_2(A)$ considered before are dual to each other under the polar map $\mathcal{D}: \Pd^8 \dashrightarrow \P^8$ of the Coble cubic. First we need an observation about quadrics

\begin{Lem}\label{cor: no-quadric-for-N}
	The image $N = \varphi_D(\Kum_2(A))$ is not contained in any quadric.
\end{Lem}
\begin{proof}
    Since $\varphi_{3\Theta}(A)\subseteq N$, the space $H^0(\Pd^8,\mathcal{I}_{N}(2))$ of quadrics containing $N$ is a subspace of the space $H^0(\Pd^8,\mathcal{I}_{A}(2))$ of quadrics containing $\varphi_{3\Theta}(A)$. By \cref{prop: A-and-N}, these are both representations of the Heisenberg group, and it is known \cite[\S 10.7, pp. 308]{BL2004}, that $H^0(\Pd^8,\mathcal{I}_{A}(2))$ is irreducible. Hence, if 
    $H^0(\Pd^8,\mathcal{I}_{N}(2))\ne 0$, then $H^0(\Pd^8,\mathcal{I}_{N}(2)) =  H^0(\Pd^8,\mathcal{I}_{A}(2))$. But if this were the case, then $N \subseteq \cap_{Q\in |\mathcal{I}_A(2)|} Q$, and this means $N\subseteq A$, since $A$ is cut out by quadrics. But this is impossible because $D$ is big so that the image $N=\varphi_D(\Kum_2(A))$ must be four-dimensional.
\end{proof}

\begin{Thm}\label{thm: duality}
	There is a commutative diagram
	\[
\begin{tikzcd}
		& \Kum_2(A)\ar[ld, "\varphi_{\mu_2(2\Theta)-\delta}"']\ar[rd, "\varphi_{\mu_2(\Theta)}"]\\
		\Pd^8 \ar[rr, dashed, "\mathcal{D}"]&& \P^8
	\end{tikzcd}
\]

	As a consequence, the restriction of $\varphi_D$ to $\Kum_2(A)\setminus (E\cup F)$ is an isomorphism onto its image.
\end{Thm}
\begin{proof}
	Set again $D=\mu_2(2\Theta)-\delta$ and $F=\mu_2(3\Theta)-2\delta$. We first check that the image of composition $\cD \circ \varphi_D$ is nondegenerate, meaning that it is not contained in a hyperplane. This is true since $\mathcal{D}$ is defined by quadrics, hence any such hyperplane would give a quadric in $\Pd^8$ vanishing on $N=\varphi_D(\Kum_2(A))$. But this is impossible by \Cref{cor: no-quadric-for-N}. This shows that the composition $\cD\circ \varphi_{D}$ is determined by a sub-linear system  $V\subseteq H^0(\Kum_2(A),2D)$, with $\dim V=9$. Notice that we take $2D$ since $\cD$ is defined by quadrics. Furthermore, since the quadrics defining $\mathcal{D}$ vanish on $\varphi_{3\Theta}(A)$, and since $\varphi_D(F) = \varphi_{3\Theta}(A)$, we see that $V\subseteq H^0(\Kum_2(A),2D \otimes \mathscr{I}_{F}) \cong H^0(\Kum_2(A),2D-F)$.  One can compute that
	\[
	2D - F \sim 2(\mu_2(2\Theta)- \delta) - (\mu_2(3\Theta)-2\delta) \sim \mu_2(\Theta).
	\]
	and since $\dim V = 9 = \dim H^0(\Kum_2(A),\mu_2(\Theta))$, it must be that $V=H^0(\Kum_2(A),\mu_2(\Theta))$ and we are done.
		The last statement follows from the fact that on the open set $U = \Kum_2(A)\setminus (E\cup F)$, we have the equality of morphisms $(\cD\circ \varphi_D)_{|U} = (\varphi_{\mu_2(\Theta)})_{|U}$, and $(\varphi_{\mu_2(\Theta)})_{|U}$ is an isomorphism onto its image thanks to \cref{thm:mapmutheta}.
\end{proof}

\section{A geometric description via secant lines}\label{sec:secants}

We now want to give a concrete geometric description of $\varphi_{D}\colon \Kum_2(A) \to \Pd^8$.  We will do it in this section by studying the geometry of $A$ and the  curves $\Theta_a \subseteq A$ for $a\in A$. For simplicity, we will consider $A$ and all the translates $\Theta_a$ as being embedded in $\Pd^8$ by $\varphi_{3\Theta}$ so that we will often write simply $A,\Theta_a \subseteq \Pd^8$ instead of $\varphi_{3\Theta}(A),\varphi_{3\Theta}(\Theta_a)$.  
Recall also from Section \ref{sec: intro jac and Theta} that for any $a\in A$ there is the Abel-Jacobi map
\[ \alpha_a\colon C \xrightarrow{\sim} \Theta_a; \qquad x\mapsto x+a-\eta. \]

\begin{Lem}\label{lem:restrtotheta}
	Let $a,b  \in A$. Then $\alpha_{a}^*\Theta_ b \sim K_C+b-a$, as divisors on $C$. Furthermore, the composition 
    \[
    \begin{tikzcd}
        C\ar[r,"\alpha_a"] & \Theta_a \ar[r, hook, "\varphi_{3\Theta}"] & \Pd^8
    \end{tikzcd}
    \]
	is also the composition of the map $\varphi_{3K_C-3a}$ induced by the complete linear system $H^0(C,3K_c-3a)$ and a linear embedding. The image $\varphi_{3\Theta}(\Theta_a)$ spans a $4$-dimensional linear space of $\Pd^8$.
\end{Lem}
\begin{proof}
 Assume first that $b=0$. If $a=0$ the statement follows from the adjunction formula: $K_{C} \sim \alpha^*(K_A+\Theta) = \alpha^*(\Theta)$. If instead $a\ne 0$ then $\Theta_a$ and $\Theta$ intersect in two points, counted with multiplicity,  and Riemann-Roch yields  the equality of sets
		\begin{align*}
			\alpha_{a}^{-1}(\Theta) &= \{ x\in C \,|\, x-\eta+a \sim y-\eta \text{ for one } y\in C \} \\
			&= \{ x\in C \,|\, h^0(C,x+a)>0\} = \{x\in C\,|\, h^0(C,K_C-a-x)>0\} 
		\end{align*}
		so that this is the support of the unique effective divisor $x_1+x_2$ linearly equivalent to $K_C-a. $
        To conclude, it is enough to observe that both divisors $\alpha^*_{x_0-a}(\Theta)$ and $x_1+x_2$ have degree two, and they have the same support, hence they must be equal. 
		For an arbitrary $b$, we observe that,  $\Theta_b = t_{-b}^*\Theta$ so that $\alpha_a ^* \Theta_b \cong (t_{-b}\circ \alpha_a)^* \Theta \cong \alpha_{a-b}^*\Theta \sim K_C+b-a$.
		
		For the second statement, we see from what we already proved that $\alpha_a^{*}(3\Theta) \sim 3K_C-3a$. We also have a short exact sequence
		\[ 0 \longrightarrow \sO_{A}(3\Theta-\Theta_a) \longrightarrow \sO_A(3\Theta) \longrightarrow \sO_{\Theta_a}(3\Theta) \longrightarrow 0 \]
		and since $3\Theta-\Theta_a$ is ample, we see that $H^1(A,3\Theta-\Theta_a) = 0$ so that the map $H^0(A,3\Theta) \to H^0(\Theta_a,3\Theta)$ is surjective. This means precisely that the composition $\varphi_{3\Theta}\circ \alpha_a$ is induced by the complete linear system of $3K_C-3a$ and a linear embedding. Finally, we see from Riemann-Roch that $h^0(C,3K_C-3a)=6+1-2 = 5$, so that $\varphi_{3\Theta}(\Theta_a)$ spans a $4$-dimensional linear space in $\Pd^8$.
\end{proof}

For any $a\in A$ we denote the $4$-dimensional linear space of \Cref{lem:restrtotheta} as
\[ \Pd^4_a := \text{ linear span of } \Theta_a \text{ in } \Pd^8 \]
We will also need to study secant lines to $\varphi_{3\Theta}(A)\subseteq \Pd^8$ using Theorem \ref{thm:invtau}: if $\xi \in A^{[2]}$ is a length two subscheme, we denote the corresponding secant line by
\[ \ell({\xi}) := \text{ line spanned by } \xi \text{ in } \Pd^8.\]

\begin{Rem}\label{rem:notrisecantA}
	By \cite[Corollary 4.2]{Terakawa1998}, the line bundle of $3\Theta$ is $2$-very ample, meaning that $\varphi_{3\Theta}(A)$ has no trisecant lines. Hence the scheme $\xi \in A^{[2]}$ is uniquely determined by the secant line $\ell(\xi)$.
\end{Rem}

\begin{Lem}\label{lem:Pafacts}
	Let $a,b \in A$ be distinct.
	\begin{enumerate}
		\item $\Pd^4_a \cap \Pd^4_b  = \ell(\Theta_a\cap \Theta_b) = \ell(\tau(a,b))$.
		\item $\Pd^4_a \cap A = \Theta_a$.
		\item The line $\Pd^4_a\cap \Pd^4_b$, as well as the intersection $\Theta_a\cap \Theta_b$, determine $\{a,b\}$. 
	\end{enumerate}
\end{Lem}
\begin{proof}
	\begin{enumerate}
		\item  The two linear spaces $\Pd^4_a,\Pd^4_b$ span a hyperplane in $\Pd^8$, cut out by an element in $H^0(A,3\Theta-\Theta_a-\Theta_b) \cong \mathbb{C}$. Hence $\Pd^4_a\cap \Pd^4_b$ must be a line. Since this line contains the length two scheme $\Theta_a\cap \Theta_b$,  it must be the line spanned by this scheme.
		
		\item For sure $\Theta_a\subseteq \Pd^4_a\cap \Theta_a$. Assume that there is a $b\in \Pd^4_a\cap A, b\notin \Theta_a$, in particular $b\ne a$. Then $b\in \Theta_b$ so that $b\in \Pd^4_a\cap \Pd^4_b$. Then the line $\Pd^4_a\cap \Pd^4_b$ contains $b$ and $\Theta_a\cap \Theta_b$, meaning that it is a trisecant line to  $\Theta_b$. But this is impossible because of  \Cref{rem:notrisecantA}.
        
		\item The line $\Pd^4_a\cap \Pd^4_b = \ell(\Theta_a\cap \Theta_b)$ determines $\Theta_a\cap\Theta_b$ because of Remark \ref{rem:notrisecantA}. Then we know that $\{a,b\} = \tau(\Theta_a\cap \Theta_b)$.
	\end{enumerate}
\end{proof}

We now want to classify when two secant lines to $A\subseteq \Pd^8$ meet: this means that there are four points on $A$ that span a plane, rather than a $3$-space in $\Pd^8$. Since $A$ has no trisecant lines, this means that the linear system $H^0(A,3\Theta)$ on $A$ does not separate the four points (see Section \ref{sec:appendix precise terracini} for the definition). We can study this situation more generally:

\begin{Prop}\label{prop:4schemes}
Let $Z\subseteq A$ be a finite subscheme of length $4$. The linear system $H^0(A,3\Theta)$ does not separate $Z$  if and only if $Z\subseteq \Theta_a$ for a certain $a\in A$, and $Z = \alpha_a(x_1+x_2+x_3+x_4)$, with $x_1+x_2+x_3+x_4 \in |2K_C-3a|$.
\end{Prop}
\begin{proof}
Since $H^0(A,3\Theta)$ separates all subschemes of length $3$, \cite[Theorem 3.2.1]{beltrametti1991zero} shows that if $H^0(A,3\Theta)$ does not separate $Z$, then there is an effective divisor $D\subseteq A$ such that $Z\subseteq D$ and 
\[ 3(\Theta\cdot D) - 4 \leq (D\cdot D) \leq \frac{3}{2}(\Theta\cdot D) <4  \]
In particular, $(\Theta\cdot D) \leq 2$, and since $\Theta$ is ample we have two possibilities.
If $(D\cdot \Theta) = 1$ then $(D\cdot D)=0$, but then \cite[Lemma 5.1]{Terakawa1998} shows that the polarized abelian variety $(A,\Theta)$ is a product of elliptic curves, and this is impossible since $A$ is a Jacobian of a smooth genus two curve. If instead $(D\cdot\Theta)=2$, then the Matsusaka-Ran criterion \cite{ran}, shows that $D$ must be have the form $\Theta_a$ for a certain $a\in A$. Hence $Z = \alpha_a(x_1+\dots+x_4)$ for an effective divisor on $C$, and $H^0(A,3\Theta)$ fails to separate $Z$ if and only if $3K_C-3a$ fails to separate $x_1+\dots+x_4$, i.e. if and only if $h^0(C,3K_C-3a-x_1-x_2-x_3-x_4)>1$. Since $3K_C-3a-x_1-x_2-x_3-x_4$ has degree two, this last condition is equivalent to $3K_C-3a-x_1-x_2-x_3-x_4\sim K_C$, and we conclude.
\end{proof}

Now we look at secant lines:

\begin{Prop}\label{prop:meetingsecants}
	Let $\zeta,\zeta' \in A^{[2]}$ be distinct. The two secant lines  $\ell(\tau(\zeta)),\ell(\tau(\zeta'))$ meet if and only if  one of the following conditions hold
	\begin{enumerate}
		\item $\zeta,\zeta'\subseteq \Theta_e \,\,\text{ for one } e\in A$. 
		\item $\zeta = \{a,b\},\zeta' = \{a,c\}$ with $a,b,c$ mutually distinct and $a+b+c\sim 0$.
		\item $\zeta = \{a,b\}, \zeta' = (a,v)$ with $a,b$ distinct, $v\in \PP(T_aA)$ and $2a+b\sim 0$.			
		\item $\zeta = (a,v), \zeta' = (a,w)$ with  $v,w\in \PP(T_a A)$ distinct and $3a \sim 0$.
	\end{enumerate}
	Furthermore, $\ell(\tau(\zeta)),\ell(\tau(\zeta'))$ meet at a point in $A$ if and only if condition (1) holds. In this case, the intersection point is $e$. 
\end{Prop}
\begin{proof}
		First we observe that $\ell(\tau(\zeta))\cap \ell(\tau(\zeta')) = e\in A$, if and only if $e\in \tau(\zeta)\cap \tau(\zeta')$, since $A$ has no trisecant lines by \Cref{rem:notrisecantA}.  \Cref{lem:symmtau} shows that $e\in \tau(\zeta)\cap \tau(\zeta')$ if and only if $\zeta,\zeta'\subseteq \Theta_e$. This proves the last statement.
		For the rest, we can assume that $\tau(\zeta)\cap\tau(\zeta')=\emptyset$, otherwise we are in the case above. Then the two secant lines $\ell(\tau(\zeta)),\ell(\tau(\zeta'))$ meet if and only if $H^0(A,3\Theta)$ fails to separate the scheme $\tau(\zeta)\cup \tau(\zeta')$. By \Cref{prop:4schemes}, this means that $\tau(\zeta),\tau(\zeta') \subseteq \Theta_a$ for a certain $a\in A$, and by \Cref{lem:symmtau} we see that $a\in \zeta\cap\zeta'$.  We now have three possibilities for $\zeta,\zeta'$:  

        \medskip
        	
			\textbf{(a)} $\zeta=\{a,b\},\zeta'=\{a,c\}$ with $a,b,c\in A$ mutually distinct. Then $\tau(\zeta) = \Theta_a\cap\Theta_b$ and $\tau(\zeta') = \Theta_a\cap \Theta_c$.  By  \cref{lem:thetafacts}.(6), there are points $x,y,z,w \in C$, such that 
			\begin{align*} 
				x-\eta+a&\sim y-\eta+b, & \iota(y)-\eta+a &\sim \iota(x)-\eta+b, \\
				z-\eta+a&\sim w-\eta+c, & \iota(w)-\eta+a&\sim \iota(z)-\eta+c 
			\end{align*}
			Then  $\tau(\zeta) = \alpha_a(x+\iota(y))$ 
and  $\tau(\zeta') = \alpha_a(z+\iota(w))$. By \Cref{prop:4schemes}, the linear system $H^0(A,3\Theta)$ fails to separate $\tau(\zeta)\cup \tau(\zeta')$ if and only if $x+\iota(y)+z+\iota(w)\sim 2K_C-3a$. That is equivalent to
            \begin{align*}
				0 &\sim x+\iota(y)+z+\iota(w)-2K_C+3a \sim x-(K_C-\iota(y))+z-(K_C-\iota(w))+3a \\
				&\sim x-y+z-w+3a \sim b-a+c-a+3a \sim a+b+c.
			\end{align*}
		  
           \textbf{(b)} $\zeta=\{a,b\},\zeta'=(a,v)$ with $a,b\in A$  distinct, $v = [z,\iota(z)]\in \PP(T_a A)$. Recall that $\tau(\zeta')=\lbrace z+a-\eta,\iota(z)+a-\eta\rbrace$ or $\tau(\zeta')=\lbrace z+a-\eta,[z,z]\rbrace$ according to whether $z\neq \iota(z)$ or not, cf. \cref{thm:invtau}. By  \cref{lem:thetafacts}.(6), there are points $x,y \in C$, such that 
		   \begin{align*} 
		   	x-\eta+a&\sim y-\eta+b, & \iota(y)-\eta+a &\sim \iota(x)-\eta+b.
		   \end{align*}
		   Then  $\tau(\zeta) = \alpha_a(x+\iota(y))$ and  $\tau(\zeta') = \alpha_a(z+\iota(z))$, and we conclude reasoning as in case (a).
		   
           \textbf{(c)} $\zeta=(a,v),\zeta'=(a,w)$ with $v=[x,\iota(x)],w=[y,\iota(y)]\in \PP(T_aA)$ distinct. Then $\tau(a,v) = \alpha_a(x+\iota(x)),\tau(a,w) = \alpha_a(y+\iota(y))$ and we conclude reasoning as in case (a).

		   \end{proof}

\begin{Cor}\label{cor:intersectionthreespaces}
Let $a,b,c\in A$ be mutually distinct. Then 
\[ \Pd^4_a \cap \Pd^4_b \cap \Pd^4_c \ne \emptyset \quad \text{ if and only if } a+b+c\sim 0 \]
In this case the intersection is a single point.
\end{Cor}
\begin{proof}
We know from \Cref{lem:Pafacts} that $\Pd^4_a \cap \Pd^4_b = \ell(\tau(a,b))$ and that $\Pd^4_a \cap \Pd^4_c = \ell(\tau(a,c))$. Hence $\Pd^4_a \cap \Pd^4_b \cap \Pd^4_c = \ell(\tau(a,b))\cap \ell(\tau(b,c))$. The conclusion follows from \Cref{prop:meetingsecants}.
\end{proof}

With \cref{prop:meetingsecants} in hand, we can give a geometric description of $\varphi_D$.

\begin{Thm}\label{prop:psiphiD}
	The map $\varphi_D\colon \Kum_2(A) \to \Pd^8$, is defined for every scheme supported in at least two points as follows:
	\begin{align*}
	\varphi_D(\{a,b,c\}) &= \Pd^4_a \cap \Pd^4_b \cap \Pd^4_c = \ell(\tau(a,b))\cap \ell(\tau(a,c)) & a,b,c \in A \text{ mutually distinct }, \\
	\varphi_D(\{(a,v),b\}) &= \ell(\tau(a,b))\cap \ell(\tau(a,v)) & a \notin A[3],v\in \PP(T_a A). 
	\end{align*}
\end{Thm}
\begin{proof}
We follow the notation of \Cref{prop:covers3}: let $U_2(A)\subseteq \Kum_2(A)$ be the open subset of schemes supported in at least two distinct points. \Cref{prop:covers3} gives a Galois cover $\widetilde{U}_2(A) \to U_2(A)$ with respect to the action of the symmetric group $\mathfrak{S}_3$. We also have a double cover $p\colon \widetilde{U}_2(A) \to U^{[2,3]}(A) \subseteq \Kum^{[2,3]}(A)$ and an induced map $p'\colon \widetilde{U}_2(A) \to A^{[2]}$. The morphism
	\[ \widetilde{\psi}\colon \widetilde{U}_2(A) \longrightarrow \Pd^8; \quad 
	\begin{matrix} 
	\widetilde{\psi}(x) &= \ell(\tau(a,b))\cap \ell(\tau(a,-a-b)) & \text{ if } p'(x) = \{a,b\}\, a\ne b \\
	\widetilde{\psi}(x) &= \ell(\tau(a,v))\cap \ell(\tau(a,-2a)) & \text{ if } p'(x) = (a,v), a \in A\setminus A[3]. 
	\end{matrix}  
	\]
	is well defined because of Proposition \ref{prop:meetingsecants}, and one can check that it is invariant by the action of $\mathfrak{S}_3$. Hence, it descends to a morphism $\psi$ on $U_2(A)$. We have to show that $\psi = (\varphi_D)_{|U_2(A)}$. We will do it by looking at the restriction to the divisor $F\subseteq \Kum_2(A)$, which is a $\PP^1$ bundle $\pi \colon F\to A$. Suppose for now that 
	\begin{equation}\label{eq:restronF}
	\operatorname{codim}(F\setminus (U_2(A)\cap F),F)\geq 2, \quad \text{and} \quad   \psi(\xi) = \varphi_{3\Theta}(\pi(\xi)) = \varphi_D(\xi) \text{ for all } \xi \in U_2(A)\cap F.
	\end{equation} 
	
	 Then $\psi_{|U_2(A)\cap F}$ extends to $\varphi_{3\Theta}\circ \pi$ on the whole of $F$, meaning that it is induced by the complete linear system $H^0(F,\pi^*(3\Theta))$. This proves that $(\psi^*\sO_{\Pd^8}(1))_{|U\cap F} \cong \pi^*(3\Theta)_{|U_2(A)\cap F} \cong (\mu_2(2\Theta)-\delta)_{|U\cap F}$. Lemma \ref{lem:FP1undleAndClass}.(4) shows that the restriction to $F$ gives an injection 
	\[ \NS(U_2(A)) \cong \NS(\Kum_2(A)) \hookrightarrow \NS(F) \cong \NS(F\cap U_2(A)) \]
	where the isomorphisms are due to the fact that the complements of the two open subsets have codimension at least two. In particular, $\psi^*\sO_{\Pd^8}(1)_{|U_2(A)} \cong (\mu_2(2\Theta)-\delta)_{|U_2(A)}$. Furthermore, since $\psi_{|F} = \varphi_{3\Theta}\circ \pi$ and since $\varphi_{3\Theta}(A)$ is not contained in any hyperplane in $\Pd^8$, the same is true of $\psi(U_2(A))$. This shows that $\psi_{|U_2(A)}$ is induced by a linear subsystem $V\subseteq H^0(U_2(A),\mu_2(2\Theta)-\delta)$ of dimension $\dim V=9$ and by comparing dimensions we see that $V = H^0(U_2(A),\mu_2(2\Theta)-\delta)$. This shows that $(\varphi_D)_{|U_2(A)\cap F} = g\circ \psi_{U_2(A)\cap F}$ for a linear change of coordinates $g\in \operatorname{PGL}(9,\CC)$, but we have already observed that $\psi_{|F} = (\varphi_{|D})_{|F}$ and the image of these restrictions is $\varphi_{3\Theta}(A)$, that spans the whole $\Pd^8$. Hence $g=\operatorname{id}_{\Pd^8}$.

	 To conclude, we need to show that the conditions of \cref{eq:restronF} are satisfied. Every element in the fiber of $\pi\colon F \to A$ over $a\in A$ is in the form $\alpha_a(\xi) \subseteq \Theta_a$, with $\xi\in |3\eta-3a|$. If $a\in A$, we see that $3\eta-3a \sim 3x$ for a point $x\in C$ if and only if $a \sim \eta-x+b$ for $b\in A[3]$. In particular, there can be only finitely many such $x$ and for a general $a\in A$ there are none. This proves that $U_2(A)\cap F$ is open in $F$ with a complement of codimension at least two. Now take a general $a\in A$ and a general reduced divisor $\xi \in |3\eta-3a|$. Then $\alpha_a(\xi) \subseteq \Theta_a$, and \cref{prop:meetingsecants} shows that
	 \[ \psi(\{\alpha_a(\xi)\}) =  \varphi_{3\Theta}(a)  \]
	 which is what we wanted to show.
\end{proof}

In particular, we record the following consequence

\begin{Cor}\label{cor:secantlineinN}
Let $a\in A\setminus A[3]$ and consider $P_a=\{\{(a,v),-2a\}\,|\, v\in \PP(T_aA)\} \subseteq \Kum_2(A)$. Then $\varphi_D$ induces an isomorphism
\[ \varphi_D\colon P_a \longrightarrow \ell(\tau(a,-2a))\]
\end{Cor}
\begin{proof}
By \cref{prop:psiphiD}, $\varphi_D$ induces a map $\varphi_{D}\colon P_a \to \ell(\tau(a,-2a))$ and since both these curves are isomorphic to $\PP^1$, the map is an isomorphism if and only if it has degree one.  We can check that the map is injective directly via \cref{prop:meetingsecants}. Otherwise, we can also observe that $\mu_2(2\Theta)_{|P_a} \cong \mathcal{O}_{P_a}$, since $P_a$ is a fiber of the Hilbert-Chow morphism, and that $(-\delta\cdot P_a)=1$, because of the computation in \cite[pag. 8]{Hassett-Tschinkel-lagrangian-planes}, so that $(D\cdot P_a)=((\mu_2(2\Theta)-\delta) \cdot P_a)=1$.
\end{proof}

\subsection{The Brian\c{c}on surfaces and tricanonical curves of genus two}

Proposition \ref{prop:psiphiD} describe $\varphi_D$ outside the Brian\c{c}on surfaces $B_a\subseteq \Kum_2(A),a\in A[3]$. Observe that if $a\in A[3]$, then the composition $C \overset{\sim}{\to} \Theta_a \overset{\varphi_{3\Theta}}{\hookrightarrow} \Pd^4_a$ is the embedding of $C$ via the tricanonical line bundle $3K_C$. In this section we will always consider $a\in A[3]$. We recall some facts about these curves:

\begin{Rem}\label{rem:tricanonicalcone} 
 The hyperelliptic cover $\varphi_{K_C}\colon C \to \PP^1_{K_C}$ is branched over $6$ points, so that 
 \[ (\varphi_{K_C})_*\sO_C \cong \sO_{\PP^1_{K_C}} \oplus \sO_{\PP^1_{K_C}}(-3), \quad (\varphi_{K_C})_*\sO(3K_C) \cong (\varphi_{K_C})_*\sO_C \otimes \sO_{\PP^1_{K_C}}(3) \cong \sO_{\PP^1_{K_C}}(3)\oplus \sO_{\PP^1_{K_C}}.\]
 Then there is a decomposition 
 \[ H^0(\Theta_a,3\Theta) \cong H^0(C,3K_C) \cong H^0(\PP^1_{K_C},\sO_{\PP^1_{K_C}}(3)) \oplus H^0(\PP^1_{K_C},\sO_{\PP^1_{K_C}}). \] 
 This coincides with the decomposition of $H^0(C,3K_C)$ into  invariant and anti-invariant sections with respect to the hyperelliptic involution $\iota\colon C \to C$. 
 This decomposition singles out a  point $O_{a} = \PP(H^0(\PP^1_{K_C},\sO_{\PP^1_{K_C}})^{\vee})  $, and a $3$-space $\Pd^3_{a} = \PP(H^0(\PP^1_{K_C},\sO_{\PP^1_{K_C}}(3))^{\vee}) $ inside $\Pd^4_a$. The linear system  $H^0(\PP^1_{K_C},\sO_{\PP^1_{K_C}}(3))$ induces an embedding $\PP^1_{K_C} \hookrightarrow \Pd^3_a$ whose image is the twisted cubic that we denote by $R_{a}$. The composition $C\to \Theta_a \hookrightarrow \Pd^4_a \dashrightarrow \Pd^3_a$ ( where the last map is the projection from $O_a$) is the hyperelliptic cover $C\to R_a \cong \P^1_{K_C}$.  
 This shows that $\Theta_{a} \subseteq \Pd^4_a$ is contained in the cone $\Sigma_a \subseteq \Pd^4_a$ with vertex $O_{a}$ over the twisted cubic $R_a \subseteq \Pd^3_a$. Furthermore, if $v_a\in R_a$ corresponds to $v = [x,\iota(x)] \in \PP^1_{K_C}$ , then the intersection of $\Theta_a$ with the line $\ell(O_a,v_a)$ between $O_a$ and $v_a\in R_a$ is the fiber of the hyperelliptic cover over $v_a$:
 \begin{equation}\label{eq:thetalinecone}     
  \Theta_a \cap \ell(O_a,v_a) = \alpha_a(x+\iota(x)). 
\end{equation}  
\end{Rem} 

With this in mind, we can describe the map $\varphi_D$ on the Brian\c{c}on surface $B_a$:

\begin{Prop}\label{prop:mapDbriancon}
	Let $a\in A[3]$. The map ${\varphi_D}$ induces an isomorphism
	\[ {\varphi_{D}}_{|B_a}\colon B_a \overset{\sim}{\longrightarrow} \Sigma_a \subseteq \Pd^4_a \]
	such that, if $v\in \PP(T_a A)$ and $(a,v)\in A^{[2]}$ is the corresponding non-reduced scheme, then the locus $\{\xi \in B_a \,|\, \xi \supseteq (a,v) \}$ is mapped isomorphically onto the line $\ell(O_a,v)$. In particular, the unique non-curvilinear scheme $\operatorname{Spec} \sO_{A,a}/\mathfrak{m}_a^2$ is mapped to $O_a$.
\end{Prop}
\begin{proof}
	We first prove a general fact: if $\xi \in \Kum_2(A)$ and $\zeta\in A^{[2]}$ are such that $\zeta \subseteq \xi$, we want to show that
	\begin{equation}\label{eq:mapDphieta} 
	\varphi_D(\xi) \subseteq \ell(\tau(\zeta))
	\end{equation}
	From the diagram of \cref{eq:kummercoverglobal} we have the irreducible subvariety $\Kum^{[2,3]}(A) \subseteq A^{[2,3]}$ defined by $\Kum^{[2,3]}(A) = \{ (\zeta,\xi) \in A^{[2]} \times \Kum_2(A) \,|\, \zeta\subseteq \xi\}$ which covers $\Kum_2(A)$. Hence, to show that \eqref{eq:mapDphieta} holds for every $(\zeta,\xi) \in \Kum^{[2,3]}(A)$ is is enough to prove it for a general element: we can assume $\xi = \{a,b,c\}$ is reduced and $\zeta = \{a,b\}$ and then the statement follows immediately from  \cref{prop:psiphiD}.
	
	Let us now consider $a\in A[3]$ and $v\in \P(T_a A)$ of the form $v=[x,\iota(x)]$. We see from Theorem \ref{thm:invtau} that $\tau(a,v)= \alpha_a(x+\iota(x))$ and Remark \ref{rem:tricanonicalcone} shows that $\ell(\tau(a,v)) = \ell(O_a,v_a)$. Thus, if we set $B_{a,v} =  \{\xi \in B_a \,|\, \xi \supseteq (a,v) \}$, we have proved that $\varphi_D(B_{a,v}) \subseteq \ell(O_a,v_a)$, and since $B_a = \bigcup_v B_{a,v}$, we have also proved that $\varphi_D(B_a)\subseteq \Sigma_a$ and that $\varphi_D(\operatorname{Spec} \sO_{A,a}/\mathfrak{m}_a^2) \subseteq \bigcap_{v_a} \ell(O_a,v_a) = O_a$.

	Now recall that $B_a$ is abstractly isomorphic to a cone over a twisted cubic, such that the loci $B_{a,v}$ are lines through the vertex of the cone. In particular, since $B_a$ is a weighted projective space $\P(1,1,3)$, its Picard group is generated 
    by a very ample line bundle $\sO_{B_a}(1)$ with $h^0(B_a,\sO_{B_a}(1)) = 5$. If we look at the restriction of $(\mu_2(2\Theta)-\delta)_{|B_a}$, we see that $\mu_2(2\Theta)_{|B_a} \cong \sO_{B_a}$ since $B_a$ is a fiber of the Hilbert-Chow morphism. Moreover we see from \cite[pp. 7]{Hassett-Tschinkel-lagrangian-planes} that $\sO_{B_a}(-\delta) \cong \sO_{B_a}(1)$, hence $\sO_{B_a}(D) \cong \sO_{B_a}(\mu_2(2\Theta)-\delta) \cong \sO_{B_a}(1)$.
    This proves that ${\varphi_D}_{|B_{a,v}}\colon B_{a,v} \to \ell(O_a,v_a)$ must be induced by the line bundle $\sO_{B_{a,v}}(1)$, hence it is an isomorphism. Then $\varphi_D(B_{a}) = \Sigma_a$, so that ${\varphi_{D}}_{|B_a}$ must be induced by a linear subsystem of $H^0(B_a,\sO_{B_a}(1))$ of dimension $5$. Since $h^0(B_a,\sO_{B_a}(1))=5$, the map is induced by the complete linear system, so that it must be an isomorphism onto the image. 
\end{proof}

\subsection{The image of $\varphi_D$}

We want to use the geometric description of $\varphi_D$ to study its fibers and its image $N=\varphi_D(A)$.

\begin{Prop}\label{prop:mapphiDemb}
	The map $\varphi_D\colon \Kum_2(A) \to \Pd^8$ contracts the divisor $F$ to $A$, $\varphi_D^{-1}(A) = F$, and $\varphi_D$ is an embedding outside $F$. 
\end{Prop}
\begin{proof}
	 We already know from \cref{prop: A-and-N} that $\varphi_D$ contracts the divisor $F$ to $A$. We need to show that $\varphi_D^{-1}(A) \subseteq F$. First, let $\xi = \{a,b,c\}\in \Kum_2(A)$ be reduced such that $\varphi_D(\xi)=e \in A$. Then, \cref{prop:psiphiD} and \cref{prop:meetingsecants} show that $\xi\subseteq \Theta_e$ so that $\xi \in F$.  A similar argument solves the case of  a scheme $\xi = \{(a,v),b\}$, with $a\notin A[3], 2a+b\ \sim 0, v\in \PP(T_a A)$. Finally, assume that $a\in A[3]$ and that $\xi \in B_a$ is a scheme supported only on $a$, such that $\varphi_D(\xi) =e \in A$. Let $(a,v) = (a,[x,\iota(x)])\subseteq \xi$ a subscheme of length two: then  \cref{prop:mapDbriancon} shows that
     $e = \varphi_D(\xi) \in \ell(O_a,v)$, we know from \cref{lem:Pafacts} and the discussion before \cref{prop:mapDbriancon}, that $\ell(O_a,v)\cap A = \ell(O_a,v)\cap A \cap \Pd^4 = \ell(O_a,v)\cap \Theta_a = \alpha_a(x+\iota(x))$. Hence, up to replacing $x$ with $\iota(x)$, we can assume that $e=\alpha_a(x)$, so that $a = \alpha_e(\iota(x))$. Then $\alpha_a(3\iota(x))\subseteq \Theta_e$ and \Cref{prop: A-and-N} shows that $\varphi_D(\alpha_a(3\iota(x))) = e$. Since ${\varphi_D}_{|B_a}$ is injective by \cref{prop:mapDbriancon}, it must be that $\xi = \alpha_a(3\iota(x)) \subseteq \Theta_e$, so that $\xi\in F$.

	We now show that $\varphi_D$ is an embedding on $U = \Kum_2(A)\setminus F$. We use the Coble duality. From what we have showed, we know that $\mathcal{D}\circ {\varphi_D}_{|U}$ is a composition of  morphisms, and \cref{thm: duality} shows that $\mathcal{D}\circ {\varphi_D}_{|U} = {\varphi_{\mu_2(\Theta)}}_{|U}$. Hence, the scheme-theoretic fibers of  ${\varphi_{D}}_{|U}$ are contained in the scheme-theoretic fibers of ${\varphi_{\mu_2(\Theta)}}_{|U}$, which are exactly the fibers of the Hilbert-Chow morphism $\operatorname{HC}\colon \Kum_a(A) \to A^{(3)}$, thanks to \Cref{thm:mapmutheta}. Hence, we need to show that $\varphi_{D}$ is an embedding when restricted to the fibers of Hilbert-Chow morphism on $U$. For the fiber of a reduced scheme, this is obvious, for the fiber over a point $\{a,-2a\}$, with $a\notin A[3]$, this is \Cref{cor:secantlineinN}, and for the Brian\c{c}on surface  over $\{3a\},a\in A[3]$, this is \cref{prop:mapDbriancon}. 
\end{proof}

We can also interpret $N = \varphi_D(\Kum_2(A))$ in terms of the geometry of $A$ in $\Pd^8$: it is the singular locus of the secant variety $\operatorname{Sec}(A)$. We need a couple of preparatory lemmas, and we will use some results from \Cref{sec: appsec}:

\begin{Lem}\label{lem:meetingtangents}
Let $b,c\in A$ be distinct points. If the two projective tangent spaces $\mathbb{T}_bA$ and $\mathbb{T}_cA$ meet, then $\tau(b,c) = \{a,-2a\}$ for a certain $a\in A\setminus A[3]$ or $\tau(b,c) = (a,v)$ for $a\in A[3]$ and $v\in \PP(T_a A)$. In both cases $\ell(b,c) \subseteq N$.
\end{Lem}
\begin{proof}
The two projective tangent spaces meet if and only if there are two tangent lines  $\ell(b,w)$ and $\ell(c,u)$ that meet. By \Cref{prop:4schemes}, this means that $(b,w),(c,u)\subseteq   \Theta_e$ for a certain $e\in A$. More precisely $b\sim x+e-\eta,c\sim y+e-\eta$ for $x,y\in C$ distinct with $2x+2y\sim 2K_C-3e$. Then we can compute explicitly that
\[
\tau(b,c) = 
\begin{cases}
\{\iota(x)-\eta+b,y-\eta+b\} & \text{ if } y\ne \iota(x) \\
(y-\eta+b,[y,\iota(y)]) & \text{ if } y= \iota(x) \\
\end{cases}
=
\begin{cases}
\{e,(x+y)-K_C+e)\} & \text{ if } y\ne \iota(x) \\
(e,[y,\iota(y)]) & \text{ if } y= \iota(x) \\
\end{cases}
\]
Set $a := (x+y)-K_C+e$: in the first case we see that $-2a \sim -2(x+y)+2K_C-2e\sim e$, so that $\tau(b,c)=\{a,-2a\}$ and $\ell(b,c) \subseteq N$ because of \Cref{cor:secantlineinN}. In the second case,  we see that $a\sim e$ and that $3e\sim 0$, so that $\tau(b,c) = (e,[y,\iota(y)])$ and the fact that $\ell(b,c) \subseteq N$ comes from \Cref{prop:mapDbriancon}.
\end{proof}

\begin{Lem}\label{lem:sep4pointsA}
Let $b\in A$ be a point and $v\in \PP(T_bA)$ a tangent direction. Assume that there is a length $4$ subscheme $Z$, supported only at $b$ such that $Z \supseteq (b,v)$ and not separated by $H^0(A,3\Theta)$. Then $\ell(b,v)\subseteq N$.
\end{Lem}
\begin{proof}
\Cref{prop:4schemes} shows that $b\sim x+e-\eta$ for a certain $e\in A$ and $x\in C$ such that $Z=\alpha_e(4x)$ with $4x\sim 2K_C-3e$. Then it must be that $(b,v)=(b,[x,\iota(x)])$ and we see that
\[
\tau(b,v) =  
\begin{cases}
\{x+b-\eta,\iota(x)+b-\eta\} & \text{ if } x\ne \iota(x) \\
(x+b-\eta,[x,x]) & \text{ if } x= \iota(x) \\
\end{cases}
=
\begin{cases}
\{2x+e-K_C,e\} & \text{ if } x\ne \iota(x) \\
(e,[x,x]) & \text{ if } x= \iota(x) \\
\end{cases}
\]
We set $a=2x+e-K_C$. In the first case, we see that $-2a \sim -4x-2e+2K_C \sim e$ so that $\tau(b,v)=\{a,-2a\}$ and $\ell(b,v) \subseteq N$ because of \Cref{cor:secantlineinN}. In the second case,  we see that $a\sim e$ and that $3e\sim 0$, so that $\tau(b,v) = (e,[x,\iota(x)])$ and the fact that $\ell(b,v) \subseteq N$ comes from \Cref{prop:mapDbriancon}.
\end{proof}

\begin{Prop}\label{prop:Nsecant}
The image $N = \varphi_D(\Kum_2(A))$ is the singular locus of $\operatorname{Sec}(A)$. It is a fourfold of degree $36$.
\end{Prop}

\begin{proof}
Consider the incidence correspondence $\mathcal{B}_2 = \{(\xi,p) \in X^{[2]} \times \Pd^8 \,|\, p\in \ell(\xi) \}$. This is a projective bundle over $X^{[2]}$ and the image of $\alpha\colon \mathcal{B}_2 \to \Pd^8$ is precisely the secant variety $\operatorname{Sec}(A)$. We first show that $N\subseteq \operatorname{Sec}(A)$: take a general point $\{a,b,c\}\in \Kum_2(A)$. Then  $a,b,c$ are pairwise distinct, and \Cref{prop:psiphiD} shows that
\[  p = \varphi_D(\{a,b,c\}) = \Pd^4_a \cap \Pd^4_b \cap \Pd^4_c = \ell(\tau(a,b))\cap \ell(\tau(a,c)) \cap \ell(\tau(b,c))\in \Sec(A).\] 
Furthermore, we claim also that 
\begin{equation}\label{eq:fiberageneral}
\alpha^{-1}(p) = \{ (\tau(a,b),p), (\tau(a,c),p), (\tau(b,c),p) \}.  
\end{equation}
Indeed, it is clear from the definition of $p$ that those points are contained in $\alpha^{-1}(p)$. For the converse, observe that, since $\{a,b,c\}\in \Kum_2(A)$ is general, we can assume $\{a,b,c\}\notin F$ so that $p\notin A$ thanks to \Cref{prop:mapphiDemb}. Furthermore, we can also assume that $2a+b \nsim 0, a+2b\nsim 0$. Now, assume that $p\in \ell(\tau(\eta))$ for $\eta\in A^{[2]}$, then $\ell(\tau(\eta))$ and $\ell(\tau(a,b))$ must meet and \Cref{prop:meetingsecants} shows that $\eta$ is one of $\{a,b\},\{a,c\},\{b,c\}$. This proves \eqref{eq:fiberageneral}. In particular $\alpha$ is generically finite, and $\operatorname{Sec}(A)$ has the expected dimension $5$.

Now we show that the map $\alpha\colon \mathcal{B}_2\setminus \alpha^{-1}(N) \to \operatorname{Sec}(A)\setminus N$ is an isomorphism. This will prove that $\operatorname{Sec}(A)$ is smooth outside of $N$. For injectivity, assume that a point $p\in \operatorname{Sec}(A)$ belongs to two secants $\ell(\tau(\eta)),\ell(\tau(\eta'))$ with $\eta,\eta'\in A^{[2]}$ distinct. Then \Cref{prop:meetingsecants}, \Cref{prop:mapphiDemb} and \Cref{prop:mapDbriancon} prove that $p\in N$. We now need to show that the differential of $\alpha$ is injective outside of $\alpha^{-1}(N)$ and we will use the precise Terracini lemma.  
Take $(\xi,p)\in \mathcal{B}_2$ such that $p\in \ell(\xi)$ and $p\notin N$. If $\xi$ consists of two distinct points, then \Cref{lem:preciseterracinitwopoints} and \Cref{lem:meetingtangents} show that the differential of $\alpha$ is injective at $(\xi,p)$. If instead $\xi$ is a double point we use \Cref{lem:preciseterracinidoublepoint} and \Cref{lem:sep4pointsA}.

To show that $\operatorname{Sec}(A)$ is not normal at $N$ we use \Cref{prop:singsec}: indeed, we have proved that the map $\alpha$ is birational and that over a general point of $N$ the fiber consists of three distinct points.

The last statement about the degree of $N$ follows by the Fujiki relation:
\[ \int_{\Kum_2(A)} D^4 = 9\cdot q(D,D)^2 = 9\cdot 4 = 36. \]
\end{proof}

We wrap up this section by proving  Theorem B from the Introduction:

\begin{proof}[Proof of Theorem B]
The facts that the line bundle $\mu_2(\Theta)$ is globally generated and the description of the map $\varphi_{\mu_2(\Theta)}$ comes from Theorem A. The fact that the line bundle $\mu_2(2\Theta)-\delta$ is globally generated is in \Cref{prop:Dglobgen}, and the commutativity of the diagram is in \Cref{thm: duality}. The explicit description of the map $\varphi_{\mu_2(2\Theta)-\delta}$ is in \Cref{prop:psiphiD}, the fact that it is a contraction of the divisor $F$ is in \Cref{prop:mapphiDemb} and the identification of the image with the singular locus of the secant variety $\operatorname{Sec}(A)$ and the statement on the degree  is in \Cref{prop:Nsecant}.

\end{proof}

\section{Weddle surfaces}\label{sec:weddle}

We want to conclude by analyzing the behavior of the surfaces $\widetilde{K}_{1,c}\subseteq \Kum_2(A)$ under the maps $\varphi_{\mu_2(\Theta)}$ and $\varphi_D$. We set $K_c \coloneqq \varphi_{\mu_2(\Theta)}(\widetilde{K}_{1,c})$ and $W_c \coloneqq \varphi_{D}(\widetilde{K}_{1,c})$

Recall that if $X_{c} \subseteq \Kum_2(A)$  is the locus of schemes containing $c$, then
\[ X_c = \widetilde{K}_{1,c} \quad \text{ if } c\notin A[3], \qquad X_c = \widetilde{K}_{1,c}\cup B_c \quad \text{ if } c\in A[3]. \]
 In any case, if $\xi\in X_{c}$ then we see that $\operatorname{HC}(\xi) = [a,b,c]$, for certain $a,b\in A$ with $a+b+c=0$, and then 
\[ \varphi_{\mu_2(\Theta)}(\xi) = \Theta_a+\Theta_b+\Theta_{c} \]
In particular the image $\varphi_{\mu_2(\Theta)}(X_{c})$ is contained in the three-dimensional linear subspace
\[ \P^3_{c} := |3\Theta-\Theta_{c}| \subseteq |3\Theta| = \P^8.  \]

\begin{Prop}
    With the previous notation, it holds that
    \[ K_c = \varphi_{\mu_2(\Theta)}(X_{c}) =\{ \Theta_a+\Theta_{-c-a}+\Theta_{c} \,|\, a\in A \} \subseteq \P^3_{c}.  \]
    This is a quartic surface with 16 nodes, isomorphic to the singular Kummer surface of $A$. 
\end{Prop}
\begin{proof}
The first two equalities come from \Cref{thm:mapmutheta}. We then see that the image is isomorphic to the surface
\[ \{ \Theta_a + \Theta_{-c-a} \,|\, a\in A \} \subseteq |\Theta_{-c}+\Theta| \cong |3\Theta-\Theta_{c}|. \]
As in the case with the usual singular Kummer surface, the map \[ A\longrightarrow |\Theta_{-c}+\Theta|; \quad a \mapsto \Theta_a+\Theta_{-c-a}\]
induces an isomorphism of $A/(\iota\circ t_c)$ with the image, which is a quartic surface with $16$ nodes, corresponding to the points in $[2]^{-1}(-c)$. Furthermore, if we fix a point $d\in [2]^{-1}(-c)$, we see that $t_{d}\circ \iota \circ t_{-d} \cong \iota\circ t_c$, hence the translation $t_d\colon A\to A$ induces an isomorphism of $A/(\iota\circ t_c)$ and $A/\iota$. The latter is the usual singular Kummer surface of $A$.
\end{proof}

If we consider the image  $W_c \coloneqq \varphi_D(\widetilde{K}_{1,c})$ we see that we have a commutative diagram
\[
\begin{tikzcd}
		& \widetilde{K}_{1,c}\ar[ld, "\varphi_{D}"']\ar[rd, "\varphi_{\mu_2(\Theta)}"]\\
		W_c \ar[rr, dashed, "\mathcal{D}"]&& K_c
	\end{tikzcd}
\]
\begin{Rem}
The map $\varphi_{\mu_2(\Theta)}\colon \widetilde{K}_{1,c} \to K_c$ is an isomorphism outside the divisor $\widetilde{K}_{1,c}\cap E \subseteq \widetilde{K}_{1,c}$. Hence, all maps in the above diagram are birational.
\end{Rem}

 We first assume that $c\in A[3]$. The involution $(\iota\circ t_c)$ on $A$ restricts to the hyperelliptic involution on $\Theta_{c}$, hence it acts on the space $\Pd^4_{c}$ where it has two invariant subspaces: the point $O_{c} \in \Pd^4_{c}$ and a three-dimensional subspace, $\Pd^3_{c}$: see \Cref{rem:tricanonicalcone}.

\begin{Prop}
If $c\in A[3]$, the image $W_c \subseteq \Pd^3_{c}$ is a singular quartic surface with 6 nodes, containing 15 lines and 1 rational normal cubic curve, which passes through the nodes. 
\end{Prop}

\begin{proof}
Up to translating by an element of $A[3]$, we can assume that $c=0$. We see from \Cref{prop:psiphiD} and \Cref{prop:mapDbriancon} that $\varphi_{D}(\widetilde{K}_{1,0}) \subseteq \Pd^4_{0}$. Furthermore, it is straightforward to see that $\widetilde{K}_{1,0}$ is fixed pointwise by the involution $\iota$ on $\Kum_2(A)$, and since $O_a \notin \varphi_D(\widetilde{K}_{1,0})$, it must be that $\varphi_D(\widetilde{K}_{1,0})\subseteq \Pd^3_0$. 
Now recall from \Cref{sec: covering}, that we have an isomorphism
\begin{equation}\
v\colon \Kum_1(A) \overset{\sim}{\longrightarrow} \widetilde{K}_{1,0}. 
\end{equation}
Consider now the set $\mathcal{W}=\{w\in C \,|\, w=\iota(w)\}$ of Weierstrass points. For any $w\in \mathcal{W}$ we have the corresponding torsion element $w-\eta \in A[2]$: recall that the trope $F_{w-\eta} \subseteq \Kum_1(A)$ is the divisor consisting of schemes $\zeta\in \Kum_1(A)$ such that $\zeta\subseteq \Theta_{w-\eta}$. It is straightforward to check that $F_{w-\eta} \subseteq v^*(F)$. On the other hand, we also have that
\[ v^*(F) \sim v^*(\mu_2(3\Theta)-2\delta) \sim \mu_1(3\Theta) - 2E_0 -  \sum_{\epsilon\in A[2]} E_{\epsilon} \sim \sum_{w \in \mathcal{W}} F_{w-\eta}  \]
where the first equivalence is from \Cref{lem:FP1undleAndClass}, the second is from \Cref{restr in cohom} and the last one is from \cite[(1.8)]{Keum}. This proves that $v^*(F) = \sum_{w\in \mathcal{W}} F_{w-\eta}$ as divisors. 
We also see from \Cref{restr in cohom} that
\[ 
\quad v^*(D) \sim \mu_1(2\Theta) -  E_{0} - \frac{1}{2}\sum_{\epsilon \in A[2]} E_{\epsilon} 
\]
In particular, $v^*(D)$ is base-point-free with $v^*(D)=4$, and since we know that the composition $\Kum_1(A) \to \widetilde{K}_{1,0} \to W_0 \subseteq \Pd^3_0$ is birational, $W_0$ must be a quartic surface, and since $h^0(\Kum_1(A),v^*D)=4$, the composition is induced by the complete linear system $H^0(\Kum_1(A),v^*D)$. We also know that the map $\Kum_1(A) \to W_0$ is an isomorphism outside of the tropes $F_{w-\eta}$, and it contracts each of these curves to points. Since all tropes are $(-2)$-curves, the $F_{w-\eta}$ are contracted to $6$ nodes on $W_0$, which are the only singularities. Finally, it is easy to see that for any $\epsilon \in A[2]$:
\[ (v^*(D)\cdot E_{\epsilon})=1 \quad\text{ if } \epsilon\ne 0, \qquad (v^*(D)\cdot E_{0}) = 3. \] 
The image of the curve $E_{\epsilon}$ for $\epsilon\ne 0$ is a line: we see from \Cref{cor:secantlineinN}, that it is the line $\ell(\tau(\varepsilon,0))$. If instead we consider $E_0$, then one can see that $E_0$ intersects all the tropes $F_{w-\epsilon}$, so that the image of $E_0$ is the unique rational normal cubic curve passing through all the nodes of $W_0$. This is the twisted cubic $R_0$ of \Cref{rem:tricanonicalcone}.

\end{proof}

\begin{Rem}
 We have seen in the previous proof that the surface $W_0$ is the image of $\Kum_1(A)$ under the map induced by the complete linear system $H^0(\Kum_1(A),v^*D)$. We can also compute that
\[ 
\quad v^*(D) \sim \mu_1(2\Theta) -  E_{0} - \frac{1}{2}\sum_{\epsilon \in A[2]} E_{\epsilon} \sim \frac{3}{2}\tau^*\mu_1(\Theta) - \frac{1}{2}\sum_{\epsilon\in A[2]\setminus (\mathcal{W}-\eta)} F_{\epsilon} 
\]
where the first equivalence is  from the previous proof, and the second one is from putting together \Cref{prop: restrizione a switch}, and the one in the previous proof. At this point \cite[Remark 4.2.5]{bolognesi}, we see that the surface  $W_0 \subseteq \Pd^3_0$ is the so-called classical Weddle surface associated to the Jacobian $A$ and odd theta characteristic $\eta$. We refer to \cite{bolognesi} for background and results on these surfaces.
\end{Rem}

Assume now that $c\notin A[3]$.

\begin{Prop}
If $c\notin A[3]$, the image $W_c\subseteq \Pd^4_c$ is a nondegenerate surface of degree $7$, containing the curve $\Theta_c$.
\begin{enumerate}
    \item If $3c\sim \iota(x)-x$ for a certain $x\in C$, then $\varphi_D\colon \widetilde{K}_{1,c} \to W_c$ is the contraction of a rational curve.
    \item Otherwise $\varphi_D\colon \widetilde{K}_{1,c} \to W_c$ is an isomorphism.
\end{enumerate}
\end{Prop}
\begin{proof}
In this case we know that $\widetilde{K}_{1,c} = X_c$ is the locus of schemes $\xi\in \Kum_2(A)$ containing $c$ in their support. We look at the fibers of $\varphi_D\colon \widetilde{K}_{1,c} \to W_c$: those that possibly contain more than one point are of the form $\varphi_D^{-1}(e) \cap \widetilde{K}_{1,c}$ for $e\in A$, and we know that $\xi \in \varphi_D^{-1}(e)$ if and only if $\xi\subseteq \Theta_e$. Assume then that $c\in \xi \subseteq \Theta_e$: then $c \sim x+e-\eta$ and $e\sim \iota(x)+c-\eta$ for a unique  $x\in C$. Furthermore $\xi = \alpha_e(x+y+z)$ for $x+y+z\sim 3\eta-3e$. This means that $y+z\sim K_C-2e-c$, so that $y+z$ is uniquely determined if $2e+c\nsim 0$ and $y+z=y+\iota(y)$ for any $y\in C$ if $2e+c\sim 0$. We notice that $2e+c\sim 0$ if and only if $2\iota(x)+3c-K_C\sim 0$, which means $3c\sim \iota(x)-x$.  
\begin{enumerate}
\item If $2e+c\sim 0$,then the rational curve  $F_e = \{\alpha_e(x+y+\iota(y)) \in \widetilde{K}_{1,c} \,|\, y\in C\}$ is contracted by $\varphi_D$, and all other fibers consist of a single point.
\item If $2e+c\sim 0$, then $\varphi_D$ restricted to $K_{1,c}$ is an isomorphism.
\end{enumerate}
In both cases, we see that the fiber over any point $e\in \Theta_c$ is nonempty. This proves that $W_c\supseteq \Theta_c$, and since $\Theta_c\subseteq \Pd^4$ is nondegenerate, $W_c$ is nondegenerate as well. To compute the degree of $W_c$, we use the isomorphism $v\colon \operatorname{Bl}_{\{c,-2c\}}K_{1,c} \to \widetilde{K}_{1,c}$ and we compute with \Cref{restr in cohom}
\[ v^*(D) \sim 2\mu_1(\Theta) - \frac{1}{2}\sum_{\epsilon \in [2]^{-1}(-c)} E_{\epsilon} - E_c\]
so that $((v^*D)^2) = 16-8-1 = 7$. 
\end{proof}

\begin{Rem}
Observe that the complete linear system $M\sim 2\mu_1(\Theta)-\frac{1}{2}\sum_{\epsilon \in [2]^{-1}(-c)} E_{\epsilon}$ on $K_{1,c}$ is very ample \cite[Lemma 3.1]{GarbagnatiSartiEvenSet} and $(M^2)=8$, so that it induces an embedding $\varphi_M\colon K_{1,c} \hookrightarrow \P^5$ into a $5$ dimensional space. The variety $W_c \subseteq \Pd^{4}_c$ is obtained via the linear projection $\P^5 \dashrightarrow \Pd^4$ from the point $\{c,-2c\}\in K_{1,c}$. 
\end{Rem}

Finally, we show that the intersection of $N=\varphi_D(\Kum_2(A))$ with the spaces $\Pd_4^c$ is given precisely by the $W_c = \varphi_D(\widetilde{K}_{1,c})$:

\begin{Prop}
    If $c\in A$ it holds that 
    \[ 
    N\cap \Pd^4_c = \begin{cases}
    W_c\cup \Sigma_c, & \text{ if } c\in A[3],\\
    W_c, & \text{ if } c\notin A[3].
    \end{cases} 
    \]
\end{Prop}
\begin{proof}
By construction $W_c \subseteq N\cap \Pd^4_c$ for all $c$ and $\Sigma_c\subseteq N$ if $c\in A[3]$. Conversely, let $\xi\in \Kum_2(A)$ such that $\varphi_D(\xi)\in \Pd^4_c$. We need to show that there is $\xi'\in \Kum_2(A)$ such that $c\in \xi'$ and $\varphi_D(\xi)=\varphi_D(\xi')$. Assume first that  $\xi$ contains at least two distinct points $a,b\in \xi$. We want to show that $c\in \xi$ as well: if $c$ is one of $a,b$, we are done. Otherwise, by \Cref{eq:mapDphieta} we know that $\varphi_D \in \ell(\tau(a,b))=\Pd^4_a \cap \Pd^4_b$ then, by assumption, $\varphi_D(\xi) \in \Pd^4_a \cap \Pd^4_b\cap \Pd^4_c$. At this point  \Cref{cor:intersectionthreespaces} proves that $a+b+c\sim 0$. Since $\xi\in \Kum_2(A)$, this means that $\xi = \{a,b,c\}$.

Assume now that $\xi$ is supported at a unique point $a\in A$: this means that $\xi\in B_a$ for a certain $a\in A[3]$, so that $\varphi_D(\xi)\in \Sigma_a$. If $a=c$ we are done, otherwise we know that $\varphi_D(\xi) \in \Sigma_a \cap \Pd^4_c = \Sigma_a\cap \Pd^4_a\cap \Pd^4_c = \Sigma_a \cap \ell(\tau(a,c))$. We observe that $\Theta_a\subseteq \Sigma_a$ so that $\Sigma_a\cap \ell(\tau(a,c)) \supseteq \tau(a,c)=\Theta_a\cap \Theta_c$. Furthermore, if the intersection $\Sigma_a\cap \ell(\tau(a,c))$ consists of more than three points, then it must be that $\ell(\tau(a,c)) \subseteq \Sigma_a$, since $\Sigma_a$ is cut out by quadrics. We discuss these two cases: if $\Sigma_a\cap \ell(\tau(a,c)) = \tau(a,c)$ then $\varphi_D(\xi)=e\in \tau(a,c)$. Then we know from \Cref{lem:symmtau} that $c\in \Theta_e$, and it is easy to see that we can find $c \in \xi' \subseteq \Theta_e$ with $\xi'\in \Kum_2(A)$: then $\varphi_D(\xi')=e$. If instead $\ell(\tau(a,c))\subseteq \Sigma_a$, then $\ell(\tau(a,c))$ must go through the vertex $O_a$ of $\Sigma_a$. Using \Cref{eq:thetalinecone} we see that
\[ \tau(a,c) \subseteq \ell(\tau(a,c))\cap \Theta_a = \alpha_a(x+\iota(x)) \] 
for a certain $x\in C$. Then $\tau(a,c)=\alpha_a(x+\iota(x))$, and $\{a,c\} = \tau(\alpha_a(x+\iota(x)))$. But this is impossible because $\tau(\alpha_a(x+\iota(x)))$ is nonreduced, by definition of $\tau$ ( see \cref{thm:invtau}).
\end{proof}

\section{Further directions}\label{sec:openquestions}


A first natural question would be to extend the results in Theorem A or Theorem B to a complete family of varieties of Kummer type. For example, it is natural to ask whether the Coble cubic or the Coble sextic deform along a complete four-dimensional family. However, even staying in the setting of Kummer arising from actual abelian varieties there are various interesting questions to consider:
\vspace{10pt}

For example, in this work we have always worked on the Jacobian of a smooth genus two curves. It would be interesting to study the analogue of our statements for a product $A = E_1\times E_2$ of two elliptic curves. In this case, there is still a unique Coble cubic hypersurface $\mathscr{C}_3$ singular along $A\subseteq \Pd^8$, but the singular locus of $\mathscr{C}_3$ is larger than $A$: it turns out to be isomorphic to a Segre embedding of $\P^2\times \P^2$ inside $\Pd^8$ \cite{Barth}.
\vspace{10pt}

We also expect the map $u\colon \Sigma_{n}(A) \to |(n+1)A|$ of Theorem A,  to be an embedding for any $n$, and not only $n=1,2$. Since the map is always injective, this amounts to studying its differential.  Furthermore, the same construction works for generalized Kummer varieties of an arbitrary abelian variety $A$, not necessarily a Jacobian or a surface. Define
\[ \Sigma_n(A) = \{[a_0,\dots,a_n] \in A^{(n+1)} \,|\, a_i\in A, a_0+\dots+a_n \sim 0\}.\]
If $L$ is a nontrivial line bundle on an abelian variety $A$, we also have the line bundle $\iota^*(L)$, and the corresponding bundle $\iota^*(L)^{(n+1)}$ on $A^{(n+1)}$ and $\Sigma_n(A)$. If $D\in |L|$ is an effective divisor, then the theorem of the square gives a map 
\[ u\colon \Sigma_n(A) \longrightarrow |(n+1)L|, \quad [a_0,\dots,a_n] \mapsto t_{a_0}(D)+\dots+t_{a_n}(D) . \]
Observe that $a\in t_b(D)$ if and only if $b\in t_a(\iota(D))$.Then one can see as in the proof of \Cref{thm:mapmutheta}, that $u^*\sO(1) \sim \iota^*(L)^{(n+1)}$. It would be interesting to determine precisely when the map is an embedding, and what is the relationship with the map induced by the complete linear system $H^0(\Kum_n(A),\mu_n(\Theta)$.
\vspace{10pt}

Finally, it would also be interesting to recast the results of Theorem A and Theorem B through the classical language of theta functions. For example, in the case of Theorem A for $n=1$, this is given by Riemann's addition formula. This could also give a way to produce explicit equations for the images $\varphi_{\mu_2(\Theta)}(\Kum_2(A)) \subseteq \P^8$ and $\varphi_{\mu_2(2\Theta)-\delta}(\Kum_2(A)) \subseteq \Pd^8$. Another way to obtain explicit equation might also be  through the point of view of \cite{BMT2021}.

\appendix

\section{Singularities of secant varieties}\label{sec: appsec}

\subsection{Precise Terracini}\label{sec:appendix precise terracini}
Here we make an excursus into singularities of secant varieties, following \cite{CPLS25}. We keep working over $\CC$, but it can be replaced throughout by an algebraically closed field of characteristic zero.
Let $X \subseteq \PP(V^{\vee})$ be an irreducible variety of dimension $n$, embedded by a linear system $V\subseteq H^0(X,\sO_X(1))$. If $\xi \in X^{[k]}$ is a finite subscheme of $X$ of length $k$, we say that \emph{$V$ separates $\xi$} if there is an exact sequence of vector spaces 
\begin{equation}\label{eq:evalxi}
0 \longrightarrow V\cap \sI_{\xi} \longrightarrow V \longrightarrow \sO_{\xi} \longrightarrow 0 
\end{equation}
where $V\cap \sI_{\xi} := V \cap H^0(X,\sI_{\xi}(1))$ and the last nontrivial map is the evaluation $V\to H^0(\sO_{\xi}(1)) \cong \sO_{\xi}$.
Equivalently, this means that the scheme $\xi$ spans a linear space 
\[ \ell( \xi ) = \PP(\sO_{\xi}^{\vee})\subseteq \PP(V^{\vee}),\qquad \sO_{\xi}^{\vee} := \Hom_{\CC}(\sO_{\xi},\CC)\] 
of expected dimension $\dim \ell( \xi ) = k-1$.
Assume now that $V$ is $(k-1)$-very ample, meaning that it separates all $\xi \in X^{[k]}$. One can then construct the incidence correspondence
\[ \mathcal{B}_k = \{(\xi,p) \in X^{[k]}\times \PP(V^{\vee}) \,|\, p\in \PP(\sO_{\xi}^{\vee})\}, \quad \pi\colon \mathcal{B}_k \to X^{[k]}, \quad \alpha\colon \mathcal{B}_k \to \PP(V^{\vee}). \]
When $V=H^0(X,\sO_X(1)))$, then $\mathcal{B}_k = \mathcal{B}_{k}(\sO_X(1))$, in the notation of \cite{CPLS25}.

The classical Terracini's lemma describes the differential of $\alpha$ when the scheme $\xi$ consists of mutually distinct points. The precise Terracini Lemma of \cite[Section 3.3]{CPLS25} generalizes it to other schemes. In order to state it, consider first the evaluation map
\begin{equation}\label{eq:evprime}
\operatorname{ev}'\colon V\cap \mathscr{I}_{\xi} \longrightarrow \sI_{\xi}/\sI_{\xi}^2 \otimes \sO_{X}(1) \cong \sI_{\xi}/\sI_{\xi}^2, \quad \overline{V \cap \sI_{\xi}} \coloneqq \operatorname{ev}'(V\cap \sI_{\xi}) 
\end{equation}
Now let $(\xi,p)\in \mathcal{B}_k$, so that $p$ is represented by a linear map $\lambda\colon \sO_{\xi} \to \CC$. There is a natural map of $\CC$-vector spaces
\begin{equation}\label{eq:gammaprime} 
\gamma'\colon \sI_{\xi}/\sI_{\xi}^2 \longrightarrow  \Hom_{\CC}(\Hom_{\sO_{\xi}}(\sI_{\xi}/\sI_{\xi}^2,\sO_{\xi}),\CC), \qquad \overline{f} \mapsto \gamma'(\overline{f})\colon \phi \mapsto \lambda(\phi(\overline{f}))
\end{equation}


The precise Terracini lemma of \cite[Section 3.3]{CPLS25} can be stated as follows:

\begin{Thm}[Precise Terracini - I, \cite{CPLS25}]\label{thm:terracini1}
    The differential of $\alpha\colon \mathcal{B}_k \to \PP(V^{\vee})$ is injective at $(\xi,p) \in \mathcal{B}_k$ if and only if the composition
    \[ \gamma'\circ \operatorname{ev}' \colon V\cap \sI_{\xi} \longrightarrow \Hom_{\CC}(\Hom_{\sO_{\xi}}(\sI_{\xi}/\sI_{\xi}^2,\sO_{\xi}),\CC) \]
    is surjective. For this to be true, the following conditions are sufficient:
    \begin{enumerate}
      \item The subscheme $\xi$ is Gorenstein and the point $p$ is not contained in the linear span of any proper subscheme of $\xi$.
    \item The linear system $V$ separates the scheme $\xi^2 := \operatorname{Spec} \sO_{X}/\sI_{\xi}^2$.
    \end{enumerate}
    Furthermore, if $\xi$ is a locally complete intersection scheme, these conditions are necessary as well.
\end{Thm}
\begin{proof}  
We spell out some details from \cite{CPLS25}. The fact that the differential of $\alpha$ is injective if and only if $\gamma'\circ \operatorname{ev}'$ is surjective is contained in \cite[Lemma 3.8 and proof of Theorem 3.9]{CPLS25}. If condition (1) holds, \cite[Lemma 3.6]{CPLS25} shows that the map $\gamma'$ is surjective. On the other hand, if $V$ separates $\xi$, it is straightforward to see that $V$ separates $\xi^2$ if and only if the map $\operatorname{ev}'$ is surjective. Hence, if conditions (1) and (2) hold, then the composition $\gamma'\circ \operatorname{ev}'$ is surjective.

Assume now that $\xi$ is a locally complete intersection scheme: then $\sI_{\xi}/\sI_{\xi}^2 \cong \sO_{\xi}^{\oplus n}$ as $\sO_{\xi}$-modules, so that the map $\gamma'$ becomes  the map
\[ \gamma'\colon (\sO_{\xi})^{\oplus n} \longrightarrow (\sO_{\xi}^{\vee})^{\oplus n} \quad \text{ induced from } \quad \sO_{\xi} \to \sO_{\xi}^{\vee}, \quad a\mapsto \lambda(a\cdot -)\colon b\mapsto \lambda(ab) \]
Hence we see that $\gamma'$ is surjective if and only if  $\gamma'$ is an isomorphism, and this happens if and only if the second map above is an isomorphism. By \cite[Lemma 3.2]{CPLS25}, this means precisely that condition (1) holds.  Assume also that $\gamma'\circ \operatorname{ev}'$ is surjective: then $\gamma'$ must be surjective, and the previous discussion shows that condition (1) holds and that $\gamma'$ must be an isomorphism. Then $\operatorname{ev}'$ must be surjective, and this means that condition (2) holds as well.
\end{proof}


What we want to show here is an equivalent statement for condition (2) in \Cref{thm:terracini1}.

\begin{Lem}\label{lem:terracinigor}
    Assume $\sO_{\xi}$ is Gorenstein of length $k$ and that $V$ separates $\xi$. Then $V$ separates $\xi^2 = \operatorname{Spec} \sO_{X}/\sI^2_{\xi}$ if and only if $V$ separates any subscheme $\xi \subseteq \zeta \subseteq \xi^2$  of length  $\dim_{\CC} \sO_{\zeta} \leq 2k$.
\end{Lem}
\begin{proof}
 If $V$ separates $\xi^2$ then it separates all its subschemes as well. For the converse,
 we reduce to commutative algebra: let $A$ be the ring of functions that are regular at the points of $\operatorname{Supp}(\xi)$ and let $I\subseteq A$ be the ideal corresponding to $\sI_{\xi}$, so that $\sO_X/\sI_{\xi} \cong A/I$ and $\sO_X/\sI_{\xi}^2 \cong A/I^2$.  We can also replace $V$ with its image in $A/I^2$ under the evaluation map. Our assumption is that, the map $V\to A/I$ is surjective and that $V\to A/J$ is surjective for all ideals $I^2\subseteq J \subseteq I$ with $\dim_{\CC}(A/J)\leq 2k$. We need to show that $V=A/I^2$.

 Since the map $V\to A/I$ is surjective by assumption, we see that $V=A/I^2$ if and only if  $V\cap (I/I^2) = (I/I^2)$. Let $J$ be an ideal as before. Since $V\to A/J$ is surjective, it must be that $(V\cap I/I^2)\to I/J = (I/I^2)/(J/I^2)$ is surjective as well. Notice that $\dim_{\CC} I/J = \dim_{\CC} (A/J) - \dim_{\CC} (A/I)\leq 2k-k\leq k$, and also that any quotient of $I/I^2$ as $A/I$-modules of lenght at most $k$ has this form. Hence, the proof is concluded by the next Lemma \ref{lemma:MGorenstein}.
\end{proof}

\begin{Lem}\label{lemma:MGorenstein}
 Let $R$ be an artinian Gorenstein algebra of length $\dim_{\CC}(R)=k$, let $M$ be a finitely generated $R$-module and let $W\subseteq M$ be a $\CC$-linear subspace such that $W \to M/N$ is surjective for any quotient of $R$-modules of length $\dim_{\CC}(M/N)\leq k$. Then $W=M$.
\end{Lem}
\begin{proof}
    We will prove that the map $\Hom_{\CC}(M,\CC) \to \Hom_{\CC}(W,\CC), \mu\mapsto \mu_{|W}$ is injective. Since $R$ is Gorenstein, there exists $\lambda\in \Hom_{\CC}(R,\CC)$  such that the map
    \[ \Hom_{R}(M,R) \longrightarrow \Hom_{\CC}(M,\CC); \quad \varphi \mapsto \lambda\circ\varphi \]
    is an isomorphism 
    Assume now that $\varphi\in \Hom_{R}(M,R)$ is such that $(\lambda\circ \varphi)_{|W} = 0$. Since $\varphi(M)\subseteq R$ it must be that $\dim_{\CC} \varphi(M)\leq k$. By hypothesis, we see that $\varphi(M)=\varphi(W)$, so that if $\lambda\circ \varphi_{|W} = 0$, then $(\lambda\circ \varphi) = 0$. This is what we wanted to prove. 
\end{proof}

Thus we have another version of the precise Terracini lemma as

\begin{Cor}[Precise Terracini - II]\label{cor:preciseterracini2}
    The differential of $\alpha\colon \mathcal{B}_k \to \PP(V^{\vee})$ is injective at a point $(\xi,p) \in B$ if the following two sufficient conditions hold:
    \begin{enumerate}
      \item The subscheme $\xi$ is Gorenstein and the point $p$ is not contained in the linear span of any proper subscheme of $\xi$.
      \item The linear system $V$ separates  any subscheme $\xi\subseteq \zeta \subseteq \xi^2$ of length $\ell(\sO_\zeta)\leq 2k$.
    \end{enumerate}
    Furthermore, if $\xi$ is a locally complete intersection scheme, these conditions are necessary as well.
\end{Cor}

\begin{Rem}
In particular, this holds whenever the linear system $V$ is $(2k-1)$-very ample.
\end{Rem}


\subsection{Singularities of secant varieties}
Now we want to apply this to study singularities of the secant variety. Thus, let $X\subseteq \PP(V^{\vee})$ be a smooth and irreducible nondegenerate projective variety, embedded by a very ample linear system $V \subseteq H^0(X,\sO_X(1))$. The \emph{secant variety} of $X$ is the image
\[ \alpha\colon \mathcal{B}_2 \twoheadrightarrow \operatorname{Sec}(X) \subseteq \PP(V^{\vee})\]
In other words, it is the union of all secant lines of $X$, where we consider tangent lines as special cases of secant lines.

We want to  look at singularities of the secant variety at \emph{identifiable} points: a point $p\in \Sec(X)$ is identifiable if the fiber $\alpha^{-1}(p)\subseteq \mathcal{B}_2$ consists of an unique point. In other words, this means that $p$ is contained in a unique secant or tangent line, and that this line arises from a single $\xi\in X^{[2]}$.

\begin{Prop}\label{prop:singsec}
   Let $X\subseteq \PP(V^{\vee})$ be a smooth and irreducible nondegenerate projective variety as above and let $p\in \operatorname{Sec}(X)\setminus X$ be a point. 
   \begin{itemize}
     \item[(a)] If $p$ is identifiable with $\alpha^{-1}(p)=(\xi,p)$ and if the differential of $\alpha$ is injective at $(\xi,p)$, then $\operatorname{Sec}(X)$ is smooth at $p$.
   \end{itemize}
   Assume now that the general point of $\operatorname{Sec}(X)$ is identifiable.
   \begin{itemize} 
     \item[(b)] If $p$ is not identifiable but the fiber $\alpha^{-1}(p)$ is finite, or if $p$ is identifiable with $\alpha^{-1}(p)=(\xi,p)$ and the differential of $\alpha$ at $(\xi,p)$ is not injective, then $\operatorname{Sec}(X)$ is not normal at $p$.
   \end{itemize}
\end{Prop}

\begin{proof}
     \begin{itemize}
\item[(a)] Since the Hilbert scheme $X^{[2]}$ is smooth and since $\pi\colon \mathcal{B}_2 \to X^{[2]}$ is a projective bundle, then $\mathcal{B}_2$ is smooth as well. By assumption, the scheme-theoretic fiber of $\alpha$ at $p$ is a single point, and the same must be true in a neighborhood of $p$. Thus, $\alpha$ is an isomorphism in a neighborhood of $(\xi,p)$ and we conclude.
\item[(b)] Consider the map $\alpha\colon \mathcal{B}_2 \to \operatorname{Sec}(X)$. Since a general point of $\operatorname{Sec}(X)$ is identifiable, this map is birational. Assume that $\operatorname{Sec}(X)$ is normal at $p$. Then we can find an open neighborhood $U\subseteq \operatorname{Sec}(X)$ of $p$ that is normal and such that $\alpha\colon \alpha^{-1}(U) \to U$ is finite and birational. Since $U$ is normal, this map must be an isomorphism.
     \end{itemize}
 \end{proof}

To check that the differential of $\alpha$ is injective, we use  precise Terracini. For elements of $X^{[2]}$, we can make it a bit more concrete:

\begin{Lem}\label{lem:preciseterracinitwopoints}
    Let $\xi=\{x,y\} \in X^{[2]}$ consist of two distinct points and let $p\in \langle x,y \rangle$. The differential of $\alpha\colon \mathcal{B}_2\to \PP(V^{\vee})$ is injective at $(\xi,p)$ if and only if $p\notin \{x,y\}$ and the two projective tangent spaces $\mathbb{T}_x X$ and $\mathbb{T}_y X$ are disjoint.
\end{Lem}
\begin{proof}
We use \Cref{thm:terracini1}, with the observation that a reduced scheme $\xi=\{x,y\}$ is a locally complete intersection. The fact that $p\in \{x,y\}$ means precisely that $p$ is not contained in the linear span of subschemes of $\xi$. Then the linear system $V$ separates $\xi^2 = \operatorname{Spec} \sO_{X}/\mathfrak{m}_x^2 \sqcup \operatorname{Spec} \sO_{Y}/\mathfrak{m}_y^2$ if and only if the span of the two projective tangent spaces $\mathbb{T}_x X$ and $\mathbb{T}_y X$ has the expected dimension. This means that the two tangent spaces are disjoint.
\end{proof}

For the next statement, recall that a finite subscheme $\zeta\subseteq X$ supported at a point $x\in X$ is \emph{curvilinear} if $\dim T_x\zeta = 1$ and it is \emph{planar} if $\dim T_x\zeta=2$. This means that $\zeta$ is contained in a smooth curve (resp. surface) inside $X$.

\begin{Lem}\label{lem:preciseterracinidoublepoint}
    Let $\xi=(x,v) \in X^{[2]}$ consist of a point and a tangent direction $v\in \PP(T_x X)$. The differential of $\alpha\colon \mathcal{B}_2\to \PP(V^{\vee})$ is injective at $(\xi,p)$ if and only if $p\ne x$ and if $V$ separates all curvilinear and planar subschemes $\xi\subseteq \zeta\subseteq \xi^2$ with $\dim_{\CC} \sO_{\zeta} \leq 4$.
\end{Lem}
\begin{proof}
We use \Cref{cor:preciseterracini2}, with the observation that the scheme $\xi=(x,v)$ is also a locally complete intersection. The fact that $p\ne x$ means precisely that $p$ is not contained in the linear span of a subscheme of $\xi$. We now have to show that under our hypotheses $V$ separates all subschemes $\xi\subseteq \zeta \subseteq \xi^2$ with $\dim \sO_{\eta} \leq 4$ and $\dim T_x\zeta\geq 3$. We will show that with these hypotheses $\eta$ is a subscheme of $\operatorname{Spec} \sO_{X}/\mathfrak{m}_x^2$, which is itself separated by $V$ since it is very ample. Choose local coordinates $(z_1,\dots,z_n)$ on $X$ around $x$ such that $\mathfrak{m}_x = (z_1,\dots,z_n)$ and $\sI_{\xi} = (z_1^2,z_2,\dots,z_n)$. In particular $\sI_{\xi} \supseteq \mathfrak{m}_x^2$ and by hypothesis $\sI_{\xi} \supseteq \sI_{\zeta}$. Then
\begin{align*} 
3\leq \dim_{\CC} T_x\zeta &= \dim_{\CC} \mathfrak{m}_x/(\mathfrak{m}_x^2+\sI_{\zeta}) = \dim_{\CC} \mathfrak{m}_x/\sI_{\xi} + \dim_{\CC} \sI_{x}/(\mathfrak{m}_x^2+\sI_{\zeta}) \\
&= \dim_{\CC} T_x\xi + \dim_{\CC} \sI_x/\sI_{\zeta} - \dim_{\CC} (\sI_{\zeta}+\mathfrak{m}_x^2)/\mathfrak{m}_x^2\\
&= 1 + \dim_{\CC} \sO_{\zeta}- \dim_{\CC} \sO_{\xi} - \dim_{\CC} (\sI_{\zeta}+\mathfrak{m}_x^2)/\mathfrak{m}_x^2\\
&=\dim_{\CC}\sO_{\zeta}-1-\dim_{\CC} (\sI_{\zeta}+\mathfrak{m}_x^2)/\mathfrak{m}_x^2
\end{align*}
Hence $\dim_{\CC} (\sI_{\zeta}+\mathfrak{m}_x^2)/\mathfrak{m}_x^2 \leq \dim_{\CC}\sO_{\zeta}-4 \leq 0$ and this means that $(\sI_{\zeta}+\mathfrak{m}_x^2) = \mathfrak{m}_x^2$, i.e. $\mathfrak{m}_x^2\supseteq \sI_{\zeta}$. This is what we wanted to show.
\end{proof}

\begin{Rem}
The last lemma suggests that what is relevant are the smooth surfaces in $X$ through $p$. However it is not clear to us how to interpret this geometrically. In the case where $X$ is an actual smooth surface, we can make everything more explicit with local computations: assume that $(z_1,z_2)$ are local coordinates around $x$ such that $\sI_{\xi} = (z_1^2,z_2)$. Then one can check explicitly that a very ample linear system $V$ separates $\xi^2$ if and only if it separates the schemes given by the ideals
\[ (z_1^4,z_2 - \alpha z_1^2 - \beta z_1^3) \text{ for all } \alpha,\beta\in \CC \quad  \text{ and  } \quad (z_1^2,z_2^2). \]
So it is enough to check the curvilinear schemes of length four and one single planar scheme. This particular planar scheme was already singled out by \cite[Remark 3.6]{BFS} in their study of $k$-very ampleness.
\end{Rem}

\bibliography{references.bib}{}
\bibliographystyle{alpha}
\end{document}